\documentclass[11pt]{amsart}

\usepackage{graphicx}
\usepackage[USenglish,english]{babel}
\usepackage[T1]{fontenc}
\usepackage[latin1]{inputenc}
\usepackage{amsfonts}
\usepackage{amssymb}
\usepackage{amsthm}
\usepackage{amsmath}
\usepackage{tikz}
\usepackage{float}
\usepackage{enumerate}
\usepackage{accents}
\usepackage{mathtools}
\usepackage{mathrsfs}
\usepackage{comment}
\usepackage{afterpage}
\usepackage{bbm}
\usepackage{dsfont}
\usepackage{stmaryrd}
\usepackage{hyperref}
\usepackage{mleftright}
\usepackage{subfig}

\theoremstyle{plain}
\newtheorem{thm}{Theorem}[section]
\newtheorem{lem}[thm]{Lemma}
\newtheorem{prop}[thm]{Proposition}
\newtheorem{cor}[thm]{Corollary}
\newtheorem*{thm*}{Theorem}
\newtheorem*{prop*}{Proposition}
\newtheorem*{cor*}{Corollary}
\newtheorem{thmintro}{Theorem}

\newtheorem{corintro}[thmintro]{Corollary}
\newtheorem{propintro}[thmintro]{Proposition}

\theoremstyle{definition}
\newtheorem{defn}[thm]{Definition}

\newtheorem{rmk}[thm]{Remark}
\newtheorem*{rmk*}{Remarks}
\newtheorem*{quest*}{Question}

\newtheorem*{ass}{Standing Assumptions}

\DeclareMathOperator{\op}{op} 
\DeclareMathOperator{\llk}{lk} 
\DeclareMathOperator{\sing}{sing} 

\renewcommand{\o}{\circ}
\newcommand{\wt}{\widetilde}
\newcommand{\R}{\mathbb{R}}
\newcommand{\Z}{\mathbb{Z}}
\newcommand{\N}{\mathbb{N}}

\newcommand{\s}{\sigma}
\newcommand{\ra}{\rightarrow}
\newcommand{\Ra}{\Rightarrow}

\newcommand{\cu}{\subseteq}
\newcommand{\g}{\gamma}
\newcommand{\G}{\Gamma}

\newcommand{\mbb}{\mathbb}
\newcommand{\mc}{\mathcal}
\newcommand{\mf}{\mathfrak}
\newcommand{\x}{\times}
\newcommand{\eps}{\epsilon}
\newcommand{\Om}{\Omega}
\newcommand{\om}{\omega}
\newcommand{\Aut}{\mathrm{Aut}}
\newcommand{\Out}{\mathrm{Out}}

\newcommand{\acts}{\curvearrowright}
\newcommand{\wh}{\widehat}
\newcommand{\mscr}{\mathscr}

\newcommand{\crt}{\mathrm{crt}}
\newcommand{\Cr}{\mathrm{cr}}

\newcommand{\lk}{\mathrm{lk}}
\newcommand{\hull}{\mathrm{Hull}}
\renewcommand{\ll}{\llbracket}
\newcommand{\rr}{\rrbracket}
\newcommand{\CAT}{{\rm CAT(0)}}
\newcommand{\CA}{{\rm CAT(-1)}}
\newcommand{\B}{\mbb{B}}

\title[Cross ratios on cube complexes and length-spectrum rigidity]{Cross ratios on $\CAT$ cube complexes and marked length-spectrum rigidity}
\author[J.\ Beyrer]{Jonas Beyrer}\address{IHES, Bures-sur-Yvette, France}\email{beyrer@ihes.fr}
\author[E.\ Fioravanti]{Elia Fioravanti}\address{Max--Planck--Institut f\"ur Mathematik, Bonn, Germany}\email{fioravanti@mpim-bonn.mpg.de} 

\begin{document}

\begin{abstract}
We show that group actions on irreducible $\CAT$ cube complexes with no free faces are uniquely determined by their $\ell^1$ length function. Actions are allowed to be non-proper and non-cocompact, as long as they are $\ell^1$--minimal and have no finite orbit in the visual boundary. This is, to our knowledge, the first length-spectrum rigidity result in a setting of \emph{non-positive} curvature (with the exception of some particular cases in dimension 2 and symmetric spaces).

As our main tool, we develop a notion of cross ratio on Roller boundaries of $\CAT$ cube complexes. Inspired by results in \emph{negative} curvature, we give a general framework reducing length-spectrum rigidity questions to the problem of extending cross-ratio preserving maps between (subsets of) Roller boundaries. The core of our work is then to show that, when there are no free faces, these cross-ratio preserving maps always extend to cubical isomorphisms. All our results equally apply to cube complexes with variable edge lengths.

As a special case of our work, we construct a compactification of the Char\-ney--Stambaugh--Vogtmann Outer Space for the group of untwisted outer automorphisms of an (irreducible) right-angled Artin group. This generalises the length function compactification of the classical Culler--Vogtmann Outer Space.
\end{abstract}

\maketitle

\vspace{-0.85cm}

\tableofcontents

\section{Introduction.}

Cross ratios arise naturally in the study of boundaries at infinity of \emph{negatively} curved spaces. In this paper and the next \cite{BF2}, we develop a similar tool for a large class of \emph{non-positively} curved spaces: $\CAT$ cube complexes.

When $X$ is the universal cover of a closed Riemannian manifold of negative sectional curvature, it is well-known that the cross ratio on $\partial_{\infty}X$ bears a close connection\footnote{See Theorem~5.1 in \cite{Biswas1}.} to the long-standing ``\emph{marked length-spectrum rigidity conjecture}'' (Problem~3.1 in \cite{Burns-Katok}). Although similar notions of cross ratio are available on boundaries of arbitrary $\CAT$ spaces, said connection normally breaks down in this more general context. 
Quite surprisingly, this is not the case for $\CAT$ cube complexes, where it is still possible to relate our cross ratio to length spectra.

This enables us to prove a version of the \emph{marked length-spectrum rigidity conjecture} for $\CAT$ cube complexes (Theorem~\ref{NFF MLSR} below). To the best of our knowledge, this is the first result of this type for spaces of high dimension and possibly containing flats\footnote{Besides the ``obvious'' cases where marked length-spectrum rigidity already follows from stronger forms of rigidity (e.g.\ symmetric spaces and Margulis' superrigidity).}. Up to now, work on the conjecture had been limited to $2$--dimensional spaces and Finsler manifolds --- the latter almost exclusively Riemannian and negatively-curved. See for instance \cite{Culler-Morgan,Otal-Ann,Croke,Croke-Fathi-Feldman,Eberlein,Hamenstaedt-GAFA,Kim-proj,Dalbo-Kim}.

Given a group $\G$, our results set the basis for the study of the \emph{space of all actions} of $\G$ on $\CAT$ cube complexes. As a first step in this direction, we use length functions to compactify large spaces of cubulations of $\G$ (cf.\ Proposition~\ref{cptf main} below). This is in many ways an analogue of Thurston's compactification of Teichm\"uller space.

\medskip
Cube complexes were originally introduced by Gromov \cite{Gromov-HypGps} and have quickly become somewhat ubiquitous within geometric group theory. They are better understood than arbitrary spaces of non-positive curvature, due to the distinctive features of their geometry. For instance, unlike arbitrary Hadamard manifolds, $\CAT$ cube complexes always contain many totally geodesic, codimension-one hypersurfaces \cite{Sag95}.

Still, cubical geometry is by no means trivial, as it displays a vast array of phenomena and it encodes the geometry of a large number of groups. Examples of \emph{cubulated} groups --- i.e.\ those groups that act properly and cocompactly on a $\CAT$ cube complex --- include right-angled Artin groups, hyperbolic or right-angled Coxeter groups \cite{Niblo-Reeves}, hyperbolic free-by-cyclic groups \cite{Hagen-Wise1,Hagen-Wise2}, hyperbolic $3$--manifold groups \cite{Bergeron-Wise}, most non-geometric $3$--manifold groups \cite{Prz-Wise1,Hagen-Prz,Prz-Wise2}, many arithmetic lattices in $SO(n,1)$ \cite{Bergeron-Haglund-Wise,Haglund-Wise-Ann}, finitely presented small cancellation groups \cite{Wise-GAFA}, random groups at low density \cite{Ollivier-Wise}, etc.

In our setting, it is more appropriate to endow $\CAT$ cube complexes with their \emph{$\ell^1$ metric}, also known as \emph{combinatorial metric}. This is bi-Lipschitz equivalent to the $\CAT$ metric, but it has the advantage of satisfying all the good properties of median geometry \cite{Roller,CDH,Fio1}. This will be crucial to many of our results and we will go into more detail on what fails for the $\CAT$ metric later in the introduction.

When a group $\G$ acts on a $\CAT$ cube complex $X$ by cubical automorphisms, we can associate to each element $g\in \G$ its \emph{translation length}:
\[\ell_X(g)=\inf_{x\in X} d(x,gx).\]
As another advantage of employing the combinatorial metric on $X$, the value $\ell_X(g)$ is always an integer. In fact, it suffices to compute the infimum over vertices $x$ of the barycentric subdivision $X'$ (cf.\ Proposition~\ref{Haglund FFT} below) and these always have an integer displacement with respect to the $\ell^1$ metric.

The map $\ell_X\colon\G\ra\N$ is normally known as \emph{length function}, or \emph{marked length spectrum} by analogy with the Riemannian setting. It is natural to wonder whether $\ell_X$ fully encodes the geometry of the $\G$--action. For cube complexes that, like Riemannian manifolds, enjoy the \emph{geodesic extension property}, we show that this is indeed the case:

\begin{thmintro}\label{NFF MLSR}
Let $X$ and $Y$ be irreducible $\CAT$ cube complexes with no free faces, each admitting properly discontinuous\footnote{It is not sufficient for our arguments that $X$ and $Y$ admit cocompact groups actions. However, this is only due to the need for Lemma~\ref{geometric implies nonelementary} in the proof of Proposition~\ref{colouring}.}, cocompact group actions. Let $\G$ be a group and let $\G\acts X$ and $\G\acts Y$ be $\ell^1$--minimal, non-elementary actions with the same $\ell^1$ length function. Then there exists a (unique) $\G$--equivariant cubical isomorphism $\Phi\colon X\ra Y$.
\end{thmintro}

This extends known results for groups acting on trees \cite{Culler-Morgan} and $2$--dimensional right-angled Artin groups acting on square complexes \cite{Charney-Margolis}.

We stress that the group $\G$ is not required to act properly or cocompactly in Theorem~\ref{NFF MLSR}. Actions are \emph{$\ell^1$--minimal}\footnote{These are closely related to the \emph{essential actions} from \cite{CS}, see Lemma~\ref{minimal vs essential} below.} when they do not leave invariant any proper convex subcomplexes of the barycentric subdivision. \emph{Non-elem\-ent\-ary} actions are those without any finite orbits in the visual compactification. 

Having \emph{no free faces} is equivalent to the geodesic extension property for the $\CAT$ metric\footnote{See Definitions~II.5.7,~II.5.9 and Proposition~II.5.10 in \cite{BH}.}. 
Note that, when $X$ has no free faces, every properly discontinuous, cocompact action of a non-virtually-cyclic group is $\ell^1$--minimal and non-elementary (see Lemmas~\ref{geometric implies nonelementary} and~\ref{minimal vs essential}).

\begin{rmk*} 
\begin{enumerate} 
\item[]
\item We are not requiring the group $\G$ to be countable, nor --- in case $\G$ has a topology --- that actions have continuous orbits.
\item Although we have preferred to state Theorem~\ref{NFF MLSR} for \emph{cube} complexes, it holds, more generally, for \emph{cuboid} complexes and $\G$--actions respecting the cell structures. In cuboid complexes, edges can have arbitrary real lengths. The price to pay is that the map $\Phi$ (provided by Theorem~\ref{NFF MLSR}) becomes a mere $\ell^1$--isometry in this context, and it will not take vertices to vertices in general. Still, we emphasise that $\ell^1$--isometries are not as violent as $\CAT$--isometries can be: $\Phi$ will still take median walls to median walls, it just might not respect which of these walls are designated as hyperplanes by the cellular structure. 

With such necessary changes, all results in this paper apply to $\CAT$ cuboid complexes. We refer the reader to Section~\ref{cuboid section} for more information. 

\item A more general version of Theorem~\ref{NFF MLSR} holds for many cube complexes that do have free faces, for instance all Davis complexes associated to (irreducible) right-angled Coxeter groups. See Theorem~\ref{most general statement} below.
\item Irreducibility --- on the other hand --- cannot be removed from the statement of the theorem. Given $\G\acts X$, the induced action on the barycentric subdivision $\G\acts X'$ and the diagonal action $\G\acts X\x X$ always have the same length function $2\cdot\ell_X$.  
\end{enumerate}
\end{rmk*}

In the next paper \cite{BF2}, we will remove the no-free-faces assumption from Theorem~\ref{NFF MLSR}, at the cost of imposing the strong restriction that $\G$ be a Gromov hyperbolic group acting properly and cocompactly on $X$ and $Y$. We will again rely on the general framework developed in the present paper, namely on the connection between marked length-spectrum rigidity and the extension of certain boundary maps (Theorem~\ref{from ell to moeb} below). However, a completely different approach will then be required in order to extend boundary maps to cubical isomorphisms:
rather than rely on Theorem~\ref{ext Moeb} below, we will aim to recognise limit sets of hyperplanes and halfspaces in the visual boundary, whereby hyperbolicity will play a key role.

\medskip
Theorem~\ref{NFF MLSR} can be used to compactify spaces of cubulations. The picture that one should have in mind is that of Culler and Vogtmann's Outer Space \cite{Culler-Vogtmann}, where every point is an action of the free group $F_n$ on a simplicial tree. This space embeds into the space of projectivised length functions $\mbb{P}(\R^{F_n})$ and can be compactified by adding length functions of certain non-proper actions on real trees \cite{Culler-Morgan,Cohen-Lustig}. 

Given a finitely generated group $\G$ and an integer $D\geq 0$, let us denote by $\mathrm{Cub}_D(\G)$ the collection of all actions with unbounded orbits of $\G$ on $\CAT$ cuboid complexes with countably many vertices and dimension at most $D$. We identify actions that are $\G$--equivariantly isometric or differ by a homothety. Each of these actions must admit elements with positive translation length by a result of Sageev (see Theorem~5.1 in \cite{Sag95}), so we can consider the composition
\[\mathrm{Cub}_D(\G)\overset{\ell}{\longrightarrow}\R^{\G}\setminus\{0\}\longrightarrow\mbb{P}(\R^{\G}),\]
where the first arrow simply assigns to every action its $\ell^1$ length function. Theorem~\ref{NFF MLSR} states that this composition is injective on a certain subset of $\mathrm{Cub}_D(\G)$. We endow $\R^{\G}$ with the product topology and $\mbb{P}(\R^{\G})$ with the quotient topology.

Combining previous work of the second author on the geometry of finite-rank median spaces \cite{Fio2} with a standard Bestvina--Paulin construction \cite{Bestvina,Paulin-arb}, we prove:

\begin{propintro}\label{cptf main}
Let $\G$ be finitely generated. For every $D\geq 0$, the image of the map $\mathrm{Cub}_D(\G)\ra\mbb{P}(\R^{\G})$ is relatively compact.
\end{propintro}

We draw the reader's attention to the generality of Proposition~\ref{cptf main}, since actions in $\mathrm{Cub}_D(\G)$ are only required to have unbounded orbits. Every point in the closure of the image of $\mathrm{Cub}_D(\G)$ in $\mbb{P}(\R^{\G})$ is the (projectivised) length function of an action of $\G$ on a finite rank median space with unbounded orbits (Proposition~\ref{cptf 3.5}).

Theorem~\ref{NFF MLSR} and Proposition~\ref{cptf main} can in particular be applied to the \emph{Salvetti blow-ups} constructed by Charney, Stambaugh and Vogtmann in \cite{Charney-Stambaugh-Vogtmann} in order to study the group of untwisted outer automorphisms $U(\G)\leq\Out(\G)$ of a right-angled Artin group $\G$.

\begin{corintro}\label{CSV application}
Let $\G$ be a right-angled Artin group that does not split as a direct product. The Charney--Stambaugh--Vogtmann Outer Space for $U(\G)$ continuously injects into $\mbb{P}(\R^{\G})$ with relatively compact image.
\end{corintro}

The $2$--dimensional case was obtained in \cite{Charney-Margolis,Vijayan}. We would expect the map in Corollary~\ref{CSV application} to be a topological embedding, but this is not immediate and we have not investigated the matter in the present paper.

\medskip \noindent
{\bf Cross ratios and more on Theorem~\ref{NFF MLSR}.} The proof of Theorem~\ref{NFF MLSR} consists of two steps. Roughly speaking, we first need to construct a \emph{boundary map at infinity} ``$f\colon\partial X\dashrightarrow\partial Y$'' and use the assumption on length functions to show that $f$ preserves a \emph{combinatorial cross ratio} on these boundaries. Things are complicated by the fact that the map $f$ will not be defined everywhere. 

The second step then amounts to showing that these boundary maps can be extended to cubical isomorphisms. It is here that it becomes crucial to endow our cube complexes with the $\ell^1$ metric. Indeed, the corresponding horofunction compactification $X\cup\partial X$, aka the \emph{Roller compactification}, is endowed with a structure of \emph{median algebra} \cite{Roller,Nevo-Sageev,Fio1}, unlike the \emph{visual compactification} $X\cup\partial_{\infty}X$. This means that any three points $x,y,z\in\partial X$ have a canonical barycentre $m(x,y,z)\in X$ and we will essentially aim to construct the isomorphism $\Phi$ by setting $\Phi(m(x,y,z))=m(f(x),f(y),f(z))$.

We now go into more detail. Given vertices $x,y,z,w\in X$, we defined their \emph{cross ratio} in previous work with Incerti-Medici \cite{BFI-new}:
\begin{align*}
\Cr(x,y,z,w)&=d(x,w)+d(y,z)-d(x,z)-d(y,w) \\
&=\#\mscr{W}(x,z|y,w)-\#\mscr{W}(x,w|y,z).
\end{align*}
Here $\mscr{W}(x,z|y,w)$ denotes the collection of hyperplanes separating $x$ and $z$ from $y$ and $w$. In Section~\ref{crossratio section} below, we show that this function extends \emph{continuously} to the Roller boundary $\partial X$ (which is totally disconnected).

We will actually only be interested in the restriction of the cross ratio to the \emph{regular boundary} $\partial_{\rm reg}X\cu\partial X$. This boundary was introduced in \cite{Fernos,Fernos-Lecureux-Matheus} and has the advantage of sharing a large subset with the contracting boundary $\partial_cX$ of Charney and Sultan \cite{Charney-Sultan}; see Theorem~D in \cite{BF2}. The first step in the proof of Theorem~\ref{NFF MLSR} is then the following.

\begin{thmintro}\label{from ell to moeb}
Let $X$ and $Y$ be locally finite, finite dimensional, irreducible $\CAT$ cube complexes. Given $\ell^1$--minimal, non-elementary actions $\G\acts X$ and $\G\acts Y$ with the same $\ell^1$ length function, there exists a $\G$--equivariant, cross-ratio preserving bijection $f\colon \mc{A}\ra \mc{B}$ between nonempty, $\G$--invariant subsets $\mc{A}\cu\partial_{\rm reg}X$ and $\mc{B}\cu\partial_{\rm reg}Y$.
\end{thmintro}

Note that the hypotheses of Theorem~\ref{from ell to moeb} are much weaker than those of Theorem~\ref{NFF MLSR}. This will allow us to rely on this result also in \cite{BF2}.

Once we have Theorem~\ref{from ell to moeb}, the proof of Theorem~\ref{NFF MLSR} is completed with:

\begin{thmintro}\label{ext Moeb}
Let $X$ and $Y$ be irreducible $\CAT$ cube complexes with no free faces, endowed with $\ell^1$--minimal, non-elementary actions of a group $\G$. Suppose that $X$ and $Y$ also admit properly discontinuous, cocompact group actions. Let ${f\colon \mc{A}\ra \mc{B}}$ be a $\G$--equivariant, cross-ratio preserving bijection, where ${\mc{A}\cu\partial_{\rm reg}X}$ and $\mc{B}\cu\partial_{\rm reg}Y$ are non\-em\-pty and $\G$--invariant. Then there exists a unique $\G$--equivariant cubical isomorphism $\Phi\colon X\ra Y$ extending $f$. 
\end{thmintro}

If we were to replace the actions $\G\acts X$ and $\G\acts Y$ with proper cocompact actions on two $\CA$ spaces, the analogue of Theorem~\ref{from ell to moeb} would be well-known. We would actually obtain a homeomorphism of \emph{Gromov boundaries} that preserves cross ratios everywhere\footnote{See e.g.\ Th\'eor\`eme~2.2 in \cite{Otal-IberoAm} or Section~5 in \cite{Kim-Top} for the Riemannian case. Their arguments actually apply to any $\CA$ space.}. 

Note, however, that this procedure fails for general $\CAT$ spaces: We can still obtain a homeomorphism of \emph{contracting boundaries} --- assuming these are nonempty --- but it is not clear whether the cross ratio arising from the $\CAT$ metric will be preserved in general. We will rely on the median properties of $\ell^1$ metrics in order to circumvent these issues for $\CAT$ cube complexes. Still, the arguments become more involved in this context and we will make use of results of \cite{Fernos-Lecureux-Matheus} on random walks. See Section~\ref{MLSR section}.

On the other hand, no analogue of Theorem~\ref{ext Moeb} is available even in $\CA$ spaces and this is perhaps the most innovative feature of our work. We stress that the lack of an analogue of Theorem~\ref{ext Moeb} for $\CA$ Riemannian manifolds is the main reason why the \emph{marked length-spectrum rigidity conjecture} (Problem~3.1 in \cite{Burns-Katok}) is still wide open (see Theorem~5.1 in \cite{Biswas1}). 

\medskip \noindent
{\bf On the proof of Theorem~\ref{ext Moeb}.} 
As mentioned earlier, we would like to define the isomorphism $\Phi$ by representing vertices $v\in X$ as $m(x,y,z)$ for $x,y,z\in\partial X$ and then setting $\Phi(v)=m(f(x),f(y),f(z))$. 

There are two main issues with this approach:
\begin{enumerate}
\item It is not a priori clear that the map $\Phi$ is well-defined. There might exist two triples $(x,y,z),(x',y',z')\in (\partial X)^3$ for which the medians $m(x,y,z)$ and $m(x',y',z')$ coincide, while $m(f(x),f(y),f(z))$ and $m(f(x'),f(y'),f(z'))$ are different.
\item The map $f$ provided by Theorem~\ref{from ell to moeb} is only defined on $\mc{A}\cu\partial_{\rm reg}X$, so we need to be able to pick $x,y,z$ within $\partial_{\rm reg}X$. In many cases, however, there exist vertices that are not representable as medians of regular points. An example is provided by the vertices that we add to any cube complex $X$ when we barycentrically subdivide it.
\end{enumerate}

A naive solution to issue~(1) might aim to recognise purely in terms of cross ratios whether the condition ``$m(x,y,z)=m(x',y',z')$'' holds. This proves to be a delicate matter, as the cross ratios of the six points alone will in general not suffice. The problem is illustrated in Figure~\ref{spanning a cube} in the simple case with $x=x'$ and $y=y'$. In both cube complexes, all $4$--tuples obtained by permuting $x$, $y$, $z$ and $z'$ have vanishing cross ratio, although we have $m(x,y,z)=m(x,y,z')$ on the left and $m(x,y,z)\neq m(x,y,z')$ on the right.

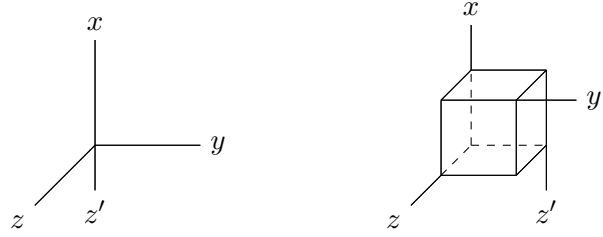
\begin{figure}
\begin{tikzpicture}
\begin{scope}
\draw [fill] (-0.4,-0.4) -- (0.4,0.4);
\draw [fill] (0.4,0.4) -- (0.4,1.8);
\draw [fill] (0.4,0.4) -- (1.8,0.4);
\draw [fill] (0.4,0.4) -- (0.4,-0.2);
\node [below left] at (-0.4, -0.4) {$z$};
\node [above] at (0.4,1.8) {$x$};
\node [right] at (1.8,0.4) {$y$};
\node [below] at (0.4,-0.2) {$z'$};
\end{scope}

\begin{scope}[xshift=5cm]
\draw [fill] (0,0) -- (0,1);
\draw [fill] (0,0) -- (1,0);
\draw [fill] (1,0) -- (1, 1);
\draw [fill] (0,1) -- (1, 1);
\draw [fill] (0,1) -- (0.4, 1.4);
\draw [fill] (1,0) -- (1.4, 0.4);
\draw [fill] (1,1) -- (1.4, 1.4);
\draw [fill] (0.4,1.4) -- (1.4,1.4);
\draw [dashed] (0,0) -- (0.4, 0.4);
\draw [dashed] (0.4,0.4) -- (0.4, 1.4);
\draw [dashed] (0.4,0.4) -- (1.4, 0.4);
\draw [fill] (1.4,0.4) -- (1.4, 1.4);
\draw [fill] (0,0) -- (-0.4, -0.4);
\draw [fill] (0.4,1.4) -- (0.4, 2);
\draw [fill] (1,1) -- (1.8, 1);
\draw [fill] (1.4,0.4) -- (1.4, -0.2);
\node [below left] at (-0.4, -0.4) {$z$};
\node [above] at (0.4,2) {$x$};
\node [right] at (1.8,1) {$y$};
\node [below] at (1.4,-0.2) {$z'$};
\end{scope}
\end{tikzpicture}
\caption{The cube complex on the left is a tree with a single branch point and four boundary points $x,y,z,z'$. On the right, $x,y,z,z'$ lie in the Roller boundary of $\R^3$.}
\label{spanning a cube} 
\end{figure}

This is a characteristic feature of cross ratios on cube complexes of dimension at least 3. Indeed, cross ratios obtained by permuting coordinates of a $4$--tuple $(x,y,z,w)$ will only provide information on the respective differences between the cardinalities of $\mscr{W}(x,y|z,w)$, $\mscr{W}(x,z|y,w)$, $\mscr{W}(x,w|y,z)$, and never on these three quantities themselves. Note in this regard that these three sets of hyperplanes are pairwise transverse.

It is possible to circumvent these problems by representing vertices $v$ as medians of triples $(x,y,z)$ that satisfy the additional condition $x\op_z y$ (Definition~\ref{opposite defn}). This means that all geodesics with endpoints $x$ and $y$ must pass through $v$. This forces one of the three sets $\mscr{W}(x,y|z,w)$, $\mscr{W}(x,z|y,w)$, $\mscr{W}(x,w|y,z)$ to be empty for every choice of $w$. 

In relation to Figure~\ref{spanning a cube}, we have $x\op_zy$ on the left, but not on the right. For more details, we refer the reader to Sections~\ref{opposite sect},~\ref{general setup sect} and~\ref{from OR sect}.

Issue~(2) is now exacerbated by the fact that we need to represent vertices $v\in X$ as $v=m(x,y,z)$ not only for $x,y,z\in\partial_{\rm reg}X$, but also with $x\op_z y$. We say that $v$ is an \emph{OR median} if this is possible (Definition~\ref{OR defn}). 

The most technical part of the paper is devoted to showing that, under the hypotheses of Theorem~\ref{ext Moeb}, there always exists at least one OR median (Theorem~\ref{OR exist thm}). It is then possible to define an isometry $\Phi$ between the subsets of OR medians (Theorem~\ref{from OR thm}). Finally, the proof of Theorem~\ref{ext Moeb} is completed by extending this partial isometry to a global isomorphism of cube complexes. This is achieved by relying on the following result, which we believe of independent interest.

\begin{thmintro}\label{ext partial iso thm}
Let $X$ and $Y$ be irreducible, locally finite $\CAT$ cube complexes with no free faces, endowed with cocompact actions $G_1\acts X$ and $G_2\acts Y$.  Let nonempty subsets $A\cu X$ and $B\cu Y$ be invariant, respectively, under $G_1$ and $G_2$. Then every distance-preserving bijection $\phi\colon A\ra B$ admits a unique extension to a cubical isomorphism $\Phi\colon X\ra Y$.
\end{thmintro}

It is interesting to remark that having no free faces is not strictly necessary in our proof of Theorem~\ref{ext partial iso thm}. What we really need is essentiality, hyperplane-essentiality\footnote{See Section~\ref{automorphisms section} for a definition.} and Lemma~\ref{semi-straight+}, which should be considerably weaker.

Thus, Theorems~\ref{NFF MLSR} and~\ref{ext Moeb} fundamentally only require absence of free faces in order to show the existence of an OR median (Theorem~\ref{OR exist thm} and Section~\ref{OR exist sect}). In particular, we obtain the following extension of Theorem~\ref{NFF MLSR}.

\begin{thmintro}\label{most general statement}
Let $X$ and $Y$ be locally finite, finite dimensional, irreducible, sector-heavy $\CAT$ cube complexes. Let $\ell^1$--minimal, non-elementary actions $\G\acts X$ and $\G\acts Y$ have the same $\ell^1$ length function. Then there exists a $\G$--equivariant isometry between the subsets of OR medians.
\end{thmintro}

Here, we say that $X$ is \emph{sector-heavy} if the intersection of any family of pairwise-transverse halfspaces contains a halfspace (Definition~\ref{sector-heavy defn}). This is automatically satisfied when $X$ has no free faces (Proposition~\ref{my corner}).

Theorem~\ref{most general statement} applies for instance to all Davis complexes associated to (irreducible) right-angled Coxeter groups. Although these often have many free faces, they are always sector-heavy and all their vertices are OR medians.

In relation to Theorem~\ref{most general statement}, we should mention that we are not aware of any (reasonable) cube complex where the subset of OR medians is empty. Thus, the following is in order:

\begin{quest*}
Let $X$ be a locally finite, finite dimensional, irreducible $\CAT$ cube complex admitting an $\ell^1$--minimal, non-elementary action of a finitely generated group $\G$. Does $X$ always contain an OR median?
\end{quest*}

{\bf Structure of the paper.} At the beginning of each subsection, we use a Standing Assumptions environment to declare various hypotheses that are implicit in the subsequent results and will not be declared in their statements.

In Section~\ref{prelims} we collect various facts on $\CAT$ cube complexes that are more or less well-known. Section~\ref{cuboid section} is devoted to cube complexes with variable edge lengths; we describe the few changes needed to adapt the arguments in the rest of the paper to the cuboidal context. In Section~\ref{crossratio section}, we introduce the cross ratio on the Roller boundary and prove that it is continuous. Section~\ref{MLSR section} revolves around the proof of Theorem~\ref{from ell to moeb}. We will introduce \emph{neatly contracting automorphisms} (Definition~\ref{neatly contracting defn}), which will allow us to link length functions and cross ratios (Theorem~\ref{cr from ell}).

The proof of Theorem~\ref{ext Moeb} is the main subject of Section~\ref{from moeb to cubical iso}, with an elementary overview being provided in Section~\ref{general setup sect}. In particular, Theorem~\ref{ext partial iso thm} is obtained in Section~\ref{ext partial iso sect}. Finally, Section~\ref{cptf sect} is devoted to compactifications of spaces of cubulations; we prove Proposition~\ref{cptf main} and Corollary~\ref{CSV application} there.

\medskip
{\bf Acknowledgements.}
This project stems from a previous collaboration with M.\ Incerti-Medici, whom we also thank for providing Figure~\ref{spanning a cube}. We are indebted to Pierre-Emmanuel Caprace, Christopher Cashen, Matt Cordes, Ruth Charney, Cornelia Dru\c{t}u, Ilya Gekhtman, Mark Hagen, Tobias Hartnick, John Mackay, Beatrice Pozzetti, Michah Sageev, Viktor Schroeder and Ric Wade for many helpful conversations. 

We are especially grateful to Talia Fern\'os for suggesting that we look at the regular boundary and to Sam Shepherd for suggesting the current proof of Proposition~\ref{continuity} via Lemma~\ref{GP continuous} (considerably simpler than our own). 

We would also like to thank the referees for their many constructive comments, which improved the exposition.

We thank the organisers of the following conferences, where part of this work was carried out: \emph{Third GEAR (Junior) Retreat}, Stanford, 2017; \emph{Moduli Spaces}, Ventotene, 2017; \emph{Young Geometric Group Theory VII}, Les Diablerets, 2018; \emph{Topological and Homological Methods in Group Theory}, Bielefeld, 2018; \emph{3--Manifolds and Geometric Group Theory}, Luminy, 2018; \emph{Representation varieties and geometric structures in low dimensions}, Warwick, 2018. EF also thanks M.\ Incerti-Medici, Viktor Schroeder and Anna Wienhard for his visits at UZH and Universit\"at Heidelberg. 

JB was supported by the Swiss National Science Foundation under Grant 200020$\setminus$175567. EF was supported by the Clarendon Fund and the Moussouris Scholarship of Merton College, Oxford.

\addtocontents{toc}{\protect\setcounter{tocdepth}{1}}
\section{Preliminaries.}\label{prelims}

\subsection{$\CAT$ cube complexes.}\label{ccc prelims}

We will assume a certain familiarity with basic properties of $\CAT$ cube complexes. The reader can consult for instance \cite{Sageev} and the first sections of \cite{Chatterji-Niblo,CS,Nevo-Sageev,CFI} for an introduction to the subject. In this subsection, we merely fix notation and recall a few facts that will be necessary over the course of the paper.

Let $X$ be a simply connected cube complex satisfying Gromov's no-$\triangle$-condition; see 4.2.C in \cite{Gromov-HypGps} and Chapter~II.5 in \cite{BH}. As customary, we will not distinguish between $X$ and its $0$--skeleton, denoting both simply by $X$. It will always be clear from the context how to interpret the notation, although points $v\in X$ will generally be understood to be vertices.

The Euclidean metrics on the cubes of $X$ fit together to yield a $\CAT$ metric on $X$, thus justifying the terminology \emph{``$\CAT$ cube complex''}. Such $\CAT$ metric, however, will generally be of no interest to us and we instead prefer to endow $X$ with its \emph{combinatorial metric} (aka \emph{$\ell^1$ metric}). This is the induced path metric obtained by endowing each cube $[0,1]^k\cu X$ with the restriction of the $\ell^1$ metric of $\R^k$. 

The notation $d(-,-)$ will always refer to this latter metric. We will employ the more familiar and concise expression ``$\CAT$ cube complex'' with the meaning of ``simply connected cube complex satisfying Gromov's no-$\triangle$-con\-di\-tion and endowed with its combinatorial metric $d$''. Unless specified otherwise, all geodesics will be implicitly assumed to be combinatorial geodesics contained in the $1$--skeleton. We refer to bi-infinite geodesics simply as \emph{lines}.

Given vertices $v,w\in X$, their combinatorial distance $d(v,w)$ coincides with the length of a shortest path connecting $v$ and $w$ within the $1$--skeleton; as such, it is always a natural number. In finite dimensional cube complexes, the $\CAT$ metric and the combinatorial metric are bi-Lipschitz equivalent and complete. It is however important to remark that $(X,d)$ is \emph{uniquely geodesic} only when $X$ is a tree.

The metric space $(X,d)$ is a \emph{median space}. This means that, given any three points $p_1,p_2,p_3\in X$, there exists a unique point $m=m(p_1,p_2,p_3)\in X$ such that $d(p_i,p_j)=d(p_i,m)+d(m,p_j)$ for all $1\leq i<j\leq 3$. We refer to $m(p_1,p_2,p_3)$ as the \emph{median} of $p_1$, $p_2$ and $p_3$. The map ${m\colon X^3\ra X}$ endows $X$ with a structure of \emph{median algebra}. We refer the reader to \cite{Roller,CDH,Fio1} for more information on median geometry.

We denote by $X'$ the \emph{barycentric subdivision} (aka \emph{first cubical subdivision}) of $X$. This is the $\CAT$ cube complex obtained by adding a vertex $v(c)$ at the centre of each cube $c\cu X$; vertices $v(c)$ and $v(c')$ are joined by an edge of $X'$ if $c$ is a codimension-one face of $c'$ or vice versa. Each $k$--cube of $X$ gives rise to $2^k$ $k$--cubes of $X'$. We assign all edges of $X'$ length $1$, so the natural inclusion $X\hookrightarrow X'$ is a homothety of a factor $2$.

Let $\mscr{W}(X)$ and $\mscr{H}(X)$ be, respectively, the sets of hyperplanes and of halfspaces of $X$. We simply write $\mscr{W}$ and $\mscr{H}$ when there is no need to specify the cube complex. Let $\mf{h}^*$ denote the complement of the halfspace $\mf{h}$. The set $\mscr{H}$ is endowed with the order relation given by inclusions; the involution $*$ is order reversing. The triple $(\mscr{H},\cu,*)$ is thus a \emph{po{\bf c}set} \cite{Sageev}.

Three hyperplanes of $X$ form a \emph{facing triple} if they are pairwise disjoint and none of the three separates the other two. Three halfspaces form a \emph{facing triple} if they are pairwise disjoint. Two halfspaces $\mf{h}$ and $\mf{k}$ are \emph{nested} if either $\mf{h}\cu\mf{k}$ or $\mf{k}\cu\mf{h}$. Two distinct hyperplanes are \emph{transverse} if they cross. 

We remark that every intersection $\mf{w}_1\cap ...\cap\mf{w}_k$ of pairwise transverse hyperplanes $\mf{w}_1,...,\mf{w}_k$ can itself be regarded as a $\CAT$ cube complex; its cells are precisely the intersections $\mf{w}_1\cap ...\cap\mf{w}_k\cap c$, where $c\cu X$ is a cube. 

If $e\cu X$ is an edge, $\mf{w}(e)$ denotes the hyperplane crossed by $e$. We say that a hyperplane $\mf{w}$ is \emph{adjacent} to a point $v\in X$ if $\mf{w}=\mf{w}(e)$ for an edge $e$ incident to $v$. The set of hyperplanes adjacent to $v$ will be denoted by $\mscr{W}_v$.

The \emph{link} of $v$ (denoted $\lk_X(v)$ or just $\llk v$) is the graph\footnote{This a non-standard definition. However, the usual notion of link is recovered by taking the flag completion of our graph $\llk v$, thus justifying the abuse of notation.} having $\mscr{W}_v$ as vertex set and edges joining two hyperplanes exactly when they are transverse. Equivalently, vertices of $\llk v$ are edges of $X$ incident to $v$, two edges being connected by an edge of $\llk v$ if and only if they span a square in $X$. Cubes based at $v$ are in one-to-one correspondence with complete subgraphs of $\llk v$; we refer to the latter as \emph{cliques}. We write $\deg(v):=\#\mscr{W}_v=\#\llk^0v$.

Given vertices $e,f\in\llk v$, we write $\delta(e,f)$ for the minimal number of edges of $\llk v$ in a path from $e$ to $f$. If no such path exists, we declare ${\delta(e,f)=+\infty}$. When $\llk v$ is connected, $\delta$ is a metric on its $0$--skeleton. We will sometimes refer to vertices of $\llk v$ simply as \emph{elements of $\llk v$}, in order to avoid confusion with vertices of $X$.

We say that two halfspaces $\mf{h}$ and $\mf{k}$ are \emph{transverse} if they are bounded by transverse hyperplanes; equivalently, the four intersections $\mf{h}\cap\mf{k}$, $\mf{h}^*\cap\mf{k}$, $\mf{h}\cap\mf{k}^*$ and $\mf{h}^*\cap\mf{k}^*$ are all nonempty. A hyperplane is \emph{transverse} to a halfspace $\mf{h}$ if it is transverse to the hyperplane that bounds $\mf{h}$. Finally, we say that two subsets $\mc{U},\mc{V}\cu\mscr{W}$ are \emph{transverse} if every element of $\mc{U}$ is transverse to every element of $\mc{V}$. Given subsets $A,B\cu X$, we adopt the notation
\[\mscr{H}(A|B)=\{\mf{h}\in\mscr{H}\mid B\cu\mf{h},~A\cu\mf{h}^*\},\]
\[\mscr{W}(A|B)=\{\mf{w}\in\mscr{W}\mid\mf{w}\text{ separates } A\text{ and }B\},~~~\mscr{W}(A)=\bigcup_{x,y\in A}\mscr{W}(x|y).\]
For vertices $u,v\in X$, we have $d(u,v)=\#\mscr{W}(u|v)$. 

We will generally conflate geodesics and their images as subsets of $X$. Every geodesic $\g\cu X$ can be viewed as a collection of edges; distinct edges $e,e'\cu\g$ must yield distinct hyperplanes $\mf{w}(e)$ and $\mf{w}(e')$. Denote by $\mscr{W}(\g)$ the collection of hyperplanes crossed by (the edges of) $\g$. If two geodesics $\g$ and $\g'$ share an endpoint $v\in X$, their union $\g\cup\g'$ is again a geodesic if and only if $\mscr{W}(\g)\cap\mscr{W}(\g')=\emptyset$. The following is then a simple observation (cf.\ Lemma~2.1 in \cite{BFI-new}).

\begin{lem}\label{straight rays 2}
Given $v\in X$ and rays $r_1,r_2\cu X$ based at $v$, consider the sets $\mscr{W}_i=\mscr{W}(r_i)\cap\mscr{W}_v$. The union $r_1\cup r_2$ is a line if and only if $\mscr{W}_1\cap\mscr{W}_2=\emptyset$.
\end{lem}

Given a vertex $v\in X$, we denote by $\s_v\cu\mscr{H}$ the set of all halfspaces containing $v$. It satisfies the following properties:
\begin{enumerate}
\item given any two halfspaces $\mf{h},\mf{k}\in\s_v$, we have $\mf{h}\cap\mf{k}\neq\emptyset$;
\item for any hyperplane $\mf{w}\in\mscr{W}$, a side of $\mf{w}$ lies in $\s_v$;
\item every descending chain of halfspaces in $\s_v$ is finite. 
\end{enumerate}
Subsets $\s\cu\mscr{H}$ satisfying $(1)$--$(3)$ are known as \emph{DCC ultrafilters}. We refer to a set $\s\cu\mscr{H}$ satisfying only $(1)$ and $(2)$ simply as an \emph{ultrafilter}. 

Let $\iota\colon X\ra 2^{\mscr{H}}$ denote the map that takes each vertex $v$ to the set $\s_v$. Its image $\iota (X)$ is precisely the collection of all DCC ultrafilters. Endowing $2^{\mscr{H}}$ with the product topology, we can consider the closure $\overline{\iota(X)}$, which coincides with the set of all ultrafilters. Equipped with the subspace topology, this is a totally disconnected, compact, Hausdorff space known as the \emph{Roller compactification} of $X$ \cite{BCGNW,Nevo-Sageev}; we denote it by $\overline X$. 

The \emph{Roller boundary} $\partial X$ is defined as the difference $\overline X\setminus X$. The inclusion $\iota\colon X\ra\overline X$ is always continuous\footnote{Note that $\iota$ is only defined on the $0$--skeleton, which has the discrete topology.}. If, moreover, $X$ is locally finite, then $\iota$ is a topological embedding, $X$ is open in $\overline X$ and $\partial X$ is compact. We prefer to imagine $\partial X$ as a set of points at infinity, rather than a set of ultrafilters. We will therefore write $x\in\partial X$ for points in the Roller boundary and employ the notation $\s_x\cu\mscr{H}$ only to refer to the ultrafilter representing $x$.

We remark that the Roller boundary $\partial X$ is naturally homeomorphic to the horofunction boundary of the metric space $(X,d)$. This is an unpublished result of U.\ Bader and D.\ Guralnik; see \cite{Caprace-Lecureux} or \cite{Fernos-Lecureux-Matheus} for a proof. On the other hand, the horofunction boundary with respect to the $\CAT$ metric coincides with the \emph{visual boundary} $\partial_{\infty}X$.

If $r\cu X$ is a (combinatorial) ray, we denote by $r(0)$ its initial vertex and by $r(n)$ its $n$--th vertex. Given a hyperplane $\mf{w}\in\mscr{W}$, there exists a unique side $\mf{h}$ of $\mf{w}$ such that $r\setminus\mf{h}$ is bounded. The collection of all such halfspaces forms an ultrafilter and we denote by $r^+\in\partial X$ the corresponding point; we refer to $r^+$ as the \emph{endpoint at infinity} of $r$.

Fixing a basepoint $v\in X$, every point of $\partial X$ is of the form $r^+$ for a ray $r$ based at $v$. We obtain a bijection between points of $\partial X$ and rays based at $v$, where we need to identify the rays $r_1$ and $r_2$ if $\mscr{W}(r_1)=\mscr{W}(r_2)$. See Proposition~A.2 in \cite{Genevois} for details.

Given $v\in X$ and $\mf{h}\in\mscr{H}$, we have $v\in\mf{h}$ if and only if $\mf{h}\in\s_v$. We can extend the halfspace $\mf{h}\cu X$ to a subset $\overline{\mf{h}}\cu\overline X$ by declaring that a point $x\in\partial X$ lies in $\overline{\mf{h}}$ if and only if $\mf{h}\in\s_x$. Our notation is justified by the observation that $\overline{\mf{h}}$ is simply the closure of $\mf{h}$ in $\overline X$. Note that $\overline{\mf{h}}$ and $\overline{\mf{h}^*}$ provide a partition of $\overline X$. Thus, given subsets $A,B\cu\overline X$, we can still employ the notation $\mscr{W}(A|B)$ and $\mscr{H}(A|B)$. For instance,
\[\mscr{H}(A|B)=\{\mf{h}\in\mscr{H}\mid B\cu\overline{\mf{h}},~A\cu\overline{\mf{h}^*}\}.\]
It is immediate from the definitions that $\mscr{W}(x|y)\neq\emptyset$ if and only if $x,y\in\overline X$ are distinct. For ease of notation, we will generally omit the overline symbol outside of the current section, thus not distinguishing between a halfspace $\mf{h}\cu X$ and its extension $\overline{\mf{h}}\cu\overline X$.

Given $x,y\in\overline X$, the \emph{interval} between $x$ and $y$ is the set
\[I(x,y)=\{z\in\overline X\mid\mscr{W}(z|x,y)=\emptyset\}.\]
We always have $I(x,x)=\{x\}$. If $x,y\in X$, the interval $I(x,y)$ is contained in $X$ and coincides with the union of all geodesics from $x$ to $y$. For any three points $x,y,z\in\overline X$, there exists a unique point $m(x,y,z)\in\overline X$ that lies in all three intervals $I(x,y)$, $I(y,z)$ and $I(z,x)$. We refer to $m(x,y,z)$ as the \emph{median} of $x$, $y$ and $z$ and remark that it is represented by the ultrafilter
\[(\s_x\cap\s_y)\cup(\s_y\cap\s_z)\cup(\s_z\cap\s_x).\] 
The map $m\colon\overline{X}^3\ra\overline X$ is continuous. It coincides with the usual median operator on $X$, thus extending to $\overline X$ the median algebra structure of $X$.

A subset $C\cu\overline X$ is \emph{convex} if $I(x,y)\cu C$ for all $x,y\in C$. Thus, a subset ${C\cu X}$ is convex if and only if all geodesics between points of $C$ are contained in $C$. Given $u,v\in\overline X$, the interval $I(u,v)\cu\overline X$ is convex and so is every halfspace $\mf{h}\cu X$. In fact, halfspaces are precisely those nonempty convex subsets of $X$ whose complement is convex and nonempty.

Intersections of convex sets are convex. Given a subset ${A\cu\overline X}$, we denote by $\hull(A)$ the smallest convex subset of $\overline X$ containing $A$. We remark that $\hull(A)\cap X$ coincides with the intersection of all halfspaces $\mf{h}\cu X$ with the property that $A$ is contained in the extension $\overline{\mf{h}}\cu\overline X$. If ${A\cu X}$, we have $\hull(A)\cu X$, as $X$ is a convex subset of $\overline X$. In this case, $\hull(A)$ is itself a $\CAT$ cube complex and a subcomplex of $X$; its hyperplane set is identified with the subset $\mscr{W}(A)\cu\mscr{W}(X)$.

If $C\cu\overline X$ is convex and closed, there exists a unique \emph{gate-projection} $\pi_C\colon\overline X\ra C$, namely a map with the property that $\pi_C(x)\in I(x,y)$ for every $x\in X$ and every ${y\in C}$. When $u,v\in\overline X$ and $C=I(u,v)$, the gate-projection is given by $x\mapsto m(u,v,x)$.

If $C\cu X$ is convex, its closure $\overline C\cu\overline X$ is closed and convex. As such, it admits a gate-projection $\pi_{\overline C}\colon\overline X\ra\overline C$; we will simply write $\pi_C\colon X\ra C$ for the restriction of $\pi_{\overline C}$ to $X$. Note that this is $1$--Lipschitz and it is just the nearest-point projection with respect to the combinatorial metric $d$.

Given a closed convex subset $C\cu\overline X$ and two points $x,y\in\overline X$, we always have $\mscr{W}(x|\pi_C(x))=\mscr{W}(x|C)$ and $\mscr{W}(x|y)\cap\mscr{W}(C)=\mscr{W}(\pi_C(x)|\pi_C(y))$; see Proposition~2.3 in \cite{Fio1}. If $C_1,...,C_k\cu\overline X$ are convex and pairwise intersect, then ${C_1\cap ...\cap C_k\neq\emptyset}$; this is known as Helly's lemma.

Consider now two convex subsets $C_1,C_2\cu X$ and simply write ${\pi_i\colon X\ra C_i}$ for the two gate-projections. We say that points $x_1\in C_1$ and ${x_2\in C_2}$ form a \emph{pair of gates} if $\pi_1(x_2)=x_1$ and $\pi_2(x_1)=x_2$. Equivalently, we have ${\mscr{W}(x_1|x_2)=\mscr{W}(C_1|C_2)}$, or $d(x_1,x_2)=d(C_1,C_2)$. 

We denote by $S=S(C_1,C_2)$ the \emph{restriction quotient} of $X$ (in the sense of p.~860 in \cite{CS}) associated to the set $\mscr{W}(C_1)\cap\mscr{W}(C_2)$. We say that $S$ is the \emph{(abstract) shore} of $C_1$ and $C_2$. Let $S_i\cu C_i$ be the subsets of points that are part of a pair of gates; we also refer to $S_1$ and $S_2$ as \emph{shores}. Finally, consider $B=B(C_1,C_2)=\hull(S_1\cup S_2)$, which we call the \emph{bridge}. The following appears e.g.\ in Section~2.G of \cite{CFI} or Section~2.2 of \cite{Fio3}.

\begin{lem}\label{basics on bridges}
Let $x_1\in S_1$, $x_2=\pi_2(x_1)$ and $I=I(x_1,x_2)$. 
\begin{enumerate}
\item The bridge $B$ splits as a product $S\x I$, corresponding to the decomposition:
\[\mscr{W}(B)=(\mscr{W}(C_1)\cap\mscr{W}(C_2))\sqcup\mscr{W}(C_1|C_2).\] 
The projections to the two factors are, respectively, the restriction of the quotient projection $X\ra S$ and the gate-projection $\pi_I$.
\item We have $B\cap C_i=S_i$ and the projection $B\ra S$ restricts to isomorphisms $S\simeq S_i$. Furthermore, $S_1=\pi_1(C_2)$ and $S_2=\pi_2(C_1)$.
\end{enumerate}
\end{lem}

We say that two disjoint halfspaces $\mf{h},\mf{k}\in\mscr{H}$ are \emph{strongly separated} if the shore $S(\mf{h},\mf{k})$ consists of a single point. Equivalently, no hyperplane is transverse to both $\mf{h}$ and $\mf{k}$. In this case, we also say that the hyperplanes bounding $\mf{h}$ and $\mf{k}$ are \emph{strongly separated}.

If $\mf{h}$ and $\mf{k}$ are strongly separated, there exists a unique point $v\in\mf{h}\cap B(\mf{h},\mf{k})$. Given points $x\in\overline{\mf{h}}$ and $y\in\overline{\mf{k}}$, we always have $m(x,v,y)=v$ and $v\in I(x,y)$.

We have already observed that the combinatorial metric on $X$ can be expressed as $d(x,y)=\#\mscr{W}(x|y)$ for any two vertices $x,y\in X$. When $x$ and $y$ are arbitrary points of $\overline X$, the same expression yields a distance function ${d\colon\overline X\x\overline X\ra\N\cup\{+\infty\}}$. The latter satisfies all the axioms of a metric, except that we now allow infinite values. Given points $x,y\in\overline X$ and $v\in X$, we define the \emph{Gromov product} 
\[(x\cdot y)_v:=d\big(v,m(v,x,y)\big)=\#\mscr{W}(v|x,y).\]
Note that this coincides with the usual Gromov product 
\[(x\cdot y)_v=\tfrac{1}{2}\cdot\big[d(v,x)+d(v,y)-d(x,y)\big],\] 
whenever the latter is defined.

\begin{lem}\label{infinite Gromov product}
Consider $x,y,z\in\overline X$ and $v\in X$.
\begin{enumerate}
\item We have $m(x,y,z)\in X$ if and only if each of the three intervals $I(x,y)$, $I(y,z)$, $I(z,x)$ intersects $X$.
\item We have $(x\cdot y)_v<+\infty$ if and only if $I(x,y)$ intersects $X$.
\end{enumerate}
\end{lem}
\begin{proof}
The medians $m(x,y,z)$ and $m(x,y,v)$ lie in $I(x,y)$. If $I(x,y)\cu\partial X$, we have $m(x,y,z)\in\partial X$, $m(x,y,v)\in\partial X$ and $(x\cdot y)_v=+\infty$. This shows one implication of both $(1)$ and $(2)$.

If $I(x,y)$, $I(y,z)$, $I(z,x)$ all intersect $X$, Helly's Lemma guarantees that $I(x,y)\cap I(y,z)\cap I(z,x)\cap X$ is nonempty, i.e.\ $m(x,y,z)$ lies in $X$. Applying this observation to the case $z=v$, we see that $I(x,y)\cap X\neq\emptyset$ is enough to guarantee that $m(x,y,v)\in X$; hence $(x\cdot y)_v<+\infty$.
\end{proof}

\begin{lem}\label{GP continuous}
Let $X$ be locally finite. Fixing a vertex $v\in X$, the function $(-\cdot -)_v\colon\overline X^2\ra\N\cup\{+\infty\}$ is continuous.
\end{lem}
\begin{proof}
Consider points $x,y\in\overline X$. Given any finite subset $F\cu\mscr{H}(v|x,y)$, let $U\cu\overline X$ be the set of points represented by ultrafilters containing $F$. Note that $U$ is open, that $x$ and $y$ lie in $U$ and that $(x'\cdot y')_v\geq\# F$ for all $x',y'\in U$. In particular, if $(x\cdot y)_v=+\infty$, this shows that the Gromov product is continuous at the pair $(x,y)$.

Suppose instead that $(x\cdot y)_v<+\infty$; the median $m=m(v,x,y)$ lies in $X$ by Lemma~\ref{infinite Gromov product}. Let $V_x$ denote the set of points $x'\in\overline X$ that satisfy $\mscr{W}(v|x')\cap\mscr{W}_m=\mscr{W}(v|x)\cap\mscr{W}_m$. Similarly, we define $V_y$ as the set of points $y'\in\overline X$ with $\mscr{W}(v|y')\cap\mscr{W}_m=\mscr{W}(v|y)\cap\mscr{W}_m$. Note that $x\in V_x$ and $y\in V_y$.

Since $X$ is locally finite, the set $\mscr{W}_m$ is finite; hence $V_x$ and $V_y$ are open in $\overline X$. Given $(x',y')\in V_x\x V_y$, no element of $\mscr{W}_m$ separates $x$ and $x'$ or $y$ and $y'$. It follows that no element of $\mscr{W}_m$ separates $m$ and $m(v,x',y')$; hence $m(v,x',y')=m$ and $(x'\cdot y')_v=(x\cdot y)_v$.
\end{proof}

We conclude this subsection introducing two notions that will be crucial to the rest of the paper, especially in Section~\ref{from moeb to cubical iso}. For the time being, however, we will not discuss them any further.

\begin{defn}
A subset $\mc{A}\cu\partial X$ is \emph{concise} if $I(x,y)\cap X\neq\emptyset$ for all distinct points $x,y\in \mc{A}$.
\end{defn}

\begin{defn}
A subset $\mc{A}\cu\partial X$ is said to be \emph{ample} if, for every $k\geq 1$ and every choice of pairwise transverse halfspaces $\mf{h}_1,...,\mf{h}_k$, the intersection $\mc{A}\cap\overline{\mf{h}_1}\cap...\cap\overline{\mf{h}_k}$ is nonempty.
\end{defn}

\subsection{Automorphisms of cube complexes.}\label{automorphisms section}

Let $X$ be a $\CAT$ cube complex. An \emph{axis} for $g\in\Aut(X)$ is a $\langle g\rangle$--invariant (combinatorial) geodesic $\g\cu X$ none of whose points is fixed by $g$. The \emph{minimal set} ${\rm Min}_X(g)$ is defined as either the union of all axes of $g$, or the set of vertices that are fixed by $g$; we also write simply ${\rm Min}(g)$ when this causes no confusion.

We say that $g$ acts \emph{stably without inversions} (Definition~4.1 in \cite{Haglund}) if no power of $g$ preserves a hyperplane and swaps its two sides. Note that every $g\in\Aut(X)$ acts stably without inversions on the barycentric subdivision $X'$ (Lemma~4.2 in \emph{op.\ cit.}). We say that $g$ acts \emph{non-transversely} (cf.\ Definition~3.15 in \cite{FFT}) if, for every ${\mf{w}\in\mscr{W}({\rm Min}(g))}$, the hyperplanes $g\mf{w}$ and $\mf{w}$ are not transverse. The notation $\ell_X$ was defined in the introduction. 

\begin{prop}\label{Haglund FFT}
Let $g\in\Aut(X)$ act stably without inversions.
\begin{enumerate}
\item We have ${\rm Min}(g)\neq\emptyset$. Thus, either $g$ fixes a vertex and $\ell_X(g)=0$, or $g$ admits an axis and $\ell_X(g)>0$. 
\item We have ${\rm Min}(g)\cu {\rm Min}(g^n)$ and $\ell_X(g^n)=|n|\cdot\ell_X(g)$ for all $n\in \Z$.
\item A vertex $v\in X$ satisfies $d(v,gv)=\ell_X(g)$ if and only if $v\in{\rm Min}(g)$.
\item If $g$ acts non-transversely, ${\rm Min}(g)$ is a convex subcomplex of $X$.
\item If $\dim X=D$, the automorphism $g^{D!}$ acts non-transversely.
\end{enumerate}
\end{prop}
\begin{proof}
Part~(1)--(3) follow from Corollary~6.2 in \cite{Haglund}. By Lemma~3.10 and Proposition~3.17 in \cite{FFT}, it suffices to prove part~(4) in the case when $g$ fixes a vertex of $X$ (indeed, in the notation of \emph{op.\ cit.}, this also implies that $X_g^{\rm fix}$ is a convex subcomplex of $X_g^{\rm ell}$).

Let us then assume that $g$ fixes vertices $v,w\in X$. Given $\mf{w}\in\mscr{W}(v|w)$, we have $g\mf{w}\in\mscr{W}(gv|gw)=\mscr{W}(v|w)$ and $d(v,\mf{w})=d(v,g\mf{w})$. It follows that $\mf{w}$ and $g\mf{w}$ are either transverse or equal. Since $g$ acts non-transversely, it must fix every element of $\mscr{W}(v|w)$. It follows that $I(v,w)\cu{\rm Min}(g)$.

Finally, we prove part~(5). Given $\mf{w}\in\mscr{W}(X)$, the hyperplanes $g^i\mf{w}$ cannot be pairwise transverse for $0\leq i\leq D$; pick $i$ so that $\mf{w}$ and $g^i\mf{w}$ are not transverse. If $\mf{w}$ crosses an axis of $g$, we can then choose a side $\mf{h}$ of $\mf{w}$ so that $g^i\mf{h}\subsetneq\mf{h}$. Since $i$ divides $D!$, we have $g^{D!}\mf{h}\subsetneq\mf{h}$. This shows that $g^{D!}\mf{w}$ and $\mf{w}$ are disjoint whenever $\mf{w}$ crosses an axis of $g$. 

The remaining elements of $\mscr{W}({\rm Min}(g))$ must separate two axes of $g$, or two fixed vertices. It suffices to consider the latter case, as, in the notation of \cite{FFT}, we can always pass to $X_g^{\rm ell}$. Now, if $v,w\in X$ are fixed by $g$ and $n\in\N$, there are at most $D$ hyperplanes of $\mscr{W}(v|w)$ at distance exactly $n+\frac{1}{2}$ from $v$; indeed, they are pairwise transverse. Since, moreover, these hyperplanes are permuted by $g$, we conclude that $g^{D!}$ fixes every element of $\mscr{W}(v|w)$. Hence $g^{D!}$ acts non-transversely.
\end{proof}

Let now $\G\acts X$ or $G\acts X$ denote an action by cubical automorphisms. Throughout the paper, we prefer to denote groups by $G$ if they act cocompactly and by $\G$ otherwise. We hope this will improve the exposition, as we will sometimes be in the presence of two actions: one of a group $\G$ --- of which we study the length spectrum --- and one of a group $G$, which acts cocompactly but otherwise bears little relevance to us.

The following is an iterated application of Exercise~1.6 in \cite{Sageev}. See e.g.\ Lemma~2.3 in \cite{FH} for a proof.

\begin{lem}\label{cocompact stabilisers}
Let $G\acts X$ be cocompact and let $\mf{w}_1,...,\mf{w}_k$ be pairwise transverse hyperplanes. Denote by $G_i$ the stabiliser of $\mf{w}_i$ in $G$. The action $G_1\cap ...\cap G_k\acts\mf{w}_1\cap ...\cap\mf{w}_k$ is cocompact.
\end{lem}

We say that $X$ is \emph{irreducible} if $X$ does not split nontrivially as a product of cube complexes. The action $\G\acts X$ is \emph{non-elementary} if there are no finite orbits in $X\cup\partial_{\infty}X$; we recall that $\partial_{\infty}X$ denotes the \emph{visual} boundary. 

The cube complex $X$ is \emph{essential} if no halfspace is contained in a metric neighbourhood of the corresponding hyperplane. The action $\G\acts X$ is \emph{essential} if no $\G$--orbit is contained in a metric neighbourhood of a halfspace. Thus, a cocompact action $G\acts X$ is essential if and only if $X$ is essential. 

We say that an element $g\in \G$ \emph{flips} a halfspace $\mf{h}\in\mscr{H}(X)$ if we have $\mf{h}^*\cap g\mf{h}^*=\emptyset$. The halfspace $\mf{h}$ is \emph{($\G$--)flippable} if such a group element exists and \emph{($\G$--)unflippable} otherwise.

\begin{lem}\label{geometric implies nonelementary}
Let $X$ be irreducible and essential and let $G$ be a discrete, non-virtually-cyclic group. Every proper cocompact action $G\acts X$ is non-elementary.
\end{lem}
\begin{proof}
If $G$ had a finite orbit within $X$, the properness assumption would force $G$ to be finite. Suppose for the sake of contradiction that $G$ has a finite orbit in $\partial_{\infty}X$. Up to passing to a finite index subgroup, we can assume that $G$ fixes a point $\xi\in\partial_{\infty}X$; let $r$ be a $\CAT$--ray representing $\xi$.

 Let $\mf{w}$ be a hyperplane crossed by $r$ at a point $p$. Observe that $d(r(t),\mf{w})$ diverges for $t\ra+\infty$. Otherwise $r$ would be asymptotic to a different $\CAT$--ray, based at $p$ and entirely contained in $\mf{w}$, which is forbidden by the $\CAT$ condition. We conclude that the side of $\mf{w}$ that contains $r(0)$ is unflippable. On the other hand, Theorem~4.7 in \cite{CS} shows that every element of $\mscr{H}(X)$ is $G$--flippable, a contradiction.
\end{proof}

We say that the action $\G\acts X$ is \emph{$\ell^1$--minimal} if there does not exist a proper, convex, $\G$--invariant subcomplex $Y\cu X'$ (recall that $X'$ is the barycentric subdivision). This notion is closely related to essentiality, as the next result shows.

\begin{lem}\label{minimal vs essential}
Every essential action $\G\acts X$ is $\ell^1$--minimal. The converse holds if $X$ is finite dimensional and $\G\acts X$ is non-elementary.
\end{lem}
\begin{proof} 
The two implications follow, respectively, from Lemma~3.1 and Proposition~3.5 in \cite{CS}. Note that the statement of \cite[Proposition~3.5]{CS} requires the unstated assumption that the action be without inversions\footnote{Otherwise, extend an arbitrary proper cocompact action $\G\acts X$ to $Y=X\x[0,1]$ so that it flips the new hyperplane according to a homomorphism $\G\ra\Z/2\Z$. Then the $\G$--essential core of $Y$ does not embed $\G$--equivariantly in $Y$, though it does in $Y'$.} (this is why we used barycentric subdivisions when defining $\ell^1$--minimality).
\end{proof}

We say that $X$ is \emph{hyperplane-essential} (cf.\ \cite{Hagen-Touikan}) if all its hyperplanes are essential when viewed as lower dimensional $\CAT$ cube complexes. An action $\G\acts X$ is \emph{hyperplane-essential} if each hyperplane-stabiliser acts essentially on the corresponding hyperplane. By Lemma~\ref{cocompact stabilisers}, a cocompact action $\G\acts X$ is hyperplane-essential if and only if $X$ is hyperplane-essential.

We say that $X$ is \emph{cocompact} if the action $\Aut(X)\acts X$ is cocompact. The following is Proposition~1 in \cite{Hagen-Tun}.

\begin{prop}\label{NS corner}
Let $X$ be cocompact, locally finite, essential, hyperplane-essential and irreducible. For any two transverse halfspaces $\mf{h}_1$ and $\mf{h}_2$, there exists a halfspace ${\mf{k}\cu\mf{h}_1\cap\mf{h}_2}$.
\end{prop}

A \emph{free face} in $X$ is a non-maximal cube $c\cu X$ that is contained in a unique maximal cube. When $X$ has no free faces, Proposition~\ref{my corner} below provides a stronger version of the previous result.

\begin{rmk}\label{nff hyp rmk}
Hyperplanes of cube complexes with no free faces are again cube complexes with no free faces. 
\end{rmk}

\begin{defn}\label{sector-heavy defn}
We say that $X$ is \emph{sector-heavy} if, for every $k\geq 2$ and every choice of pairwise transverse halfspaces $\mf{h}_1,...,\mf{h}_k$, there exists a halfspace $\mf{k}\cu\mf{h}_1\cap...\cap\mf{h}_k$.
\end{defn} 

\begin{prop}\label{my corner}
If $X$ is cocompact, locally finite, irreducible and has no free faces, then $X$ is sector-heavy.
\end{prop}
\begin{proof}
Since $X$ is finite dimensional and without free faces, $X$ is essential. By a repeated application of Remark~\ref{nff hyp rmk}, any intersection of pairwise-transverse hyperplanes of $X$ is itself a cube complex with no free faces. We conclude that $X$ is hyperplane-essential, and the case $k=2$ follows from Proposition~\ref{NS corner}. For the general case, we proceed by induction on $k\geq 3$.

Let $\mf{h}_1,...,\mf{h}_k$ be pairwise transverse halfspaces. By the inductive hypothesis, there exists a halfspace $\mf{j}\cu\mf{h}_1\cap ...\cap\mf{h}_{k-1}$; let $\mf{w}$ be the hyperplane corresponding to $\mf{j}$. If $\mf{w}\cu\mf{h}_k$, either we have $\mf{j}\cu\mf{h}_k$ and we are done, or $\mf{h}_k^*\cu\mf{j}$; in the latter case, $\mf{j}$ intersects $\mf{h}_1^*$, a contradiction. If $\mf{w}$ is transverse to $\mf{h}_k$, we also immediately conclude by invoking Proposition~\ref{NS corner}, which yields a halfspace $\mf{k}\cu\mf{j}\cap\mf{h}_k$. We are left to consider the case when $\mf{w}\cu\mf{h}_k^*$. 

Let $G<\Aut(X)$ be the subgroup that preserves each of the halfspaces $\mf{h}_1,...,\mf{h}_{k-1}$ and denote by $\mf{w}_i$ the hyperplane corresponding to $\mf{h}_i$. Since $X$ has no free faces, the cube complex $\mf{w}_1\cap ...\cap\mf{w}_{k-1}$ is essential. The action $G\acts\mf{w}_1\cap ...\cap\mf{w}_{k-1}$ is cocompact by Lemma~\ref{cocompact stabilisers} and therefore essential. Now, applying Proposition~3.2 in \cite{CS} to the hyperplane $\mf{w}_1\cap ...\cap\mf{w}_k$ of the cube complex $\mf{w}_1\cap ...\cap\mf{w}_{k-1}$, we obtain an element $g\in G$ with $\mf{h}_k\subsetneq g\mf{h}_k$. For large $n\geq 1$, either $\mf{w}$ is transverse to $g^n\mf{h}_k$ or $\mf{w}\cu g^n\mf{h}_k$. The discussion above then shows that, in both cases, there exists a halfspace 
\[\mf{j}'\cu\mf{h}_1\cap ...\cap\mf{h}_{k-1}\cap g^n\mf{h}_k=g^n(\mf{h}_1\cap ...\cap\mf{h}_k).\] 
We conclude by taking $\mf{k}=g^{-n}\mf{j}'$.
\end{proof}

\subsection{Regular points.}

Throughout this subsection:

\begin{ass}
Let the $\CAT$ cube complex $X$ be irreducible. 
\end{ass}

The following notion was first introduced in \cite{Fernos,Fernos-Lecureux-Matheus}. See Proposition~7.5 in \cite{Fernos} for the formulation presented here.

\begin{defn}\label{regular defn}
A point $x\in\partial X$ is \emph{regular} if the ultrafilter $\s_x$ contains an infinite chain $\mf{h}_0\supsetneq\mf{h}_1\supsetneq ...$ such that $\mf{h}_n^*$ and $\mf{h}_{n+1}$ are strongly separated for every $n\geq 0$. We refer to the latter as a \emph{strongly separated chain}. We denote by $\partial_{\rm reg}X\cu\partial X$ the subset of regular points.
\end{defn}

We record a few observations for later use. 

\begin{lem}[Lemma~5.12 in \cite{Fernos-Lecureux-Matheus}]\label{single point in ss chain}
Let $\mf{h}_0\supsetneq\mf{h}_1\supsetneq ...$ be a strongly separated chain. There exists a unique point $x\in\overline X$ that lies in each $\mf{h}_n$. Moreover, $x$ is regular.
\end{lem}

\begin{lem}\label{disjoining sets}
Let $\G\acts X$ be essential and non-elementary. For every finite subset $\mc{F}\cu\partial_{\rm reg}X$, there exists $g\in \G$ with $g\mc{F}\cap\mc{F}=\emptyset$.
\end{lem}
\begin{proof}
The orbit $\G\cdot x$ is infinite for every $x\in\mc{F}$; this follows for instance from Proposition~5.2 in \cite{Fio2} and the fact that $\G$ has no finite orbits in the visual boundary of $X$. Lemma~\ref{single point in ss chain} shows that any strongly separated chain containing a point $y\in (\G\cdot x)\setminus\mc{F}$ must include a halfspace $\mf{h}$ with $\mc{F}\cu\mf{h}^*$. The Flipping Lemma of \cite{CS} now yields $g\in \G$ with $g\mf{h}^*\cap\mf{h}^*=\emptyset$, so in particular $g\mc{F}$ and $\mc{F}$ are disjoint.
\end{proof}

\begin{rmk}\label{regular is concise}
Lemma~5.14 in \cite{Fernos-Lecureux-Matheus} shows that the median $m(x,y,z)$ lies in $X$ whenever the three points $z\in\overline X$ and $x,y\in\partial_{\rm reg} X$ are pairwise distinct. In particular, $\partial_{\rm reg}X$ is a concise subset of $\partial X$.
\end{rmk}

\begin{lem}\label{ample sets}
Let $\G\acts X$ be an essential non-elementary action. Consider a nonempty, $\G$--invariant subset $\mc{A}\cu\partial_{\rm reg}X$.
\begin{enumerate}
\item For every halfspace $\mf{h}$, we have $\mf{h}\cap\mc{A}\neq\emptyset$.
\item If $X$ is sector-heavy, the set $\mc{A}$ is ample.
\end{enumerate}
\end{lem}
\begin{proof}
Part~(2) immediately follows from part~(1), so let us only prove the latter. Pick $x\in\mc{A}$ and a strongly separated chain $\mf{k}_0\supsetneq\mf{k}_1\supsetneq ...$ containing $x$. The hyperplane bounding $\mf{h}$ must be contained in $\mf{k}_n^*$ for all sufficiently large values of $n$. We thus have either $x\in\mf{k}_n\cu\mf{h}$ or $\mf{k}_n\cap\mf{h}=\emptyset$ for some $n\geq 0$. In the latter case, the Flipping and Double Skewering Lemmas of \cite{CS} yield an element $g\in \G$ with $g\mf{k}_n\cu\mf{h}$. In particular, $gx\in\mf{h}\cap \mc{A}$.
\end{proof}

\subsection{$\CAT$ cuboid complexes.}\label{cuboid section}

A \emph{cuboid complex} is made up of Euclidean cuboids $[0,d_1]\x ...\x [0,d_k]$ (aka \emph{orthotopes}, for instance in \cite{BCV}) with arbitrary edge lengths, rather than unit Euclidean cubes $[0,1]^k$. Glueing maps are still required to be isometries of faces. In order to make things slightly more precise, we consider the following construction.

Let $X$ be a $\CAT$ cube complex. Every function $\mu\colon\mscr{W}(X)\ra\R_{>0}$ determines a \emph{weighted combinatorial metric} $d_\mu$ on $X$. This is given by:
\[d_{\mu}(v,w)=\sum_{\mf{w}\in\mscr{W}(v|w)}\mu(\mf{w})\]
for all vertices $v,w\in X$. The usual combinatorial metric $d$ then arises from the function $\mu_1$ that assigns the value $+1$ to each hyperplane.

\begin{defn}
A \emph{$\CAT$ cuboid complex} $\mbb{X}$ is any metric cell complex $(X,d_{\mu})$ arising from this construction.
\end{defn}

Whenever dealing with cuboid complexes --- that is, mostly here and in Sections~\ref{ext partial iso sect} and~\ref{cptf sect} --- we forsake our custom of implicitly requiring points to be vertices. Note that the combinatorial metric $d_{\mu}$ can be extended outside of the $0$--skeleton, so that every closed $k$--cell in $\mbb{X}$ is isometric to a cuboid $[0,d_1]\x ...\x [0,d_k]$ endowed with the restriction of the $\ell^1$ metric on $\R^k$. It is useful to observe that this gives $\mbb{X}$ a structure of \emph{median space} \cite{Fio1}.

\begin{rmk}
As the name suggests, a canonical $\CAT$ metric can also be constructed on $\mbb{X}$. It suffices to instead endow each cuboid with its Euclidean metric and then consider the induced path metric. This metric, however, will be of no use to us in this paper.
\end{rmk}

We say that $\CAT$ cuboid complexes $\mbb{X}=(X,d_{\mu})$ and $\mbb{Y}=(Y,d_{\nu})$ are \emph{isomorphic} if there exists an isometric cellular isomorphism $f\colon\mbb{X}\ra\mbb{Y}$. In other words, $f\colon X\ra Y$ is an isomorphism of $\CAT$ cube complexes inducing a map $f_*\colon\mscr{W}(X)\ra\mscr{W}(Y)$ such that $\mu=\nu\o f_*$. 

Note however that there can be isometries $\mbb{X}\ra\mbb{Y}$ that do not preserve the cellular structures. For instance, consider the cuboid complex $\mbb{X}'$ arising from the barycentric subdivision $X'$; edges of $X'$ are assigned half the length of the corresponding edges of $X$. The identity map $\mbb{X}\ra\mbb{X}'$ is then a (surjective) isometry, but never an isomorphism.

Let $\Aut(\mbb{X})\leq\Aut(X)$ denote the group of automorphisms of $\mbb{X}$. All group actions $\G\acts\mbb{X}$ will be assumed to be by automorphisms, i.e.\ to arise from a homomorphism $\G\ra\Aut(\mbb{X})$. This is the same as specifying an action $\G\acts X$ leaving the function $\mu$ invariant.

All results in the present paper equally apply to $\CAT$ cuboid complexes $\mbb{X}$ and actions on $\mbb{X}$ by automorphisms. With the exception of the discussion in Section~\ref{ext partial iso sect}, all proofs immediately generalise to the cuboidal context simply by performing the following adaptations.
\begin{enumerate}
\item Given a subset $\mc{U}\cu\mscr{W}(\mbb{X})=\mscr{W}(X)$, the \emph{cardinality} $\#\mc{U}$ should always be replaced by the \emph{weight} $\sum_{\mf{w}\in\mc{U}}\mu(\mf{w})$.
\item The cross ratio $\Cr$ does not take values in $\Z\cup\{\pm\infty\}$, but rather in $M\cup\{\pm\infty\}$, where $M$ is the $\Z$--module generated by the image of the map $\mu$. Similar observations apply to $\crt$ (Definition~\ref{crt defn}) and length functions.
\end{enumerate}

The proofs of Theorem~\ref{ext partial iso thm} and Proposition~\ref{cptf main} in Sections~\ref{ext partial iso sect} and~\ref{cptf sect}, respectively, will require a few more notions, which we introduce here. 

Every edge $e\cu \mbb{X}$ gives rise to a gate-projection $\pi_e\colon \mbb{X}\ra e$. We stress that $\pi_e$ is now defined on the entire $\mbb{X}$ and not just on its $0$--skeleton. A \emph{median halfspace} of $\mbb{X}$ is any subset $\mf{h}\cu \mbb{X}$ arising as $\mf{h}=\pi_e^{-1}(J)$ for some edge $e\cu \mbb{X}$ and some open sub-interval $J\cu e$ containing exactly one endpoint of $e$. We say that $\mf{h}$ is of \emph{type~2} if $J$ is exactly $e$ minus one endpoint; otherwise, $\mf{h}$ is of \emph{type~1}. We denote the collection of all median halfspaces by $\mf{H}(\mbb{X})$.

Every median halfspace is an open, convex subset of $\mbb{X}$ with convex complement. Given $\mf{w}\in\mscr{W}(\mbb{X})$, edges crossing $\mf{w}$ all give rise to the same family of median halfspaces, which we denote by $\mf{H}(\mf{w})$. We say that these median halfspaces are \emph{subordinate to $\mf{w}$}. 

Note that $\mf{H}(\mf{w})$ is naturally in bijection with two copies of any edge transverse to $\mf{w}$. We can use this map to define a measure $\nu_{\mf{w}}$ on $\mf{H}(\mf{w})$ with total mass $2\cdot\mu(\mf{w})$. These measures piece together to yield a measure $\nu$ on the entire $\mf{H}(\mf{w})$. Given points $x,y\in\mbb{\mbb{X}}$, we then have $d_{\mu}(x,y)=\nu(\mf{H}(x|y))$, where $\mf{H}(x|y)$ denotes the set of median halfspaces containing $y$ but not $x$.

Given $x\in \mbb{X}$, we denote by $\mf{H}_x$ the collection of median halfspaces that contain $x$ in their frontier. If $x$ is a vertex, elements of $\mf{H}_x$ naturally correspond to edges incident to $x$; more precisely, for every $\mf{w}\in\mscr{W}_x$, there exists a unique element of $\mf{H}_x$ subordinate to $\mf{w}$ and it is of type~2. 

More generally, let $c\cu \mbb{X}$ be the cube containing $x$ in its interior. For every hyperplane $\mf{w}$ cutting the cube $c$, there are two disjoint elements of $\mf{H}_x$ subordinate to $\mf{w}$; they are both of type~1. For every other hyperplane $\mf{u}$ containing $c$ in its carrier, there exists exactly one element of $\mf{H}_x$ subordinate to $\mf{u}$; this is of type~2. Every element of $\mf{H}_x$ arises in one of these two ways.

Finally, let us consider a geodesic segment $\g\cu \mbb{X}$. We say that $\g$ is an \emph{edge parallel} if all median halfspaces containing exactly one endpoint of $\g$ are subordinate to the same hyperplane of $\mbb{X}$. Equivalently, $\g$ is contained in a cuboid $c$ and is obtained by letting only one coordinate of $c$ vary.

\addtocontents{toc}{\protect\setcounter{tocdepth}{2}}
\section{The cross ratio on the Roller boundary.} \label{crossratio section}

Let $\Om$ be a set and $\mscr{A}\cu\Om^4$ a subset closed under permutations of the four coordinates. We say that a map $\B\colon\mscr{A}\ra\R\cup\{\pm\infty\}$ is an \emph{(abstract) cross ratio} on $\Om$ if it satisfies the following conditions for all $4$--tuples $(x,y,z,w)$, $(x,y,z,t)$ and $(x,y,t,w)$ lying in $\mscr{A}$:
\begin{enumerate}
\item[(i)] $\B(x,y,z,w)=-\B(y,x,z,w)$;
\item[(ii)] $\B(x,y,z,w)=\B(z,w,x,y)$;
\item[(iii)] $\B(x,y,z,w)=\B(x,y,z,t)+\B(x,y,t,w)$;
\item[(iv)] $\B(x,y,z,w)+\B(y,z,x,w)+\B(z,x,y,w)=0$.
\end{enumerate}
This definition should be compared e.g.\ with those in \cite{Otal-IberoAm,Hamenstaedt-ErgodTh,Labourie-IM}.

Let now $X$ be a $\CAT$ cube complex. Given a basepoint $v\in X$, we consider the subset $\mscr{A}\cu(\overline X)^4$ of $4$--tuples $(x,y,z,w)$ such that at most one of the three quantities $(x\cdot y)_v+(z\cdot w)_v$, $(x\cdot z)_v+(y\cdot w)_v$ and $(x\cdot w)_v+(y\cdot z)_v$ is infinite. By Lemma~\ref{infinite Gromov product}, the set $\mscr{A}$ does not depend on $v$ and, by Lemma~\ref{GP continuous}, it is an open subset of $(\overline X)^4$. Note moreover that, by Remark~\ref{regular is concise}, every $4$--tuple of pairwise distinct points of $\partial_{\rm reg}X$ lies in $\mscr{A}$.

It is easy to check that the map $\Cr_v\colon\mscr{A}\ra\Z\cup\{\pm\infty\}$ defined by 
\[\Cr_v(x,y,z,w)=(x\cdot z)_v+(y\cdot w)_v-(x\cdot w)_v-(y\cdot z)_v\]
is well-defined and a cross ratio on $\overline X$. When $x$, $y$, $z$ and $w$ lie in $X$, we have 
\[\Cr_v(x,y,z,w)=d(x,w)+d(y,z)-d(x,z)-d(y,w),\]
which shows independence of the choice of the basepoint $v$ in this case. In general, basepoint-independence follows from the following result, which we obtained in previous work with Incerti-Medici (Proposition~3.2 in \cite{BFI-new}).

\begin{prop}\label{independent of basepoint}
For every $v\in X$ and every $(x,y,z,w)\in\mscr{A}$, we have
\[\Cr_v(x,y,z,w)=\#\mscr{W}(x,z|y,w)-\#\mscr{W}(x,w|y,z).\]
\end{prop}

\begin{defn}
We will thus simply write $\Cr\colon\mscr{A}\ra\Z\cup\{\pm\infty\}$ from now on and refer to it as \emph{the cross ratio} on $\overline X$ (or $\partial X$).
\end{defn}

Despite taking values in a discrete set, the cross ratio $\Cr$ is continuous; note in this regard that $\overline X$ is totally disconnected. Endowing $\mscr{A}\cu (\overline X)^4$ with the subspace topology, this follows from Lemma~\ref{GP continuous}:

\begin{prop}\label{continuity}
If $X$ is locally finite, the cross ratio $\Cr$ is continuous.
\end{prop}

It is often useful to simultaneously record information about the cross ratios of all $4$--tuples obtained by permuting the coordinates of $(x,y,z,w)$. This is the purpose of Definition~\ref{crt defn} below. Note that identities~(i) and~(ii) imply that $|\Cr(x,y,z,w)|$ is invariant under a subgroup of order $8$ of the symmetric group on $4$ elements. Thus, we only need to record $24/8=3$ `meaningful' values for every subset $\{x,y,z,w\}$. These values are precisely the three cross ratios appearing in identity~(iv), so they are not independent. 

Given $a_1,b_1,c_1,a_2,b_2,c_2\in\N\cup\{+\infty\}$, we say that the triples $(a_1,b_1,c_1)$ and $(a_2,b_2,c_2)$ are \emph{equivalent} if there exists $n\geq 0$ such that 
\[a_i=a_j+n;~~~~b_i=b_j+n;~~~~c_i=c_j+n,\] 
where $\{i,j\}=\{1,2\}$. We denote the equivalence class of the triple $(a,b,c)$ by $\ll a:b:c\rr$. Note that $\ll+\infty:+\infty:+\infty\rr$ is the only class consisting of a single triple; every other equivalence class has a unique representative with at least one zero entry. We also remark that all the triples in a given equivalence class have the same infinite entries.

\begin{defn}\label{crt defn}
Given $x,y,z,w\in\overline X$ and $v\in X$, the \emph{cross ratio triple} $\crt_v(x,y,z,w)$ is the equivalence class
\[\Big\ll(x\cdot y)_v+(z\cdot w)_v:(x\cdot z)_v+(y\cdot w)_v:(x\cdot w)_v+(y\cdot z)_v\Big\rr.\]
This is the same as $\ll\#\mscr{W}(x,y|z,w):\#\mscr{W}(x,z|y,w):\#\mscr{W}(x,w|y,z)\rr$ when $(x,y,z,w)\in\mscr{A}$, as a consequence of Proposition~\ref{independent of basepoint}.
\end{defn}

Note that $\crt_v$ is always independent of the choice of $v$. This follows from Proposition~\ref{independent of basepoint} when $(x,y,z,w)\in\mscr{A}$ and is clear otherwise. We are therefore allowed to simply write $\crt$.

We remark that all entries of a cross ratio triple are nonnegative. Cross ratios are recovered by taking the difference of two entries of the triple. 

Let now $Y$ be another $\CAT$ cube complex. We write $\mscr{A}(X)$ rather than just $\mscr{A}$ when it is necessary to specify the cube complex under consideration. By analogy with the context of $\CA$ spaces, we give the following:

\begin{defn}
Let $\mc{A}\cu\partial X$ and $\mc{B}\cu\partial Y$ be subsets. A map $f\colon \mc{A}\ra \mc{B}$ is \emph{M\"obius} if ${f^4(\mscr{A}(X)\cap \mc{A}^4)\cu\mscr{A}(Y)}$ and, for each $(x,y,z,w)\in\mscr{A}(X)\cap \mc{A}^4$, we have 
\[\Cr(f(x),f(y),f(z),f(w))=\Cr(x,y,z,w).\] 
The latter happens if and only if $\crt(f(x),f(y),f(z),f(w))=\crt(x,y,z,w)$ for all $4$--tuples ${(x,y,z,w)\in\mscr{A}(X)\cap \mc{A}^4}$.
\end{defn}

A bijection $f\colon \mc{A}\ra \mc{B}$ is M\"obius if and only if its inverse $f^{-1}\colon \mc{B}\ra \mc{A}$ is.

\section{From length spectra to M\"obius maps.}\label{MLSR section}

This section is entirely devoted to the proof of Theorem~\ref{from ell to moeb}. 

In Section~\ref{sc sect}, we introduce \emph{neatly contracting isometries} --- a particular class of elements in the automorphism group of an irreducible $\CAT$ cube complex. Their length functions are related to cross ratios via Theorem~\ref{cr from ell}, which will act as our bridge between length spectra and M\"obius maps. 

The latter maps will be constructed in Section~\ref{proving MLSR}, relying on a random-walk argument of \cite{Fernos-Lecureux-Matheus}, in order to show that neatly contracting isometries are always `sufficiently abundant'.

\subsection{Neatly contracting isometries.}\label{sc sect}

In this subsection:

\begin{ass}
Let $X$ be an irreducible, locally finite $\CAT$ cube complex. We allow $X$ to be infinite dimensional.
\end{ass}

\begin{defn}\label{neatly contracting defn}
A cubical automorphism $g\in\Aut(X)$ is \emph{neatly contracting} if there exist halfspaces $\mf{h}_1$ and $\mf{h}_2$ such that $g\mf{h}_1\cu\mf{h}_2\cu\mf{h}_1$ and both pairs $(\mf{h}_2,\mf{h}_1^*)$ and $(g\mf{h}_1,\mf{h}_2^*)$ are strongly separated.
\end{defn}

As motivation for our terminology, observe that neatly contracting automorphisms are contracting isometries (Lemma~6.2 in \cite{CS}). 

\begin{rmk}\label{SWI and NT}
Let $\mf{h}_1$, $\mf{h}_2$ and $g$ be as in Definition~\ref{neatly contracting defn}. Given any $\mf{w}\in\mscr{W}(X)$, strong separation yields $n\in\Z$ and $i\in\{1,2\}$ such that $\mf{w}\cu g^n\mf{h}_i\cap g^{n+1}\mf{h}_i^*$. In particular, for every $k\in\Z\setminus\{0\}$, the hyperplanes $\mf{w}$ and $g^k\mf{w}$ are distinct and not transverse. This shows that $g$ acts stably without inversions and non-transversely, as defined in Section~\ref{automorphisms section}.
\end{rmk}

\begin{prop}\label{g+-}
Let $g\in\Aut(X)$ be neatly contracting.
\begin{enumerate}
\item The automorphism $g$ fixes exactly two points $g^{\pm}\in\partial_{\rm reg}X$. 
\item If $x\in\overline{X}\setminus\{g^{\pm}\}$, we have $g^nx\ra g^+$ and $g^{-n}x\ra g^-$ as $n\ra+\infty$. 
\item If $\mf{j}_1,\mf{j}_2\in\mscr{H}(g^-|g^+)$, there exists $N$ such that $g^n\mf{j}_1\cu\mf{j}_2$ for $n\geq N$.
\end{enumerate}
\end{prop}
\begin{proof}
Let $\mf{h}_1$ and $\mf{h}_2$ be halfspaces as in Definition~\ref{neatly contracting defn}. By Lemma~\ref{single point in ss chain}, there exists a unique point $g^+$ lying in $g^n\mf{h}_1$ for all $n\geq 0$. Similarly, there exists a unique point $g^-$ lying in $g^n\mf{h}_1^*$ for all $n\leq 0$. It is clear that the points $g^{\pm}\in\partial_{\rm reg}X$ are fixed by $g$; their uniqueness will follow from part~$(2)$. Given $x\in\overline{X}\setminus\{g^{\pm}\}$, we must have $x\in g^n\mf{h}_1\cap g^{n+1}\mf{h}_1^*$ for some $n\in\Z$; hence part~$(2)$ in turn follows from part~$(3)$.

Let thus $\mf{w}_1$ and $\mf{w}_2$ be the hyperplanes determined by $\mf{j}_1$ and $\mf{j}_2$. By Remark~\ref{SWI and NT}, there exist integers $n_1,n_2\in\Z$ and indices $i_1,i_2\in\{1,2\}$ such that $\mf{w}_1\cu g^{n_1}\mf{h}_{i_1}\cap g^{n_1+1}\mf{h}_{i_1}^*$ and $\mf{w}_2\cu g^{n_2}\mf{h}_{i_2}\cap g^{n_2+1}\mf{h}_{i_2}^*$. Since $\mf{j}_1$ and $\mf{j}_2$ lie in $\mscr{H}(g^-|g^+)$, we have $\mf{j}_1\cu g^{n_1}\mf{h}_{i_1}$ and $g^{n_2+1}\mf{h}_{i_2}\cu\mf{j}_2$. Hence $g^n\mf{j}_1\cu\mf{j}_2$ for all $n\geq n_2-n_1+2$.
\end{proof}

The length function $\ell_X\colon\Aut(X)\ra\N$ was defined in the introduction.

\begin{prop}\label{axes}
Let $g\in\Aut(X)$ be neatly contracting. A vertex $v\in X$ lies in the interval $I(g^-,g^+)$ if and only if $d(v,gv)=\ell_X(g)$.
\end{prop}
\begin{proof}
By Remark~\ref{SWI and NT}, the automorphism $g$ acts stably without inversions and non-transversely. Proposition~\ref{Haglund FFT} then shows that $d(v,gv)=\ell_X(g)$ holds if and only if $v$ lies in the convex subcomplex ${\rm Min}(g)\cu X$.

If a halfspace $\mf{h}$ contains a semi-axis of $g$, then it also contains either $g^+$ or $g^-$. Hence $\mf{h}$ contains an entire axis of $g$ if and only if it contains both $g^-$ and $g^+$ and this happens if and only if $\mf{h}$ contains \emph{all} axes of $g$. On the one hand, the intersection of all such halfspaces must coincide with the convex set ${\rm Min}(g)$; on the other, it must be $\hull(\{g^{\pm}\})\cap X=I(g^-,g^+)\cap X$.
\end{proof}

\begin{prop}\label{g^nh^n}
Let $g$ and $h$ be neatly contracting automorphisms with $\#\{g^{\pm},h^{\pm}\}=4$. For all sufficiently large values of $n>0$, the automorphism $g^nh^n$ is neatly contracting. Furthermore, we have $(g^nh^n)^+\ra g^+$ and $(g^nh^n)^-\ra h^-$ for $n\ra +\infty$.
\end{prop}
\begin{proof}
Let $\mf{h}_0\supsetneq\mf{h}_1\supsetneq ... $ and $\mf{k}_0\supsetneq\mf{k}_1\supsetneq ... $ be strongly separated chains containing $g^+$ and $h^-$, respectively. Up to discarding finitely many halfspaces, we can assume that $\mf{h}_0\cu\mf{k}_0^*$ and that $g^-$ and $h^+$ lie in $\mf{h}_0^*\cap\mf{k}_0^*$. Since $h^+$ and $g^-$ are also contained in strongly separated chains, we can find halfspaces $\mf{j}\ni h^+$ and $\mf{m}\ni g^-$ such that $\mf{h}_0$, $\mf{k}_0$, $\mf{j}$ and $\mf{m}$ are pairwise disjoint.

Given $k\geq 0$, part~$(3)$ of Proposition~\ref{g+-} provides $n_k\geq 0$ such that $h^n\mf{k}_k^*\cu\mf{j}$ and $g^n\mf{m}^*\cu\mf{h}_k$ for all $n\geq n_k$. Observe that we then have:
\[g^nh^n\mf{k}_k^*\cu g^n\mf{j}\cu g^n\mf{m}^*\cu\mf{h}_k.\]
In particular, $g^nh^n$ is neatly contracting for all $n\geq n_1$. In light of the construction of $(g^nh^n)^{\pm}$ in Proposition~\ref{g+-}, this also shows that $(g^nh^n)^+$ lies in $\mf{h}_k$ and $(g^nh^n)^-$ lies in $\mf{k}_k$ for ${n\geq n_k}$. We conclude that $(g^nh^n)^+$ and $(g^nh^n)^-$ converge to $g^+$ and $h^-$, respectively, as $k$ goes to infinity.
\end{proof}

The following should be compared to Th\'eor\`eme~2.2 in \cite{Otal-IberoAm} and Theorem~1 in Section~5 of \cite{Kim-Top} --- both concerned with the context of simply connected, negatively curved Riemannian manifolds.

\begin{thm}\label{cr from ell}
Let $g$ and $h$ be neatly contracting automorphisms such that ${\#\{g^{\pm},h^{\pm}\}=4}$. Then there exists $N\geq 0$ such that every $n\geq N$ satisfies:
\[\Cr(g^-,h^-,g^+,h^+)=\ell_X(g^n)+\ell_X(h^n)-\ell_X(g^nh^n).\]
\end{thm}
\begin{proof}
Let $\mf{h}_0\supsetneq\mf{h}_1\supsetneq ... $ and $\mf{h}_0^*\supsetneq\mf{h}_{-1}^*\supsetneq ... $ be strongly separated chains containing $g^+$ and $g^-$, respectively. Similarly define halfspaces $\mf{k}_n$ for $h^{\pm}$. We can assume that $\mf{h}_0^*$, $\mf{h}_1$, $\mf{k}_0^*$ and $\mf{k}_1$ are pairwise disjoint. 

Given $k\geq 1$, Proposition~\ref{g^nh^n} yields $n_k\geq 0$ such that $(g^nh^n)^-\in\mf{k}_{-k-1}^*$ and $(g^nh^n)^+\in\mf{h}_1$ for all $n\geq n_k$. Enlarging $n_k$ if necessary, part~$(3)$ of Proposition~\ref{g+-} allows us to moreover assume that $g^n\mf{h}_{-k-1}\cu\mf{h}_k$ and $h^n\mf{k}_{-k}\cu\mf{k}_k$ for all $n\geq n_k$. Let $x_k\in\mf{h}_{-k}$ be the only point in the bridge $B(\mf{h}_{-k},\mf{h}^*_{-k-1})$ and $y_k\in\mf{k}_{-k}$ the only point in the bridge $B(\mf{k}_{-k},\mf{k}^*_{-k-1})$. 

For $n\geq n_k$, the points $h^-$ and $(g^nh^n)^-$ lie in $\mf{k}^*_{-k-1}$, while $h^+$ and $(g^nh^n)^+$ lie in $\mf{k}_{-k}$. Thus, the point $y_k$ belongs to $I(h^-,h^+)\cap I((g^nh^n)^-,(g^nh^n)^+)$ and similarly $x_k\in I(g^-,g^+)$. Proposition~\ref{axes} now implies that we have $\ell_X(g^nh^n)=d(y_k,g^nh^ny_k)$, $\ell_X(h^n)=d(y_k,h^ny_k)$ and $\ell_X(g^n)=d(x_k,g^nx_k)$. 

Observe that $y_k\in\mf{k}_{-k+1}^*\cu\mf{k}_0^*\cu\mf{h}_1^*\cu\mf{h}_k^*$; hence $g^{-n}y_k\in g^{-n}\mf{h}_k^*\cu\mf{h}_{-k-1}^*$ since we are assuming that $g^n\mf{h}_{-k-1}\cu\mf{h}_k$. On the other hand, the point $h^ny_k$ lies in $h^n\mf{k}_{-k}\cu\mf{k}_k\cu\mf{k}_1\cu\mf{h}_0\cu\mf{h}_{-k}$. Since $x_k$ is a gate for $\mf{h}_{-k}$ and $\mf{h}^*_{-k-1}$, we conclude that $m(g^{-n}y_k,x_k,h^ny_k)=x_k$ for all $n\geq n_k$. This implies that $m(y_k,g^nx_k,g^nh^ny_k)=g^nx_k$. Hence, for all $n\geq n_k$:
\[d(y_k,g^nh^ny_k)=d(y_k,g^nx_k)+d(g^nx_k,g^nh^ny_k)=d(y_k,g^nx_k)+d(x_k,h^ny_k)\]
and we have:
\begin{multline*}
\ell_X(g^n)+\ell_X(h^n)-\ell_X(g^nh^n)=d(x_k,g^nx_k)+d(y_k,h^ny_k)-d(y_k,g^nh^ny_k)= \\
=d(x_k,g^nx_k)+d(y_k,h^ny_k)-d(y_k,g^nx_k)-d(x_k,h^ny_k)= \\
=\Cr(x_k,y_k,g^nx_k,h^ny_k).
\end{multline*}
Proposition~\ref{continuity} shows that $\Cr(x_k,y_k,g^nx_k,h^ny_k)=\Cr(g^-,h^-,g^+,h^+)$ for all sufficiently large $k\geq 0$ and every $n\geq n_k$. It follows that every sufficiently large $n\geq 0$ satisfies ${\Cr(g^-,h^-,g^+,h^+)=\ell_X(g^n)+\ell_X(h^n)-\ell_X(g^nh^n)}$.
\end{proof}

\subsection{Constructing a M\"obius map.}\label{proving MLSR}

Throughout this subsection:

\begin{ass}
Let $X$ and $Y$ be locally finite, finite dimensional, irreducible $\CAT$ cube complexes. Let a group $\G$ act $\ell^1$--minimally and non-elementarily on $X$ and $Y$; by Lemma~\ref{minimal vs essential}, the actions are also essential. 
\end{ass}

We start by finding neatly contracting automorphisms.

\begin{lem}\label{cNT RW}
There exists $g\in \G$ that simultaneously acts as a neatly contracting automorphism of $X$ and $Y$.
\end{lem} 

Before proving the lemma, a note is needed in case $\G$ is uncountable.

\begin{rmk}\label{extracting a countable subgroup}
There exist countable subgroups $\G_X,\G_Y\leq\G$ so that the actions $\G_X\acts X$ and $\G_Y\acts Y$ are essential and non-elementary. We can thus replace $\G$ with the countable subgroup $\langle\G_X,\G_Y\rangle$.

As to the existence of $\G_X$, observe that $X$ has only countably many hyperplanes. Thus, essentiality of the action $\G\acts X$ is witnessed by countably many elements of $\G$, which generate a countable subgroup $\G_1\leq\G$ acting essentially on $X$. Furthermore, there exists a countable subgroup $\G_2\leq\G$ such that $\G_2\acts X$ is non-elementary. We can even take $\G_2$ to be a finitely generated free group and its existence follows e.g.\ from Propositions~6.4(2) and~5.2 in \cite{Fio2}. Finally, we set $\G_X=\langle\G_1,\G_2\rangle$.
\end{rmk}

\begin{proof}[Proof of Lemma~\ref{cNT RW}]
By Remark~\ref{extracting a countable subgroup}, we can assume that $\G$ is countable. Set $\Om=\G^{\N}$ and let $\mu$ be a probability measure on $\G$ whose support generates $\G$ as a semigroup. Denote by $Z_n(\om)$ the random walk on $\G$ induced by $\mu$, where $\om\in\Om$, and let $\mbb{P}$ be the corresponding measure on $\Om$. Lemmas~11.3 and~11.4 in \cite{Fernos-Lecureux-Matheus} and the proof of Theorem~11.5 in \emph{loc.\ cit.}\ show that 
\[\lim_{n\ra+\infty}\tfrac{1}{n}\cdot\#\{k\leq n\mid Z_k(\om) \text{ is neatly contracting in } X\}=1\]
for $\mbb{P}$--almost every $\om\in\Om$. Applying the same reasoning to the action $\G\acts Y$, we conclude that there exists an element $g\in \G$ that is neatly contracting in both $X$ and $Y$.
\end{proof}

We remark that it is possible to give an alternative proof of Lemma~\ref{cNT RW} by relying on Theorem~5.1 in \cite{Clay-Uyanik} as suggested in Theorem~6.67 of \cite{Genevois-surv}.

Let us now fix for the rest of the subsection an element $g\in \G$ that is simultaneously neatly contracting in $X$ and $Y$. We denote by $g_X^+\in\partial_{\rm reg}X$ and $g_Y^+\in\partial_{\rm reg}Y$ the points introduced in Proposition~\ref{g+-}. Observe that $hgh^{-1}$ is neatly contracting for each $h\in \G$ and $(hgh^{-1})^+_X=h\cdot g^+_X$. 

\begin{prop}\label{regular stabilisers}
Assume that the length functions $\ell_X$ and $\ell_Y$ coincide. An element $k\in \G$ fixes $g_X^+\in\partial_{\rm reg}X$ if and only if it fixes $g_Y^+\in\partial_{\rm reg}Y$.
\end{prop}
\begin{proof}
Consider the sets $U_1=\{g^{\pm}_X\}\cup\{k\cdot g^{\pm}_X\}$ and $U_2=\{g^{\pm}_Y\}\cup\{k\cdot g^{\pm}_Y\}$. Lemma~\ref{disjoining sets} guarantees the existence of $\g_1,\g_2\in \G$ with $\g_1U_1\cap U_1=\emptyset$ and $\g_2U_2\cap(U_2\cup\g_1^{-1}U_2)=\emptyset$. If $\g_2g_X^+\not\in U_1$, we set $\g=\g_2$, otherwise $\g=\g_1\g_2$. In any case, we have $\g g^+_X\not\in U_1$ and $\g g^+_Y\not\in U_2$.

The element $h_1=\g g\g^{-1}$ is neatly contracting in both $X$ and $Y$; moreover, $(h_1)^+_X\not\in U_1$ and $(h_1)^+_Y\not\in U_2$. Taking $h_2=h_1^ngh_1^{-n}$ for a sufficiently large integer $n\geq 0$, part~$(2)$ of Proposition~\ref{g+-} guarantees that $\{(h_2)_X^{\pm}\}\cap U_1=\emptyset$ and $\{(h_2)_Y^{\pm}\}\cap U_2=\emptyset$.

Now, we consider the elements $a_n=h_2^{-n}(kg^{-n}k^{-1})$ for all $n\geq 0$. Proposition~\ref{g^nh^n} shows that $(a_n)^+_X\ra (h_2)^-_X$ and $(a_n)^-_X\ra k\cdot g^+_X$. Proposition~\ref{continuity} then implies that
\[\Cr\left((a_n)^-_X,g^-_X,(a_n)^+_X,g^+_X\right)\longrightarrow\Cr\left(k\cdot g^+_X,g^-_X,(h_2)^-_X,g^+_X\right),\]
as the points $g^-_X$, $g^+_X$ and $(h_2)^-_X$ are pairwise distinct. Observe that the limit is $-\infty$ if and only if $k$ fixes $g^+_X$. The same observations hold within $Y$.

Our choice of $h_2$ implies that $(a_n)^+_X\neq (h_2)^-_X$ and ${(a_n)^-_X\neq k\cdot g^+_X}$ for all sufficiently large $n\geq 0$. In particular, the sequences $(a_n)^+_X$ and $(a_n)^-_X$ are not eventually constant and the four points $(a_n)^-_X$, $g^-_X$, $(a_n)^+_X$ and $g^+_X$ are pairwise distinct for all large $n\geq 0$. As the same holds in $Y$, Theorem~\ref{cr from ell} and the assumption on length functions show that 
\[\Cr\left((a_n)^-_X,g^-_X,(a_n)^+_X,g^+_X\right)=\Cr\left((a_n)^-_Y,g^-_Y,(a_n)^+_Y,g^+_Y\right)\]
for all large $n$. We conclude that the limits coincide:
\[\Cr\left(k\cdot g^+_X,g^-_X,(h_2)^-_X,g^+_X\right)=\Cr\left(k\cdot g^+_Y,g^-_Y,(h_2)^-_Y,g^+_Y\right).\]
Hence $k\cdot g^+_X=g^+_X$ if and only if $k\cdot g^+_Y=g^+_Y$.
\end{proof}

We now consider the orbits $\mc{A}=\G\cdot g^+_X\cu\partial_{\rm reg}X$ and $\mc{B}=\G\cdot g^+_Y\cu\partial_{\rm reg}Y$. By Proposition~\ref{regular stabilisers}, there exists a unique $\G$--equivariant bijection $\psi\colon \mc{A}\ra \mc{B}$ with $\psi(g^+_X)=g^+_Y$. Observe that every point of $\mc{A}$ is of the form $k^+_X$ for some element $k\in \G$ conjugate to $g$; moreover, $\psi(k^+_X)=k^+_Y$.

\begin{rmk}\label{additional procedure}
If there exist conjugates $a,b\in \G$ of $g$ such that $a^-_X=b^-_X$, but $a^+_X\neq b^+_X$, we replace $g$ with $g^{-1}$ in the construction of the set $\mc{A}$. This way:
\begin{itemize}
\item either $a^-_X\neq b^-_X$ for all conjugates $a,b$ of $g$ that satisfy $a^+_X\neq b^+_X$,
\item or, for every $x\in \mc{A}$, there exist conjugates $a,b\in \G$ of $g$ such that $x=a^+_X=b^+_X$ and $a^-_X\neq b^-_X$.
\end{itemize}
The need for this property will become apparent in the proof of Lemma~\ref{easy fix}.
\end{rmk}

\begin{lem}\label{easy fix}
If $x,y\in \mc{A}$ are distinct, there exist elements $g_n\in \G$ satisfying:
\[(g_n)^-_X\ra x,~~(g_n)^-_Y\ra\psi(x);~~~~~(g_n)^+_X\ra y,~~(g_n)^+_Y\ra\psi(y).\]
\end{lem}
\begin{proof}
Let $a,b\in \G$ be conjugates of $g$ such that $x=a^+_X$ and $y=b^+_X$; we also have $\psi(x)=a^+_Y$ and $\psi(y)=b^+_Y$. Replacing $a$ if necessary, Remark~\ref{additional procedure} enables us to ensure that $a^-_X\neq b^-_X$.

If $a^-_X=b^+_X$, we can take all $g_n$ to coincide with $a^{-1}$; similarly, if $b^-_X=a^+_X$, we set $g_n=b$. If none of these equalities is verified, we have $\#\{a^{\pm}_X,b^{\pm}_X\}=4$ and, by Proposition~\ref{g^nh^n}, it suffices to take $g_n=b^na^{-n}$. 
\end{proof}

\begin{prop}\label{psi preserves cr}
If $\ell_X$ and $\ell_Y$ coincide, the map $\psi$ is M\"obius.
\end{prop}
\begin{proof}
It suffices to show that $\psi$ preserves cross ratios of $4$--tuples $(x,y,z,w)$ with pairwise-distinct coordinates. All other $4$--tuples have cross ratio $0$ or $\pm\infty$, depending on which coordinates coincide. 

Lemma~\ref{easy fix} thus yields elements $g_n,h_n\in \G$ with:
\[(g_n)^-_X\ra x,~~(g_n)^-_Y\ra\psi(x);~~~~~(h_n)^-_X\ra y,~~(h_n)^-_Y\ra\psi(y);\]
\[(g_n)^+_X\ra z,~~(g_n)^+_Y\ra\psi(z);~~~~~(h_n)^+_X\ra w,~~(h_n)^+_Y\ra\psi(w).\]
Theorem~\ref{cr from ell} implies that
\[\Cr((g_n)_Y^-,(h_n)_Y^-,(g_n)_Y^+,(h_n)_Y^+)=\Cr((g_n)_X^-,(h_n)_X^-,(g_n)_X^+,(h_n)_X^+)\]
for all $n\geq 0$. Proposition~\ref{continuity} now shows that the two sides of the above equality converge to $\Cr(\psi(x),\psi(y),\psi(z),\psi(w))$ and $\Cr(x,y,z,w)$.
\end{proof}

This concludes the proof of Theorem~\ref{from ell to moeb}. 

\section{From M\"obius maps to cubical isomorphisms.}\label{from moeb to cubical iso}

Let $X$ and $Y$ be $\CAT$ cube complexes. This chapter is devoted to the proof of Theorem~\ref{NFF MLSR}. Having obtained Theorem~\ref{from ell to moeb} in Section~\ref{MLSR section}, it suffices to prove Theorem~\ref{ext Moeb} and this will be our sole focus.

The proof of Theorem~\ref{ext Moeb} will be divided into three steps, which will be carried out, respectively, in Sections~\ref{OR exist sect},~\ref{from OR sect} and~\ref{ext partial iso sect}. The last step, in Section~\ref{ext partial iso sect}, is actually the proof of Theorem~\ref{ext partial iso thm}. 

Section~\ref{general setup sect} will contain an overview of Sections~\ref{OR exist sect},~\ref{from OR sect} and~\ref{ext partial iso sect} and a complete proof of Theorem~\ref{ext Moeb} assuming those results. Before that, however, we need to introduce \emph{opposite} points in Section~\ref{opposite sect}.

\subsection{Opposite points.}\label{opposite sect}

Given points $x_1,x_2\in\partial X$, we consider the $\CAT$ cube complex $T=X\cap I(x_1,x_2)$.

Recall that $\mscr{W}_v$ denotes the set of hyperplanes adjacent to a vertex ${v\in X}$.

\begin{lem}\label{2-cuts}
The following conditions are equivalent for $v\in T$:
\begin{enumerate}
\item $\lk_T(v)$ has exactly two connected components; 
\item $\lk_T(v)$ is a union of two disjoint cliques;
\item $I(x_1,x_2)=I(x_1,v)\cup I(v,x_2)$.
\end{enumerate} 
In this case, the connected components of $\lk_T(v)$ are given by the sets of edges crossing the hyperplanes in $\mscr{W}(v|x_1)$ and $\mscr{W}(v|x_2)$.
\end{lem}
\begin{proof}
For each $i$, the hyperplanes in $\mscr{W}(v|x_i)\cap\mscr{W}_v$ are pairwise transverse and originate a clique in $\llk v$. As $\mscr{W}(T)=\mscr{W}(x_1|x_2)=\mscr{W}(x_1|v)\sqcup\mscr{W}(v|x_2)$, either $\lk_T(v)$ is connected, or $v$ satisfies $(1)$ and the two connected components of $\lk_T(v)$ are given precisely by $\mscr{W}(v|x_1)$ and $\mscr{W}(v|x_2)$. This shows the equivalence of $(1)$ and $(2)$. 

Regarding $(2)\Ra(3)$, if a point $w\in I(x_1,x_2)$ did not lie in the union $I(x_1,v)\cup I(v,x_2)$, the set $\mscr{W}(v|w)$ would intersect both $\mscr{W}(v|x_1)$ and $\mscr{W}(v|x_2)$. For $i=1,2$, pick hyperplanes $\mf{w}_i\in\mscr{W}(v|w)\cap\mscr{W}(v|x_i)$; choosing $\mf{w}_i$ closest to $v$, we can assume that $\mf{w}_i\in\mscr{W}_v$. Let $\mf{h}_i$ be the the side of $\mf{w}_i$ that contains $v$. Since $v$ satisfies condition~$(2)$, we must have $\mf{h}_1^*\cap\mf{h}_2^*=\emptyset$; on the other hand, $w$ lies in both $\mf{h}_1^*$ and $\mf{h}_2^*$. This is a contradiction.

Finally, if $\mf{w}_1\in\mscr{W}(v|x_1)$ and $\mf{w}_2\in\mscr{W}(v|x_2)$ are adjacent to $v$ and transverse, there exists a point $w\in I(x_1,x_2)$ with $\mscr{W}(v|w)=\{\mf{w}_1,\mf{w}_2\}$. In this case, we have $w\not\in I(x_1,v)\cup I(v,x_2)$. This shows $(3)\Ra (1)$. 
\end{proof}

We refer to a point $v\in T$ satisfying the equivalent conditions in Lemma~\ref{2-cuts} as a \emph{cut point} for $T$.

\begin{defn}\label{opposite defn}
Given $x,y,z\in\overline X$, we say that $x$ and $y$ are \emph{opposite} with respect to $z$ (written $x\op_z y$) if the point $m=m(x,y,z)$ lies in $X$ and it is a cut point for $I(x,y)\cap X$. Equivalently, $I(x,y)=I(x,m)\cup I(m,y)$.
\end{defn}

The terminology above is inspired by the case when $z$ coincides with the median $m(x,y,z)\in X$. We are however mostly interested in the general situation where the points $x$, $y$ and $z$ lie in $\partial X$. 

\begin{lem}\label{one is zero}
Consider points $x,y,z\in\overline X$ with $x\op_z y$ and $m=m(x,y,z)$. Then, given any $w\in\partial X$, we have either $(x\cdot w)_m=0$ or $(y\cdot w)_m=0$. 
\end{lem}
\begin{proof}
Since $I(x,y)=I(x,m)\cup I(m,y)$, the gate-projection $n:=m(x,y,w)$ falls either in $I(m,y)$ or in $I(x,m)$. If $n\in I(m,y)$, we have: 
\[\mscr{W}(w|m)=\mscr{W}(w|n)\sqcup\mscr{W}(n|m)\cu\mscr{W}(w|x,y)\cup\mscr{W}(w|m,x)\cu\mscr{W}(w|x).\] 
Hence $\mscr{W}(m|w,x)=\emptyset$ and $(x\cdot w)_m=0$. If instead if $n\in I(x,m)$, we similarly obtain $(y\cdot w)_m=0$. 
\end{proof}

As will become clear in Section~\ref{from OR sect}, opposite points are key to our proof of Theorem~\ref{ext Moeb}. The rest of this subsection is devoted to showing that M\"obius bijections take triples of opposite points to triples with the same property.

\begin{prop}\label{new op vs lcrt}
Let $X$ be locally finite. Consider an ample, concise subset $\mc{A}\cu\partial X$ and pair\-wise-distinct points $x_1,x_2,y\in \mc{A}$. The condition $x_1\op_y x_2$ fails if and only if there exists $z\in \mc{A}$ such that ${\crt(x_1,x_2,y,z)=\ll a:b:c\rr}$ with $a<\min\{b,c\}$ and $\max\{b,c\}<+\infty$.
\end{prop}
\begin{proof}
Since $\mc{A}$ is concise, the point $m=m(x_1,x_2,y)$ lies in $X$ by Lemma~\ref{infinite Gromov product}. We have:
\[\crt_m(x_1,x_2,y,z)=\big\ll (y\cdot z)_m:(x_2\cdot z)_m:(x_1\cdot z)_m\big\rr.\] 
If $x_1\op_y x_2$ and $z\in\partial X$, Lemma~\ref{one is zero} yields $\min\{(x_2 \cdot z)_m,(x_1 \cdot z)_m\}=0$ and we cannot have $(y\cdot z)_m<0$.

Consider instead the case when $x_1$ and $x_2$ are not opposite with respect to $y$. There exist transverse hyperplanes $\mf{w}_i\in\mscr{W}(m|x_i)\cap\mscr{W}_m$. Denoting by $m'\in X$ the point with ${\mscr{W}(m|m')=\{\mf{w}_1,\mf{w}_2\}}$, the hyperplanes $\mf{w}_1$ and $\mf{w}_2$ lie in $\mscr{W}(m'|y)\cap\mscr{W}_{m'}$. This set is finite, as $X$ is locally finite. Let $\mf{h}_1,...,\mf{h}_k$ be the sides containing $m'$ of the elements of $\mscr{W}(m'|y)\cap\mscr{W}_{m'}$.

As $\mc{A}$ is ample, there exists a point ${z\in \mc{A}\cap\mf{h}_1\cap...\cap\mf{h}_k}$. This property and Lemma~\ref{straight rays 2} imply that $(y\cdot z)_m=0$. The fact that $\mc{A}$ is concise guarantees that the Gromov products $(x_1\cdot z)_m$ and $(x_2\cdot z)_m$ are finite. Observing that $\mf{w}_i\in\mscr{W}(m|x_i,z)$, we moreover have $(x_i\cdot z)_m\geq 1$. Since ${\crt(x_1,x_2,y,z)=\ll 0:(x_2\cdot z)_m:(x_1\cdot z)_m\rr}$, this concludes the proof.
\end{proof} 

\begin{cor}\label{op-preserving}
Let $X$ and $Y$ be locally finite. Consider ample, concise subsets $\mc{A}\cu\partial X$ and $\mc{B}\cu\partial Y$ and a M\"obius bijection $f\colon\mc{A}\ra\mc{B}$. For all $x,y,z\in \mc{A}$, we have $x\op_z y\Leftrightarrow f(x)\op_{f(z)}f(y)$.
\end{cor}

\subsection{Proof outline.}\label{general setup sect}

We begin by introducing a few key notions.

\begin{defn}\label{SL defn}
A geodesic $\g\cu X$ is \emph{straight} if no two hyperplanes in $\mscr{W}(\g)$ are transverse. An oriented segment $\alpha\cu X$ is \emph{straight-looking (SL)} if no hyperplane in $\mscr{W}(\alpha)$ is transverse to the first hyperplane crossed by $\alpha$. A ray $r\cu X$ is \emph{regular} if $r^+\in\partial_{\rm reg}X$. Finally, a ray $r\cu X$ is \emph{straight-looking regular (SLR)} if it is straight-looking and regular. 
\end{defn}

Despite the many requirements, SLR rays always exist.

\begin{lem}\label{SLR rays exist}
If $\partial_{\rm reg}X\neq\emptyset$, there exists an SLR ray $r\cu X$.
\end{lem}
\begin{proof}
Pick a hyperplane $\mf{u}$ and a point $x\in\partial_{\rm reg}X$. Let $\overline x$ denote the gate-projection of $x$ to the carrier of $\mf{u}$; by regularity of $x$, the point $\overline x$ lies inside $X$. Let $w\in X$ be the only vertex with $\mscr{W}(w|\overline x)=\{\mf{u}\}$. Then any ray from $w$ to $x$ is SLR.
\end{proof}

To avoid confusion, we will speak of \emph{elements} of a link $\llk v$ with the implicit assumption that they are actually vertices of said link.

\begin{defn}\label{SLR-ext defn}
An oriented edge $e\cu X$ is \emph{SLR-extendable} if it is the initial edge of an SLR ray. Given a vertex $v\in X$, we say that an element $e\in\llk v$ is \emph{SLR-extendable} if the corresponding edge of $X$ is SLR-extendable (when oriented away from $v$).
\end{defn} 

\begin{defn}\label{OR defn}
A vertex $v\in X$ is an \emph{OR median} (OR stands for \emph{`opposite regular'}) if there exist $x,y,z\in\partial_{\rm reg}X$ with $m(x,y,z)=v$ and $x\op_zy$. We denote by $\mf{M}(X)\cu X$ the subset of OR medians.
\end{defn}

Not every vertex is an OR median in general, even when $\partial_{\rm reg}X$ is infinite. For instance, if $X$ is obtained by barycentrically subdividing a cube complex $X_0$, every OR median in $X$ is a vertex of $X_0$. In fact, it is not even clear whether `nicer' cube complexes should always contain such medians. 

The next lemma is our main source of OR medians. Recall from Section~\ref{ccc prelims} that $\llk v$ is endowed with a metric $\delta$ (possibly attaining the value $+\infty$).

\begin{lem}\label{easy OR}
If $\llk v$ contains pairwise-distinct, SLR-extendable elements $e_1,e_2,e_3$ with $\delta(e_1,e_2)\geq 2$, then $v$ is an OR median.
\end{lem}
\begin{proof}
Extend each $e_i$ to an SLR ray $r_i\cu X$. The points $r_i^+$ are all regular and, by Lemma~\ref{straight rays 2}, their median is $v$. Since $\delta(e_1,e_2)\geq 2$, the ball of radius $1$ around $v$ within $I(r_1^+,r_2^+)$ is precisely $e_1\cup e_2$. This implies that $v$ is a cut point for $I(r_1^+,r_2^+)$, hence $r_1\op_{r_3}r_2$.
\end{proof}

Theorem~\ref{ext Moeb} will be proved by piecing together three results. The first is the following, which will be obtained in Section~\ref{OR exist sect}.

\begin{thm}\label{OR exist thm}
Let $X$ be irreducible and with no free faces. Suppose that there exists a properly discontinuous, cocompact group action $G\acts X$. If $X$ is neither a single point nor $\R$, then $\mf{M}(X)\neq\emptyset$.
\end{thm}

Then, in Section~\ref{from OR sect}, we will prove:

\begin{thm}\label{from OR thm}
Let $X,Y$ be locally finite and $\mc{A}\cu\partial_{\rm reg}X$, $\mc{B}\cu\partial_{\rm reg}Y$ ample. 
\begin{enumerate}
\item Every M\"obius bijection ${f\colon \mc{A}\ra\mc{B}}$ determines a distance-preserving bijection $\phi\colon\mf{M}(X)\ra\mf{M}(Y)$. 
\item If $f$ is equivariant with respect to actions of a group $\G$, so is $\phi$.
\item The map $\phi$ extends $f$ in the following sense: if vertices $v_n\in\mf{M}(X)$ converge to a point $w\in\mc{A}$, we have $\phi(v_n)\ra f(w)$.
\item If $\Phi\colon X\ra Y$ is an isomorphism whose natural extension to $\partial X$ restricts to $f$ on $\mc{A}$, then $\Phi$ and $\phi$ coincide on $\mf{M}(X)$.
\end{enumerate}
\end{thm}

Finally, the third step in the proof of Theorem~\ref{ext Moeb} is Theorem~\ref{ext partial iso thm} from the introduction. Section~\ref{ext partial iso sect} will be devoted to a proof of the latter. 

Using these results:

\begin{proof}[Proof of Theorem~\ref{ext Moeb}]
Let $\mc{A}$, $\mc{B}$ and $\G$ be as in the statement of the theorem. Since the $\G$--actions are non-elementary, $X$ and $Y$ are neither single points nor $\R$; moreover, by Lemma~\ref{minimal vs essential}, the actions are essential. Theorem~\ref{OR exist thm} guarantees that $\mf{M}(X)$ and $\mf{M}(Y)$ are nonempty. Proposition~\ref{my corner} and Lemma~\ref{ample sets} show that $\mc{A}$ and $\mc{B}$ are ample. Theorem~\ref{from OR thm} now yields a $\G$--equivariant distance-preserving bijection $\phi\colon\mf{M}(X)\ra\mf{M}(Y)$. Since the action $\Aut(X)\acts X$ is cocompact and preserves $\mf{M}(X)$, Theorem~\ref{ext partial iso thm} shows that $\phi$ uniquely extends to a full isomorphism $\Phi\colon X\ra Y$. 

Note that the uniqueness of $\Phi$ guarantees its $\G$--equivariance. Indeed, given $g\in\G$, the maps $g\o\Phi$ and $\Phi\o g$ are, respectively, extensions of $g\o\phi$ and $\phi\o g$. Since the latter coincide by equivariance of $\phi$, we have $g\o\Phi=\Phi\o g$.

Since $\mf{M}(X)$ is invariant under the essential action $\G\acts X$, every point of $\partial_{\rm reg}X$ is limit of a sequence in $\mf{M}(X)$. Part~(3) of Theorem~\ref{from OR thm} thus guarantees that $f$ is precisely the boundary extension of $\Phi$. Finally, $\Phi$ is the unique isomorphism extending $f$ by part~(4) of Theorem~\ref{from OR thm} and the uniqueness part in Theorem~\ref{ext partial iso thm}.
\end{proof}

\subsection{Existence of OR medians.}\label{OR exist sect}

We consider the following setting:

\begin{ass}
Let $X$ be irreducible, locally finite, with no free faces and $\dim X\geq 1$. Let $G$ be a group acting cocompactly on $X$. 

Later in this subsection (from Proposition~\ref{colouring} onwards), we will require $G$ to be a discrete group acting properly on $X$.
\end{ass}

It will be useful to consider the projection $\pi_v\colon\overline X\ra 2^{\mscr{W}_v}$ defined by: 
\[{\pi_v(x)=\mscr{W}(v|x)\cap\mscr{W}_v}.\] 
The proof of Theorem~\ref{OR exist thm} will be based on the following dichotomy.

\begin{lem}\label{straight SLR ray or}
Suppose that $\partial_{\rm reg}X\neq\emptyset$. Then either $X$ contains a straight regular ray, or there exist a point $x\in\partial_{\rm reg}X$, a vertex $v\in X$ and a hyperplane $\mf{w}\in\mscr{W}_v$ such that $\#\pi_v(x)\geq 2$ and no element of $\pi_v(x)$ is transverse to $\mf{w}$. 
\end{lem}
\begin{proof}
By Lemma~\ref{SLR rays exist}, there exists an SLR ray $r\cu X$. Set $x=r^+$. If $r$ is straight, we are done.

Otherwise, a vertex $v=r(n)$ must satisfy $\#\pi_v(x)\geq 2$; we choose $v$ so as to minimise $n\geq 1$. Let $\mf{w}$ be the only hyperplane separating $v$ and $r(n-1)$. If $\mf{w}$ were transverse to some $\mf{v}\in\pi_v(x)$, then $\mf{w}$ and $\mf{v}$ would both be adjacent to $r(n-1)$, violating minimality of $n$.  
\end{proof}

It is in the second case of Lemma~\ref{straight SLR ray or} that it will be particularly hard to prove Theorem~\ref{OR exist thm}. We will construct OR medians based on Lemma~\ref{easy OR}, so we will need some results that allow us to piece together geodesics in order to assemble SLR rays. We now prove a sequence of results with this purpose, culminating in Theorem~\ref{roaming components 3}.

\begin{lem}\label{stabilising hyperplanes}
Given an element $g\in G$ and a vertex $v\in X$, there exists $n\geq 1$ such that every hyperplane in $\mscr{W}_v\cap\mscr{W}_{g^nv}$ is preserved by $g^n$.
\end{lem}
\begin{proof}
Let us write $\mscr{W}_v=P\sqcup Q$, where $P$ is the subset of hyperplanes preserved by some power of $g$. Possibly replacing $g$ with a power, we can assume that $g$ fixes every element of $P$. If a hyperplane $\mf{w}$ lies in $\mscr{W}_v\cap\mscr{W}_{g^mv}$ for infinitely many $m>0$, we must have $g^{-m}\mf{w}\in\mscr{W}_v$ for the same integers $m$; since $\mscr{W}_v$ is finite, we then have $\mf{w}\in P$. We conclude that, if $n$ is sufficiently large, we have $\mscr{W}_v\cap\mscr{W}_{g^nv}\cu P$. 
\end{proof}

Given an element $e\in\llk v$, we denote by $\mf{w}(e)$ the hyperplane dual to the edge determined by $e$. More generally, given a subset $F\cu\llk v$, we write $\mf{w}(F)=\{\mf{w}(e)\mid e\in F\}$. Note that each $g\in G$ takes $e$ to an element $ge\in\llk gv$ and $F$ to a subset $gF\cu\llk gv$; we have $\mf{w}(ge)=g\mf{w}(e)$ and $\mf{w}(gF)=g\mf{w}(F)$. 

\begin{lem}\label{semi-straight+}
Given $v\in X$ and $e\in\llk v$, there exists a non-elliptic element $g\in G$ such that:
\begin{enumerate}
\item $\pi_v(gv)=\{\mf{w}(e)\}$;
\item every hyperplane in $\mscr{W}_v\cap\mscr{W}_{gv}$ is preserved by $g$. 
\end{enumerate}
\end{lem}
\begin{proof}
Let $e_1=e,e_2,...,e_k$ form a maximal clique in $\llk v$. Consider the $\CAT$ cube complex $Z=\mf{w}(e_2)\cap...\cap\mf{w}(e_k)$ and let $H\leq G$ be the stabiliser of all sides of all $\mf{w}(e_i)$ with $i\geq 2$ (if $k=1$, we have $Z=X$ and $H=G$). 

Note that $v$ projects to a vertex $\overline v\in Z$ and $e$ projects to an isolated point $\overline e\in\lk_Z(\overline v)$. The cube complex $Z$ has no free faces, by a repeated application of Remark~\ref{nff hyp rmk}. The action $H\acts Z$ is cocompact by Lemma~\ref{cocompact stabilisers} and, therefore, it is essential. By Proposition~3.2 in \cite{CS}, there exists an element $g\in H$ such that $\langle g\rangle\cdot\overline v$ is unbounded and $\mf{w}(\overline e)\in\mscr{W}(\overline v|g^n\overline v)$ for all $n\geq 1$. As $\overline e$ is isolated, we have $\pi_{\overline v}(g^n\overline v)=\{\mf{w}(\overline e)\}$. 

Note that every hyperplane of $X$ separating $v$ and $g^nv$ must be transverse to all $\mf{w}(e_i)$ with $i\geq 2$; hence $\pi_v(g^nv)=\{\mf{w}(e)\}$ for all $n\geq 1$. Finally, Lemma~\ref{stabilising hyperplanes} allows us to assume that $g$ also satisfies condition~(2).
\end{proof}

We collect here two useful remarks for the subsequent discussion.

\begin{rmk}\label{SL and transversality}
Let $v,w\in X$ be vertices with $\pi_v(w)=\{\mf{v}\}$. For a hyperplane $\mf{w}\in\mscr{W}_w$ with $\mf{w}\neq\mf{v}$, we have $\mf{w}\in\mscr{W}_v$ if and only if $\mf{w}$ is transverse to $\mf{v}$. Indeed, if $\mf{w}$ and $\mf{v}$ are not transverse, $\mf{v}$ separates $v$ and $\mf{w}$; hence $\mf{w}\not\in\mscr{W}_v$. On the other hand, if $\mf{w}$ is not adjacent to $v$, there exists $\mf{u}\in\mscr{W}_v$ separating $v$ and $\mf{w}$. Since $w$ lies in the carrier of $\mf{w}$, we then have $\mf{u}\in\mscr{W}(v|w)$; hence $\mf{u}=\mf{v}$ and $\mf{w}$ is not transverse to $\mf{v}$.
\end{rmk}

\begin{rmk}\label{finding disjoint cliques}
Let $v\in X$ be a vertex and $c\cu\llk v$ a clique. We can always find a maximal clique $c'\cu\llk v$ with $c\cap c'=\emptyset$. In order to see this, let us denote the elements of $c$ by $e_1,...,e_k$. First, observe that there exists $e_1'\in\llk v$ such that $e_1',e_2,...,e_k$ is a clique and ${\delta(e_1,e_1')=2}$. Indeed, we can extend $c$ to a maximal clique $\tilde c_1$ and, since $X$ has no free faces, the clique $\tilde c_1\setminus\{e_1\}$ must also be contained in a maximal clique $\tilde c_2\neq\tilde c_1$. We can then take any point of $\tilde c_2\setminus\tilde c_1$ as $e_1'$.

Now, we repeat this procedure to obtain a clique $e_1',e_2',e_3,...,e_k$ with $\delta(e_2,e_2')=2$ and so on. In the end, we obtain a clique $e_1',...,e_k'$ with $\delta(e_i,e_i')=2$ for every $i$. The latter property implies that any maximal clique containing $\{e_1',...,e_k'\}$ is disjoint from $c$.
\end{rmk}

Recall that, given two graphs $\mc{G}_1,\mc{G}_2$, their \emph{join} $\mc{G}_1\ast\mc{G}_2$ is obtained from the disjoint union $\mc{G}_1\sqcup\mc{G}_2$ by adding edges connecting every vertex of $\mc{G}_1$ to every vertex of $\mc{G}_2$. Given a graph $\mc{G}$, its \emph{opposite} is the graph $\mc{G}^o$ with the same vertex set and such that a pair of vertices is joined by an edge in $\mc{G}^o$ if and only if it is not in $\mc{G}$.

A graph $\mc{G}$ is \emph{irreducible} if it does not split as a join of two proper subgraphs; equivalently, the opposite $\mc{G}^o$ is connected. Every finite graph $\mc{G}$ can be decomposed as a join $\mc{G}=\mc{G}_1\ast ...\ast\mc{G}_k$ of pairwise-disjoint, irreducible subgraphs $\mc{G}_i$. These factors are unique up to permutation, as the $\mc{G}_i^o$ are exactly the connected components of the opposite $\mc{G}^o$. We refer to the $\mc{G}_i$ as the \emph{irreducible components} of $\mc{G}$.

\begin{prop}\label{roaming components 1}
Given an irreducible component $L\cu\llk v$ and a element $e\in L$, there exists $g\in G$ with $\pi_v(gv)=\{\mf{w}(e)\}$ and $\mf{w}(L)\cap\mscr{W}_{gv}=\emptyset$.
\end{prop}
\begin{proof}
Among elements that satisfy conditions~(1) and~(2) of Lem\-ma~\ref{semi-straight+}, let us pick $g\in G$ minimising the cardinality of the set $\mf{W}:=\mf{w}(L)\cap\mscr{W}_{gv}$. Suppose for the sake of contradiction that $\mf{W}$ is nonempty. 

Let $\mc{W}\cu L$ be the subset with $\mf{w}(\mc{W})=\mf{W}$. Observe that $\mf{w}(e)$ does not lie in $\mscr{W}_{gv}$, or it would be the only hyperplane separating $v$ and $gv$. In this case, we would have $g\mf{w}(e)=\mf{w}(e)$ and $g^2v=v$, violating the fact that $g$ is not elliptic. We conclude that $L\setminus\mc{W}$ is nonempty, as it contains $e$. 

Since the graph $L$ does not split as a join of $\mc{W}$ and $L\setminus\mc{W}$, there exist elements $f\in L\setminus\mc{W}$ and $f'\in\mc{W}$ with $\delta(f,f')\geq 2$. Note that $\mf{w}(gf)$ is not transverse to $\mf{w}(e)$, or Remark~\ref{SL and transversality} would yield $\mf{w}(gf)\in\mscr{W}_v\cap\mscr{W}_{gv}$ and $\mf{w}(gf)=g^{-1}\mf{w}(gf)=\mf{w}(f)$, violating $f\not\in\mc{W}$. Note moreover that $\mf{w}(f')=\mf{w}(gf')\in\mscr{W}_v\cap\mscr{W}_{gv}$ is transverse to every element of $\mscr{W}(v|gv)$. In particular $\mf{w}(gf)\not\in\pi_{gv}(v)$, as $\delta(gf,gf')=\delta(f,f')\geq 2$.

Now, Lemma~\ref{semi-straight+} provides $h\in G$ with $\mscr{W}(gv|hgv)\cap\mscr{W}_{gv}=\{\mf{w}(gf)\}$ and such that every hyperplane in $\mscr{W}_{gv}\cap\mscr{W}_{hgv}$ is preserved by $h$. As $\mf{w}(gf)$ and $\mf{w}(e)$ are not transverse and $\mf{w}(gf)\not\in\pi_{gv}(v)$, we can concatenate geodesics from $v$ to $gv$ and from $gv$ to $hgv$ in order to obtain an SL geodesic from $v$ to $hgv$. It follows that $\pi_v(hgv)=\{\mf{w}(e)\}$. Every hyperplane in $\mscr{W}_v\cap\mscr{W}_{hgv}$ must lie in $\mscr{W}_v\cap\mscr{W}_{gv}\cap\mscr{W}_{hgv}$ and is therefore preserved by $h$, $g$ and $hg$. Thus, $hg$ satisfies conditions~(1) and~(2) of Lemma~\ref{semi-straight+}.

By minimality of $\#\mc{W}$, the inclusion $\mf{w}(L)\cap\mscr{W}_{hgv}\cu\mf{W}$ must be an equality. Hence $\mf{w}(f')\in\mscr{W}_v\cap\mscr{W}_{hgv}$ and $\mf{w}(f')=\mf{w}(gf')$ lies in $\mscr{W}_{gv}\cap\mscr{W}_{hgv}$, contradicting the fact that $\delta(gf,gf')=\delta(f,f')\geq 2$.
\end{proof}

In the rest of the subsection, we also require $G\acts X$ to be properly discontinuous.

\begin{prop}\label{colouring}
Let $v\in X$ be a vertex and ${c^-,c^+\cu\llk v}$ maximal cliques. There exists a neatly contracting $g\in G$ with ${\pi_v(g^nv)=\mf{w}(c^-)}$ and $\pi_{g^nv}(v)=\mf{w}(g^nc^+)$ for all $n\geq 1$, unless possibly if $X=\R$ and $c^-=c^+$. 
\end{prop}
\begin{proof}
By Proposition~\ref{my corner}, the cube complex $X$ is sector-heavy. Let $\mf{h}_+$ be a halfspace contained in the intersection of all sides of elements of $\mf{w}(c^+)$ that do not contain $v$. Similarly, take $\mf{h}_-$ in the intersection of all sides of elements of $\mf{w}(c^-)$ that do not contain $v$. We only need to find a neatly contracting element $g\in G$ such that $g^nv\in\mf{h}_-$ and $g^{-n}v\in\mf{h}_+$ for all $n\geq 1$. For this, it suffices to construct a neatly contracting element $g\in G$ with $g^+\in\mf{h}_-$ and $g^-\in\mf{h}_+$, passing then to a large power if necessary.

The case $X=\R$ is immediate. If $X\neq\R$, the action $G\acts X$ is non-elementary by Lemma~\ref{geometric implies nonelementary}. Since $X$ has no free faces, $G\acts X$ is also essential. Lemma~\ref{cNT RW} provides a neatly contracting element $h\in G$. Applying Lemma~\ref{ample sets} to the set $G\cdot h^+\cu\partial_{\rm reg}X$, we obtain $g_1,g_2\in G$ with ${g_1\cdot h^+\in\mf{h}_-}$, $g_2\cdot h^+\in\mf{h}_+$ and $g_1\cdot h^+\neq g_2\cdot h^+$. Considering strongly separated chains containing these two points, we find strongly separated halfspaces $\mf{k}_{\pm}$ with $\mf{k}_+\cu\mf{h}_+$ and $\mf{k}_-\cu\mf{h}_-$. The Double Skewering Lemma of \cite{CS} now provides $g\in G$ with $g\mf{k}_+^*\cu\mf{k}_-$ and $g^2$ is the required neatly contracting element.
\end{proof}

The following is a small but essential improvement on Proposition~\ref{roaming components 1}.

\begin{lem}\label{roaming components 2}
Suppose that $X\neq\R$ and let $v\in X$ be a vertex with irreducible link. For every element $e\in\llk v$, there exists an element $g\in G$ satisfying $\pi_v(g^nv)=\{\mf{w}(e)\}$ and $\mscr{W}_v\cap\mscr{W}_{g^nv}=\emptyset$ for all $n\geq 1$.
\end{lem}
\begin{proof}
We start by constructing $g\in G$ with $\pi_v(gv)=\{\mf{w}(e)\}$, $\mscr{W}_v\cap\mscr{W}_{gv}=\emptyset$ and $\mf{w}(ge)\not\in\pi_{gv}(v)$. Proposition~\ref{roaming components 1} already provides an element $h\in G$ with $\pi_v(hv)=\{\mf{w}(e)\}$ and $\mscr{W}_v\cap\mscr{W}_{hv}=\emptyset$. Let $c\cu\llk v$ be the clique with $\pi_{hv}(v)=\mf{w}(hc)$. If $e\not\in c$, we can take $g=h$; suppose instead that $e\in c$.

Remark~\ref{finding disjoint cliques} yields a maximal clique $c'\cu\llk v$ with $c\cap c'=\emptyset$. Since ${X\neq\R}$, Proposition~\ref{colouring} provides an element $k\in G$ with $\pi_{hv}(kv)=\mf{w}(hc')$ and $\pi_{kv}(hv)=\mf{w}(kc')$. By Remark~\ref{SL and transversality}, no element of $\mf{w}(hc')$ can be transverse to $\mf{w}(e)$ or it would lie in $\mscr{W}_v\cap\mscr{W}_{hv}=\emptyset$. Note moreover that $\pi_{hv}(kv)=\mf{w}(hc')$ is disjoint from $\pi_{hv}(v)=\mf{w}(hc)$. It follows that a geodesic from $v$ to $kv$ is SL and passes through $hv$; in particular, we have $\pi_v(kv)=\{\mf{w}(e)\}$ and $\mscr{W}_v\cap\mscr{W}_{kv}\cu\mscr{W}_v\cap\mscr{W}_{hv}=\emptyset$. Since the clique $c'$ is maximal, we have $\pi_{kv}(v)=\pi_{kv}(hv)=\mf{w}(kc')$, which does not contain $\mf{w}(ke)$. In conclusion, we can take $g=k$.

We now show by induction on $n\geq 1$ that we have $\pi_v(g^nv)=\{\mf{w}(e)\}$, ${\mscr{W}_v\cap\mscr{W}_{g^nv}=\emptyset}$, $\mf{w}(g^ne)\not\in\pi_{g^nv}(v)$ and, moreover, that a geodesic from $v$ to $g^nv$ passes through all $g^iv$ with $1\leq i\leq n-1$. This will conclude the proof.

The case $n=1$ is trivial; let us assume that the above hold for some $n\geq 1$. Since $\mf{w}(g^ne)\not\in\pi_{g^nv}(v)$, a geodesic from $v$ to $g^{n+1}v$ is obtained by concatenating any geodesic from $v$ to $g^nv$ with any geodesic from $g^nv$ to $g^{n+1}v$. By convexity of hyperplane carriers, we then have $\mscr{W}_v\cap\mscr{W}_{g^{n+1}v}\cu\mscr{W}_v\cap\mscr{W}_{g^nv}=\emptyset$, $\mscr{W}_v\cap\mscr{W}(gv|g^{n+1}v)\cu\mscr{W}_v\cap\mscr{W}_{gv}=\emptyset$, ${\mscr{W}_{g^{n+1}v}\cap\mscr{W}(v|g^nv)\cu\mscr{W}_{g^{n+1}v}\cap\mscr{W}_{g^nv}=\emptyset}$. The last two inclusions guarantee that $\pi_v(g^{n+1}v)=\pi_v(gv)=\{\mf{w}(e)\}$ and $\pi_{g^{n+1}v}(v)=\pi_{g^{n+1}v}(g^nv)\not\ni\mf{w}(g^{n+1}e)$.
\end{proof}

The following will be our key tool in proving Theorem~\ref{OR exist thm}.

\begin{thm}\label{roaming components 3}
Consider an irreducible component $L\cu\llk v$ with $\#L\geq 3$. Given $e\in L$ and a maximal clique $c\cu L$, there exists $g\in G$ satisfying:
\begin{enumerate}
\item $\pi_v(gv)=\{\mf{w}(e)\}$;
\item every hyperplane in $\mscr{W}_v\cap\mscr{W}_{gv}$ is preserved by $g$;
\item $\mf{w}(L)\cap\mscr{W}_{gv}=\emptyset$;
\item $\pi_{gv}(v)=\mf{w}(gc)$.
\end{enumerate}
\end{thm}
\begin{proof}
Let $\llk v=L_1\ast ... \ast L_s$ be the decomposition into irreducible components, with $L_1=L$. Choose maximal cliques $\kappa_i\cu L_i$ for each $i\geq 2$. Any two hyperplanes in $\mf{U}:=\mf{w}(\kappa_2)\cup...\cup\mf{w}(\kappa_s)$ are transverse; let $Z$ denote their intersection. The vertex $v\in X$ projects to a vertex $\overline v\in Z$ and $\lk_Z(\overline v)$ is naturally identified with $L_1$. Let $H<G$ be the subgroup that stabilises every side of every hyperplane in $\mf{U}$. The action $H\acts Z$ is proper and, by Lemma~\ref{cocompact stabilisers}, also cocompact. Since $\lk_Z(\overline v)$ is irreducible, with at least three vertices, $Z$ is irreducible and $Z\neq\R$. Note that $Z$ has no free faces.

Observe that $e$ and $c$ project to an element $\overline e\in\lk_Z(\overline v)$ and a maximal clique $\overline c\cu\lk_Z(\overline v)$. Applying Lemma~\ref{roaming components 2} to the action $H\acts Z$, we find $h_1\in H$ with $\pi_{\overline v}(h_1\overline v)=\{\mf{w}(\overline e)\}$ and $\mscr{W}_{\overline v}\cap\mscr{W}_{h_1\overline v}=\emptyset$. By Lemma~\ref{stabilising hyperplanes}, we can moreover assume that $h_1$ preserves every hyperplane of $X$ lying in $\mscr{W}_v\cap\mscr{W}_{h_1v}$. 

Let $\overline c_1\cu\lk_Z(\overline v)$ be the clique with $\pi_{h_1\overline v}(\overline v)=\mf{w}(h_1\overline c_1)$. As $Z$ has no free faces, Remark~\ref{finding disjoint cliques} provides a maximal clique $\overline c_2\cu\lk_Z(\overline v)$ with ${\overline c_1\cap\overline c_2=\emptyset}$. Proposition~\ref{colouring} and Lemma~\ref{stabilising hyperplanes} now yield an element $h_2\in H$ satisfying $\pi_{h_1\overline v}(h_2h_1\overline v)=\mf{w}(h_1\overline c_2)$, $\pi_{h_2h_1\overline v}(h_1\overline v)=\mf{w}(h_2h_1\overline c)$ and such that $h_2$ preserves every hyperplane of $X$ lying in $\mscr{W}_{h_1v}\cap\mscr{W}_{h_2h_1v}$. Since $\pi_{h_1\overline v}(\overline v)$ and $\pi_{h_1\overline v}(h_2h_1\overline v)$ are disjoint, a geodesic from $\overline v$ to $h_2h_1\overline v$ passes through $h_1\overline v$. As $\mscr{W}_{\overline v}\cap\mscr{W}_{h_1\overline v}=\emptyset$, Remark~\ref{SL and transversality} shows that every such geodesic is SL. In particular, we have $\pi_{\overline v}(h_2h_1\overline v)=\{\mf{w}(\overline e)\}$, $\mscr{W}_{\overline v}\cap\mscr{W}_{h_2h_1\overline v}\cu\mscr{W}_{\overline v}\cap\mscr{W}_{h_1\overline v}=\emptyset$ and, since $\overline c$ is maximal, also $\pi_{h_2h_1\overline v}(\overline v)=\pi_{h_2h_1\overline v}(h_1\overline v)=\mf{w}(h_2h_1\overline c)$.

Let us now set $g=h_2h_1$. Note that a hyperplane of $X$ lies in $\mscr{W}(v|gv)$ if and only if it is transverse to every element of $\mf{U}$ and originates a hyperplane of $Z$ that lies in $\mscr{W}(\overline v|g\overline v)$. It follows that $g$ satisfies conditions~(1) and~(4). 

Observing that a geodesic from $v$ to $gv$ passes through the point $h_1v$, we have $\mscr{W}_v\cap\mscr{W}_{gv}\cu\mscr{W}_v\cap\mscr{W}_{h_1v}\cap\mscr{W}_{h_2h_1v}$ and it follows that $g$ satisfies condition~(2). Finally, each element of $\mf{w}(L)\cap\mscr{W}_{gv}$ is transverse to every hyperplane in $\mf{U}$ and therefore originates a hyperplane of $Z$ lying in $\mscr{W}_{\overline v}\cap\mscr{W}_{g\overline v}$. As the latter is empty, we conclude that $g$ also satisfies condition~(3).
\end{proof}

\begin{proof}[Proof of Theorem~\ref{OR exist thm}.]
Observe that $\partial_{\rm reg}X\neq\emptyset$ (for instance, by Lem\-mas~\ref{geometric implies nonelementary} and~\ref{cNT RW}). Thus, according to Lemma~\ref{straight SLR ray or}, there are two cases to consider. First, suppose that there exists a straight regular ray $r$ and set $x=r^+$. Let $\mf{u}$ be the first hyperplane crossed by $r$ and let $\mf{h}$ its side containing $r(0)$; let $y$ be any regular point in $\mf{h}$. Let $\mf{k}$ be a halfspace strongly separated from $\mf{h}$, with $x\in\mf{k}$. For every regular point $z\in\mf{k}$ with $z\neq x$, we have $x\op_zy$ as a consequence of $r$ being straight. If no such $z$ exists, a sub-ray of $r$ consists only of vertices that have degree $2$ in $X$, hence $X=\R$.

If no straight regular ray exists, let $x\in\partial_{\rm reg}X$, $v\in X$ and $\mf{w}\in\mscr{W}_v$ be as provided by Lemma~\ref{straight SLR ray or}. There exist $f\in\llk v$ and a clique $c'\cu\llk v$ with $\mf{w}=\mf{w}(f)$ and $\pi_v(x)=\mf{w}(c')$. The set $c'\cup\{f\}$ is entirely contained in a single irreducible component $L\cu\llk v$ and we have $\#L\geq 3$. We are going to show that every $e\in L$ is SLR-extendable. We then conclude by taking SLR rays extending $f$ and two elements of $c'$, in order to define, respectively, points $x,y,z\in\partial_{\rm reg}X$ with $m(x,y,z)=v$ and $x\op_zy$.

Remark~\ref{finding disjoint cliques} provides a maximal clique $c\cu L$ such that $c\cap c'=\emptyset$. Given $e\in L$, let $g\in G$ be an element satisfying conditions~(1)--(4) of Theorem~\ref{roaming components 3}. Since $\pi_{gv}(gx)=\mf{w}(gc')$ is disjoint from $\pi_{gv}(v)=\mf{w}(gc)$, a ray $\rho$ from $v$ to $gx$ passes through $gv$. If some $\mf{u}\in\mf{w}(gc')$ were transverse to $\mf{w}(e)$, Remark~\ref{SL and transversality} would yield $\mf{u}\in\mscr{W}_v\cap\mscr{W}_{gv}$ and we would have $\mf{u}=g^{-1}\mf{u}\in\mf{w}(c')\cu\mf{w}(L)$, violating the assumption that $\mf{w}(L)\cap\mscr{W}_{gv}=\emptyset$. Hence, no element of $\mf{w}(gc')$ is transverse to $\mf{w}(e)$ and $\rho$ is an SLR ray extending $e$.
\end{proof}

\subsection{From OR medians to partial isomorphisms.}\label{from OR sect}

\begin{ass}
Throughout this subsection, we only assume that $X$ and $Y$ are locally finite. We consider ample, concise subsets $\mc{A}\cu\partial X$ and $\mc{B}\cu\partial Y$ and a M\"obius bijection $f\colon \mc{A}\ra\mc{B}$.
\end{ass}

To avoid cumbersome formulas, we introduce the following notation for $x,y,z,w\in \mc{A}$ and $v\in Y$:
\[(x\cdot y)_v^f:=(f(x)\cdot f(y))_v, \hspace{.5cm} m^f(x,y,z):=m(f(x),f(y),f(z)),\] 
\[\crt^f(x,y,z,w):=\crt(f(x),f(y),f(z),f(w)).\]
Note that, since $\mc{A}$ is concise, a $4$--tuple $(x,y,z,w)\in\mc{A}^4$ lies in $\mscr{A}(X)$ if and only if no three of the four points $x$, $y$, $z$ and $w$ coincide (cf.\ Lemma~\ref{infinite Gromov product}). 

\begin{lem}\label{order is preserved}
Consider $x,y,z,w\in \mc{A}$ with $x\op_z y$. Setting $m=m(x,y,z)$ and $m'=m^f(x,y,z)$, we have:
\[(x\cdot w)_m=(x\cdot w)^f_{m'}, \hspace{.5cm} (y\cdot w)_m=(y\cdot w)^f_{m'}, \hspace{.5cm} (z\cdot w)_m=(z\cdot w)^f_{m'}.\]
\end{lem}
\begin{proof}
Recall from Section~\ref{crossratio section} that $\mscr{A}(X)\cu(\overline X)^4$ is the set of $4$--tuples $(x,y,z,w)$ such that at most one of the three quantities $(x\cdot y)_v+(z\cdot w)_v$, $(x\cdot z)_v+(y\cdot w)_v$ and $(x\cdot w)_v+(y\cdot z)_v$ is infinite. Note that $(x,y,z,w)\in\mscr{A}(X)$ and:
\[\crt_m(x,y,z,w)=\ll (z\cdot w)_m:(y\cdot w)_m:(x\cdot w)_m\rr,\]
\[\crt^f_{m'}(x,y,z,w)=\ll (z\cdot w)^f_{m'}:(y\cdot w)^f_{m'}:(x\cdot w)^f_{m'}\rr.\]
Since $x\op_zy$, Lemma~\ref{one is zero} shows that either $(y\cdot w)_m=0$ or $(x\cdot w)_m=0$. Since $f(x)\op_{f(z)}f(y)$ by Corollary~\ref{op-preserving}, also one among $(y\cdot w)^f_{m'}$ and $(x\cdot w)^f_{m'}$ must vanish. The equality $\crt(x,y,z,w)=\crt^f(x,y,z,w)$ then implies that $(x\cdot w)_m=(x\cdot w)^f_{m'}$, $(y\cdot w)_m=(y\cdot w)^f_{m'}$ and $(z\cdot w)_m=(z\cdot w)^f_{m'}$.
\end{proof}

For the next results, we consider points $x_1,x_2,x,y_1,y_2,y\in \mc{A}$ satisfying $x_1\op_x x_2$ and $y_1\op_y y_2$. We set:
\begin{align*} 
&m_x=m(x_1,x_2,x), & &m_y=m(y_1,y_2,y), \\ 
&m_x'=m^f(x_1,x_2,x), & &m_y'=m^f(y_1,y_2,y).
\end{align*}

\begin{lem}\label{more preserved products}
If $u,v\in\mc{A}$ are distinct points, we have $(u\cdot v)_{m_x}=(u\cdot v)^f_{m_x'}$.
\end{lem}
\begin{proof}
As $(x_1,x_2,u,v)\in\mscr{A}(X)$, we have $\crt(x_1,x_2,u,v)=\crt^f(x_1,x_2,u,v)$. Computing the cross ratio triples $\crt(x_1,x_2,u,v)$ and $\crt^f(x_1,x_2,u,v)$ in terms of Gromov products based at $m_x$ and $m_x'$, respectively, we obtain
\[\ll(u\cdot v)_{m_x}: b:c\rr=\ll (u\cdot v)^f_{m_x'}: b':c'\rr,\]
where $b=b'$ and $c=c'$ by Lemma~\ref{order is preserved}. Hence $(u\cdot v)_{m_x}=(u\cdot v)^f_{m_x'}$.
\end{proof}

\begin{prop}\label{d(mx,my)}
We have:
\[d(m_x,m_y)=(y_1\cdot y_2)_{m_x}+\left|(y_1\cdot y)_{m_x}-(y_2\cdot y)_{m_x}\right|.\]
\end{prop}
\begin{proof}
Set $p=m(y_1,y_2,m_x)$. As $p$ is the gate-projection of $m_x$ to $I(y_1,y_2)$, we have 
\[d(m_x,m_y)=d(m_x,p)+d(p,m_y),\]
where $d(m_x,p)=(y_1\cdot y_2)_{m_x}$. Up to exchanging $y_1$ and $y_2$, we can assume that $p$ lies in $I(m_y,y_2)$. Since no element of $\mscr{W}(p|y_2)=\mscr{W}(m_x,y_1|y_2)$ separates $m_x$ and $y$, it follows that the set $\mscr{W}(m_x,y_1|y_2,y)$ is empty. We conclude that $(y_2\cdot y)_{m_x}=\#\mscr{W}(m_x|y_1,y_2,y)$. On the other hand, observing that $\mscr{W}(p|m_y)=\mscr{W}(m_x,y_2|y_1,y)$, we have 
\[\mscr{W}(m_x|y_1,y)=\mscr{W}(m_x|y_1,y_2,y)\sqcup\mscr{W}(p|m_y)\]
and $(y_1\cdot y)_{m_x}=(y_2\cdot y)_{m_x}+d(p,m_y)$. 
\end{proof}

As an immediate consequence of Lemma~\ref{more preserved products} and Proposition~\ref{d(mx,my)}:

\begin{cor}\label{distance is preserved} 
We have $d(m_x,m_y)=d(m_x',m_y')$. In particular, $m_x'$ and $m_y'$ coincide whenever $m_x$ and $m_y$ do.
\end{cor}

Now consider the set:
\[\mf{M}_{\mc{A}}(X)=\{m(x,y,z)\mid x,y,z\in \mc{A}, \ x\op_zy\}\cu X.\]
The subset $\mf{M}_{\mc{B}}(Y)\cu Y$ is defined similarly. Applying Corollary~\ref{distance is preserved} to $f$ and $f^{-1}$, we obtain maps $\phi\colon\mf{M}_{\mc{A}}(X)\ra\mf{M}_{\mc{B}}(Y)$ and ${\psi\colon\mf{M}_{\mc{B}}(Y)\ra\mf{M}_{\mc{A}}(X)}$ that preserve distances and satisfy:
\[\phi(m(x,y,z))=m^f(x,y,z), \hspace{.5cm} \psi(m(x,y,z))=m^{f^{-1}}(x,y,z),\] 
whenever $x\op_zy$. By Corollary~\ref{op-preserving}, the maps $f$ and $f^{-1}$ both preserve the relation $\op$; it follows that $\phi\o\psi$ and $\psi\o\phi$ are the identity. We have shown:

\begin{cor}\label{isom on M(X)}
The map $\phi\colon\mf{M}_{\mc{A}}(X)\ra\mf{M}_{\mc{B}}(Y)$ is a distance-preserving bijection.
\end{cor}

Before proving Theorem~\ref{from OR thm}, we need to make one more observation.

\begin{lem}\label{an orbit suffices} 
If $\mc{A}\cu\partial_{\rm reg}X$, the sets $\mf{M}_{\mc{A}}(X)$ and $\mf{M}(X)$ coincide.
\end{lem}
\begin{proof}
Since $\mc{A}\cu\partial_{\rm reg}X$, it is clear that $\mf{M}_{\mc{A}}(X)\cu\mf{M}(X)$. Let us consider a vertex $v\in\mf{M}(X)$ and $x,y,z\in\partial_{\rm reg}X$ with $m(x,y,z)=v$ and $x\op_zy$. 

Consider a strongly separated chain $\mf{k}_0\supsetneq\mf{k}_1\supsetneq ...$ containing $x$. Given $\mf{w}\in\mscr{W}_v$, the halfspace $\mf{k}_n$ is disjoint from $\mf{w}$ and on the same side of $\mf{w}$ as $x$ for all sufficiently large values of $n$. Since $\mscr{W}_v$ is finite, there exists a halfspace $\mf{h}_x\in\s_x$ that lies on the same side as $x$ with respect to all hyperplanes in $\mscr{W}_v$. We similarly define halfspaces $\mf{h}_y$ and $\mf{h}_z$. Since $\mc{A}$ is ample, we can pick points $x'\in \mc{A}\cap\mf{h}_x$, $y'\in \mc{A}\cap\mf{h}_y$ and $z'\in \mc{A}\cap\mf{h}_z$. These satisfy $x'\op_{z'}y'$ and $m(x',y',z')=v$, concluding the proof.
\end{proof}

\begin{proof}[Proof of Theorem~\ref{from OR thm}.]
Note that $\mc{A}$ and $\mc{B}$ are concise by Remark~\ref{regular is concise}. Part~(1) follows from Corollary~\ref{isom on M(X)} and Lemma~\ref{an orbit suffices}, whereas parts~(2) and~(4) are immediate from our construction. We are left to show that $\phi(v_n)\ra f(w)$ whenever vertices $v_n\in\mf{M}(X)$ converge to a point $w\in\mc{A}$.

Say $v_n=m(x_n,y_n,z_n)$ with $x_n,y_n,z_n\in\mc{A}$ and $x_n\op_{z_n}y_n$. Let us write $a_n$, $b_n$ and $c_n$ for the values $(x_n\cdot w)_{v_0}$, $(y_n\cdot w)_{v_0}$ and $(z_n\cdot w)_{v_0}$ ordered so that $a_n\leq b_n\leq c_n$. Consider a strongly separated chain $\mf{k}_0\supsetneq\mf{k}_1\supsetneq ...$ containing $w$. Note that we have $v_n\ra w$ if and only if, for every $k\geq 0$, there exists $N(k)\geq 0$ such that $v_n\in\mf{k}_k$ for all $n\geq N(k)$; the latter happens if and only if at least two elements of $\{x_n,y_n,z_n\}$ lie in $\mf{k}_k$. We conclude that $v_n\ra w$ if and only if $b_n\ra+\infty$. Lemma~\ref{more preserved products} now shows that $v_n\ra w$ if and only if the points $\phi(v_n)=m^f(x_n,y_n,z_n)$ converge to $f(w)$.
\end{proof}

\subsection{Extending partial isomorphisms.}\label{ext partial iso sect}

In this subsection:

\begin{ass}
Let $X$ and $Y$ be irreducible, locally finite and with no free faces. Let $\phi\colon A\ra B$ be a distance-preserving bijection between nonempty subsets $A\cu X$ and $ B\cu Y$. We also assume that $A$ and $B$ are preserved, respectively, by cocompact actions $G_1\acts X$ and $G_2\acts Y$. For now, we require $A$ and $B$ to only consist of vertices; this will change later on, when we consider actions on $\CAT$ \emph{cuboid} complexes.
\end{ass}

We will obtain a cubical isomorphism between $X$ and $Y$ by progressively extending $\phi$ to vertices adjacent to $A$. In order to do so, we will need to recognise halfspaces adjacent to $A$ by their intersections with $A$. This is the purpose of the next few results, up to Lemma~\ref{recognising W(v)}.

\begin{lem}\label{H vs h 1}
Given halfspaces $\mf{h}_1,\mf{h}_2\in\mscr{H}(X)$ satisfying $\mf{h}_1\cap A=\mf{h}_2\cap A$, we have $\mf{h}_1\cu\mf{h}_2$ or $\mf{h}_2\cu\mf{h}_1$.
\end{lem}
\begin{proof}
The cocompact action $G_1\acts X$ leaves $A$ invariant and, since $X$ has no free faces, it is essential and hyperplane-essential. The halfspaces $\mf{h}_1$ and $\mf{h}_2$ are not transverse, or Proposition~\ref{NS corner} would provide a halfspace contained in $\mf{h}_1\cap\mf{h}_2^*$ and, by essentiality of $G_1\acts X$, we would have $\mf{h}_1\cap\mf{h}_2^*\cap A\neq\emptyset$. We also cannot have $\mf{h}_1\cu\mf{h}_2^*$ or $\mf{h}_2^*\cu\mf{h}_1$, as $\mf{h}_1\cap\mf{h}_2\supseteq\mf{h}_1\cap A\neq\emptyset$ and $\mf{h}_1^*\cap\mf{h}_2^*\supseteq\mf{h}_1^*\cap A\neq\emptyset$. We conclude that either $\mf{h}_1\cu\mf{h}_2$ or $\mf{h}_2\cu\mf{h}_1$. 
\end{proof}

Consider a vertex $v\in A$, a hyperplane $\mf{w}\in\mscr{W}_v$, and the side $\mf{h}$ of $\mf{w}$ that does not contain $v$. It is convenient to introduce the notation $H_{\mf{w}}:=\mf{h}\cap A$.

\begin{rmk}\label{useful fact}
For every $v\in A$ and $\mf{w}\in\mscr{W}_v$, there exists $x\in H_{\mf{w}}$ with $\pi_v(x)=\{\mf{w}\}$. Since $A$ is $G_1$--invariant, this follows from Lemma~\ref{semi-straight+} applied to the cocompact action $G_1\acts X$.
\end{rmk}

\begin{lem}\label{H_w}
Consider $v\in A$ and $\mf{w}\in\mscr{W}_v$.
\begin{enumerate}
\item A point $x\in A$ lies in $H_{\mf{w}}$ if and only if, for every $y\in H_{\mf{w}}$, we have $d(x,y)<d(x,v)+d(v,y)$ (equivalently, $\pi_v(x)\cap\pi_v(y)\neq\emptyset$).
\item If $\mf{w'}\in\mscr{W}_v\setminus\{\mf{w}\}$, then $H_{\mf{w}}\neq H_{\mf{w}'}$.
\end{enumerate}
\end{lem}
\begin{proof}
Part~(2) follows from Lemma~\ref{H vs h 1}. Let us prove part~(1).

Since $\mf{w}\in\pi_v(y)$ for all $y\in H_{\mf{w}}$, the intersection $\pi_v(x)\cap\pi_v(y)$ is nonempty for all $x,y\in H_{\mf{w}}$. For the other implication, Remark~\ref{useful fact} provides $y_0\in H_{\mf{w}}$ with $\pi_v(y_0)=\{\mf{w}\}$. If $x\in A$ and $\pi_v(x)\cap\pi_v(y_0)\neq\emptyset$, it follows that $\mf{w}\in\pi_v(x)$, i.e.\ $x\in H_{\mf{w}}$.
\end{proof}

Given a vertex $v\in A$, let $\mc{H}_v$ denote the collection of all subsets $H\cu A$ with the property from part~(1) of Lemma~\ref{H_w}. More precisely, for every $x\in A$, we have:
\begin{equation*}\tag{$\ast$}\label{mcH defn} x\in H~\Leftrightarrow~\forall y\in H,~d(x,y)<d(x,v)+d(v,y). \end{equation*}
Recall that $d(x,y)<d(x,v)+d(v,y)$ is equivalent to $\pi_v(x)\cap\pi_v(y)\neq\emptyset$. 

\begin{rmk}\label{if a singleton}
If an element $H\in\mc{H}_v$ contains a point $y_0$ with $\#\pi_v(y_0)=1$, we must have ${H=H_{\mf{u}}}$ for the only hyperplane $\mf{u}\in\pi_v(y_0)$. Indeed, property~(\ref{mcH defn}) shows that $\mf{u}\in\pi_v(y)$ for all $y\in H$. Hence, a point $x\in A$ lies in $H$ if and only if $\mf{u}\in\pi_v(x)$, i.e.\ if and only if $x\in H_{\mf{u}}$.
\end{rmk}

\begin{lem}\label{recognising W(v)}
An element $H\in\mc{H}_v$ is of the form $H=H_{\mf{w}}$ for some hyperplane $\mf{w}\in\mscr{W}_v$ if and only if there exists a point $x\in H$ that belongs to no other element of $\mc{H}_v$.
\end{lem}
\begin{proof}
If, given $x\in A$, two hyperplanes $\mf{u}$ and $\mf{v}$ lie in $\pi_v(x)$, we have ${x\in H_{\mf{u}}\cap H_{\mf{v}}}$. By Lemma~\ref{H_w} and Remark~\ref{if a singleton}, we conclude that $x$ lies in a single element of $\mc{H}_v$ if and only if we have $\#\pi_v(x)=1$. 

By Remarks~\ref{useful fact} and~\ref{if a singleton}, an element $H\in\mc{H}_v$ can contain such a point if and only if it is of the form $H=H_{\mf{w}}$ for a hyperplane $\mf{w}\in\mscr{W}_v$.
\end{proof}

We will also employ the notation $\mc{H}_v$ and $H_{\mf{w}}$ for vertices $v\in Y$ and hyperplanes $\mf{w}\in\mscr{W}(Y)$, of course replacing $A$ with $B$ in the definitions.

In the following discussion, we consider two edges $e_1,e_2\cu X$, so that $e_i$ has endpoints $v_i\in A$ and $w_i\in X$. We set $\mf{w}_i=\mf{w}(e_i)$ and let $\mf{h}_i\in\mscr{H}(X)$ be the side of $\mf{w}_i$ containing $w_i$. As $\phi$ preserves distances, it induces a bijection $\mc{H}_{v_i}\ra\mc{H}_{\phi(v_i)}$. Lemma~\ref{recognising W(v)} guarantees the existence of a hyperplane $\phi_*(\mf{w}_i)\in\mscr{W}_{\phi(v_i)}$ with ${\phi(H_{\mf{w_i}})=H_{\phi_*(\mf{w}_i)}}$. We also define $\phi_*(\mf{h}_i)\in\mscr{H}(Y)$ as the side of $\phi_*(\mf{w}_i)$ that satisfies $\phi(\mf{h}_i\cap A)=\phi(H_{\mf{w_i}})=H_{\phi_*(\mf{w}_i)}=\phi_*(\mf{h}_i)\cap B$.

Finally, let $\wt\phi(w_i)\in\phi_*(\mf{h}_i)$ denote the vertex of $Y$ that is separated from $\phi(v_i)$ only by $\phi_*(\mf{w}_i)$. We will show in a moment that this definition of $\wt\phi(w_i)$ is independent of the choice of $e_i$ and $v_i$.

\begin{prop}\label{h1=h2}
\begin{enumerate}
\item[]
\item We have $\mf{h}_1\cap A=\mf{h}_2\cap A$ if and only if $\mf{h}_1=\mf{h}_2$.
\item If $\mf{h}_1\cap A=\mf{h}_2^*\cap A$, exactly one of the following happens:
	\begin{enumerate}
	\item $\mf{h}_2^*\subsetneq\mf{h}_1$ and $d(x,v_2)\leq d(x,v_1)+d(v_1,v_2)-4$ for all $x\in\mf{h}_1\cap A$;
	\item $\mf{h}_1=\mf{h}_2^*$ and $d(x_0,v_2)=d(x,v_1)+d(v_1,v_2)-2$ for some point $x_0\in\mf{h}_1\cap A$.
	\end{enumerate}
\end{enumerate}
\end{prop}
\begin{proof}
If we had $\mf{h}_1\cap A=\mf{h}_2\cap A$ and $\mf{h}_1\neq\mf{h}_2$, Lemma~\ref{H vs h 1} would allow us to assume that $\mf{h}_1\subsetneq\mf{h}_2$. Since $v_1$ is adjacent to $\mf{w}_1$, we would then have $v_1\in\mf{h}_2$. On the other hand, $v_1\in A$ and $v_1\not\in\mf{h}_1$, a contradiction.

We now address part~(2). First, suppose that $\mf{h}_1=\mf{h}_2^*$ and consider a point $x_0\in A$ with $\pi_{v_1}(x_0)=\{\mf{w}_1\}$, as provided by Remark~\ref{useful fact}. Note that $x_0\in\mf{h}_1\cap A$. Since $v_2\in\mf{h}_2^*=\mf{h}_1$ lies in the carrier of $\mf{w}_2=\mf{w}_1$, every element of $\mscr{W}(v_2|w_1)$ is transverse to $\mf{w}_1$, whereas no element of $\mscr{W}(w_1|x_0)$ is. It follows that $\mscr{W}(v_2|w_1)$ and $\mscr{W}(w_1|x_0)$ are disjoint, hence $m(v_1,v_2,x_0)=w_1$. Since $d(v_1,w_1)=1$, we conclude that $d(x_0,v_2)=d(x,v_1)+d(v_1,v_2)-2$.

Assume now instead that $\mf{h}_1\cap A=\mf{h}_2^*\cap A$, but $\mf{h}_1\neq\mf{h}_2^*$. By Lemma~\ref{H vs h 1}, we have either $\mf{h}_2^*\subsetneq\mf{h}_1$ or $\mf{h}_1\subsetneq\mf{h}_2^*$. In the latter case, since $v_2$ is adjacent to $\mf{w}_2\cu\mf{h}_1^*$, we would have $v_2\in\mf{h}_1^*$; so $v_2\in\mf{h}_2^*\setminus\mf{h}_1$, a contradiction. We conclude that $\mf{h}_2^*\subsetneq\mf{h}_1$. 

Note that $v_2\in\mf{h}_2^*\cap A=\mf{h}_1\cap A$. Thus, we have $m=m(v_1,v_2,x)\in\mf{h}_1\cap\mf{h}_2^*$ for every $x\in\mf{h}_1\cap A$. Since $v_1\in\mf{h}_1^*\cap A=\mf{h}_2\cap A$, it follows that $\mf{w}_1,\mf{w}_2$ lie in $\mscr{W}(v_1|m)$ and $d(v_1,m)\geq 2$. Hence $d(x,v_2)\leq d(x,v_1)+d(v_1,v_2)-4$.
\end{proof}

Recalling that $\phi$ preserves distances and $\phi(\mf{h}_i\cap A)=\phi_*(\mf{h}_i)\cap B$, this yields:

\begin{cor}\label{phi* preserves equality}
We have $\mf{h}_1=\mf{h}_2$ if and only if $\phi_*(\mf{h}_1)=\phi_*(\mf{h}_2)$. Furthermore, $\mf{h}_1=\mf{h}_2^*$ if and only if $\phi_*(\mf{h}_1)=\phi_*(\mf{h}_2)^*$.
\end{cor}

\begin{cor}\label{wt phi wd}
The point $\wt\phi(w_i)$ is independent of the choice of $e_i$ and $v_i$. If $w_i\in A$, we have $\wt\phi(w_i)=\phi(w_i)$. 
\end{cor}
\begin{proof}
It is clear that $\wt\phi(w_i)=\phi(w_i)$ when $w_i\in A$. For the other part of the lemma, observe that:
\[v_1\neq v_2 \text{ and } w_1=w_2 \Longleftrightarrow d(v_1,v_2)=2,\ v_2\in\mf{h}_1,\ v_1\in\mf{h}_2 \text{ and } \mf{h}_1\neq\mf{h}_2^*. \]
Corollary~\ref{phi* preserves equality} then shows that we have $w_1=w_2\Leftrightarrow\wt\phi(w_1)=\wt\phi(w_2)$.
\end{proof}

Let $\wt{A}\cu X$ denote the set of vertices that either lie in $A$ or are connected to a vertex of $A$ by an edge of $X$; the set $\wt{B}\cu Y$ is defined similarly. Corollary~\ref{wt phi wd} yields a well-defined map $\wt\phi\colon\wt A\ra\wt B$ extending $\phi\colon A\ra B$. 

We can now apply the same construction to the inverse $\psi=\phi^{-1}\colon B\ra A$ and obtain an extension $\wt\psi\colon\wt B\ra\wt A$. It is clear that the two compositions $\wt\psi\o\wt\phi$ and $\wt\phi\o\wt\psi$ are the identity. Note moreover that the sets $\wt{A}$ and $\wt{B}$ are still preserved by $G_1$ and $G_2$, respectively.

We are left to show that the map $\wt\phi$ is distance-preserving. Observe that:
\begin{align*}
&d(w_1,v_2)=
\begin{cases}
d(v_1,v_2)+1 & \text{if } v_2\in\mf{h}_1^*, \\
d(v_1,v_2)-1 & \text{if } v_2\in\mf{h}_1;
\end{cases}
\\
&d(w_1,w_2)=
\begin{cases}
d(v_1,v_2)+2 & \text{if } v_1\in\mf{h}_2^* \text{ and } v_2\in\mf{h}_1^* \text{ and } \mf{h}_1\neq\mf{h}_2, \\
d(v_1,v_2)-2 & \text{if }  v_1\in\mf{h}_2 \text{ and } v_2\in\mf{h}_1 \text{ and } \mf{h}_1\neq\mf{h}_2^*, \\
d(v_1,v_2) & \text{otherwise, i.e.\ if } 
\begin{cases} \mf{h}_1\in\{\mf{h}_2,\mf{h}_2^*\} \text{, or} \\ v_1\in\mf{h}_2 \text{ and } v_2\in\mf{h}_1^* \text{, or}  \\ v_1\in\mf{h}_2^* \text{ and } v_2\in\mf{h}_1. \end{cases}
\end{cases}
\end{align*}
Since the vertices $v_1$ and $v_2$ lie in $A$, it is clear that, for $i,j\in\{1,2\}$, we have $v_i\in\mf{h}_j\Leftrightarrow\phi(v_i)\in\phi_*(\mf{h}_j)$. 

\begin{cor}\label{extension to 1-nbd}
The bijection $\wt\phi\colon\wt A\ra\wt B$ is distance-preserving.
\end{cor}
\begin{proof}
Given $x_1,x_2\in\wt A$, we need to show that $d(\wt\phi(x_1),\wt\phi(x_2))=d(x_1,x_2)$. If $x_1$ and $x_2$ both lie in $A$, this is clear since $\phi$ is distance-preserving. Let us thus assume that $x_1=w_1\in\wt A\setminus A$ and that $w_1$ is connected by an edge to a vertex $v_1\in A$; we define $\mf{w}_1$ and $\mf{h}_1$ as above.

If $x_2$ coincides with a vertex $v_2\in A$, it follows from the first equality above that $d(\wt\phi(w_1),\wt\phi(v_2))=d(w_1,v_2)$. Indeed, $\wt\phi(w_1)$ and $\phi(v_1)$ are connected by an edge of $Y$ dual to the hyperplane $\phi_*(\mf{w}_1)$ and we have $v_2\in\mf{h}_1$ if and only if $\phi(v_2)\in\phi_*(\mf{h}_1)$.

Otherwise, $x_2$ is a vertex $w_2\in\wt A\setminus A$ and it is connected by an edge to a vertex $v_2\in A$; we define $\mf{w}_2$ and $\mf{h}_2$ as above. Again, $\wt\phi(w_2)$ and $\phi(v_2)$ are connected by an edge of $Y$ dual to $\phi_*(\mf{w}_2)$ and we have $v_1\in\mf{h}_2$ if and only if $\phi(v_1)\in\phi_*(\mf{h}_2)$. Moreover, Corollary~\ref{phi* preserves equality} shows that $\mf{h}_1\in\{\mf{h}_2,\mf{h}_2^*\}$ if and only if $\phi_*(\mf{h}_1)\in\{\phi_*(\mf{h}_2),\phi_*(\mf{h}_2)^*\}$. The second equality above thus yields $d(\wt\phi(w_1),\wt\phi(w_2))=d(w_1,w_2)$, concluding the proof.
\end{proof}

We can now repeatedly apply Corollary~\ref{extension to 1-nbd} to extend the domain of $\phi$. Since $A$ is invariant under the cocompact action $G_1\acts X$, this process results in an extension to the entire $X$ in finitely many steps. 

This completes the proof of Theorem~\ref{ext partial iso thm}, as uniqueness is clear from the construction. Note that we have only used the no-free-faces assumption in the form of Lemma~\ref{semi-straight+}, which is a priori much weaker.

\medskip
When dealing with $\CAT$ cuboid complexes $\mbb{X}$ and $\mbb{Y}$, the above discussion needs to be slightly adapted. See Section~\ref{cuboid section} for some background.

We now stop requiring all points to be vertices and allow the sets $A\cu\mbb{X}$ and $B\cu\mbb{Y}$ to contain arbitrary points. We stress that we will construct an isometry $\Phi\colon\mbb{X}\ra\mbb{Y}$, but this will not, in general, map vertices to vertices. Note that this may be the case even when the map $\phi$ does take \emph{some} vertices to vertices, e.g.\ when the initial map $\phi$ originates from Theorem~\ref{from OR thm}.

Consider the collection of distance-preserving bijections ${\wh\phi\colon\wh A\ra\wh B}$ that extend $\phi\colon A\ra B$ and for which the sets $\wh A$ and $\wh B$ are invariant, respectively, under the actions of $G_1$ and $G_2$. By Zorn's lemma, there exists a maximal element of this collection and we will assume that it is $\phi$ itself. Note that $A$ and $B$ must then be closed, or we would be able to extend $\phi$ to their closures. 

Our goal is then to show that $A=\mbb{X}$ and $B=\mbb{Y}$. In the interest of brevity, we will deliberately omit some of the details.

In the following discussion, all geodesics $\g\cu \mbb{X}$ will be implicitly assumed to be based at some point $\g(0)\in A$ and to consist of a finite union of closed edge parallels. The same holds for geodesics in $\mbb{Y}$.

We denote the length of $\g$ by $\ell(\g)$. For $0\leq t<\ell(\g)$, let $\g_t\cu\g$ be the closed sub-segment of length $t$ containing $\g(0)$. We denote by $\wh\g(t)\cu \mbb{X}$ the (unique) median halfspace that intersects $\g$ precisely at $\g\setminus\g_t$. 

By analogy with Definition~\ref{SL defn}, we say that $\g$ is \emph{$t$--SL} if $\wh\g(s_2)\cu\wh\g(s_1)$ for every $0\leq s_1<\min\{s_2,t\}$. We say that $\g$ is \emph{straight} if it is $\ell(\g)$--SL. If $\g$ is $t$--SL, its initial segment $\g_t$ is straight. Note that every straight geodesic $\g\cu \mbb{X}$ is convex and thus endowed with a gate-projection $\pi_{\g}\colon \mbb{X}\ra\g$.

Given a point $p\in A$ and a real number $t\geq 0$, we define $\mc{H}_p(t)$ as the collection of subsets $H\cu A$ with the property that, for every $x\in A$:
\begin{equation}\tag{$\ast\ast$} x\in H~\Leftrightarrow~\forall y\in H,~d(x,y)<d(x,p)+d(p,y)-2t. \end{equation}
Note that $d(x,y)<d(x,p)+d(p,y)-2t$ is equivalent to $d(p,m(p,x,y))>t$. 

As a straightforward analogue of Lemma~\ref{H_w} and Remark~\ref{if a singleton}, we have:

\begin{lem}\label{cuboid 1}
Let $\g$ be a $t$--SL geodesic joining points $p,x\in A$. For every $0\leq s<t$, the set $\wh\g(s)\cap A$ is the only element of $\mc{H}_p(s)$ that contains $x$.
\end{lem}

The following will replace Remark~\ref{useful fact}.

\begin{lem}\label{cuboid 1.5}
Let $\g$ be a straight geodesic with $\g(0)\in A$. If $\g$ consists of at most two edge parallels, then $\g$ can be extended to an $\ell(\g)$--SL geodesic $\g'$ terminating at a point of $A$.
\end{lem}
\begin{proof}
There exist a cube $c\cu X$, a vertex $v\in c$ and two edges $e,f$ at $v$ such that $f\cup e$ is straight and such that $c\x(f\cup e)\cu X$ contains the interior of $\g$ in its interior. Let us assume that $\g(0)$ lies in the cube $c\x f$. We now apply Lemma~\ref{semi-straight+} to the vertex $v$ and the edge $e$, ensuring that the maximal cube picked at the beginning of the proof of the lemma contains $c\x e$. We obtain $g\in G_1$ such that $v$ and $gv$ are endpoints of an SL geodesic extending $e$. By Remark~\ref{SL and transversality} and condition~(2) in Lemma~\ref{semi-straight+}, the hyperplane $g\cdot\mf{w}(f)$ is not transverse to $\mf{w}(e)$. It follows that $v$ and $g\cdot\g(0)\in A$ are endpoints of an SL geodesic and, similarly, so are $\g(0)$ and $g\cdot\g(0)$.
\end{proof}

Let now $\eta(\mbb{X})$ denote the shortest length of an edge of $\mbb{X}$. Note that every straight geodesic $\g\cu X$ with $\ell(\g)\leq\eta(\mbb{X})$ consists of at most two edge parallels.

\begin{lem}\label{cuboid 2}
Consider $0<t\leq\eta(\mbb{X})$. Distinct points $p,x\in A$ are connected by a $t$--SL geodesic if and only if, for every $0\leq s<t$, the point $x$ belongs to exactly one element of $\mc{H}_p(s)$.
\end{lem}
\begin{proof}
One implication is immediate from Lemma~\ref{cuboid 1}. For the other, consider a geodesic $\g$ with endpoints $p$ and $x$ and suppose it is not $t$--SL. Let $s<t$ be the largest value for which $\g$ is $s$--SL (possibly $s=0$) and let $p_s\in\g$ the point at distance $s$ from $p$. Then $p_s$ lies in the frontier of two transverse median halfspaces $\mf{h}_1,\mf{h}_2$ containing $x$.

Since the shortest geodesic from $p$ to $\overline{\mf{h}_i}$ is straight, Lemma~\ref{cuboid 1.5} yields a point $y_i\in\mf{h}_i\cap A$ joined to $p$ by an $r_i$--SL geodesic with $r_i>s$. Lemma~\ref{cuboid 1} then guarantees that $\mf{h}_1\cap A$ and $\mf{h}_2\cap A$ lie in $\mc{H}_p(s)$, completing the proof.
\end{proof}

Setting $\eta=\min\{\eta(\mbb{X}),\eta(\mbb{Y})\}$, we have the following.

\begin{cor}\label{cuboid 3}
Let $p\in A$ and $q\in \mbb{X}$ be endpoints of a straight geodesic $\g$ of length $0<\ell\leq\eta$. There exists a straight geodesic $\phi_*\g\cu \mbb{Y}$ with an endpoint at $\phi(p)$ and the same length $\ell$. For every $0\leq t<\ell$, we have $\phi(\wh\g(t)\cap A)=\wh{\phi_*\g}(t)\cap B$.
\end{cor}
\begin{proof}
Lemma~\ref{cuboid 1.5} allows us to extend $\g$ to an $\ell$--SL geodesic ending at a point $x\in A$. Observe that the map $\phi\colon A\ra B$ naturally induces a bijection $\mc{H}_p(t)\ra\mc{H}_{\phi(p)}(t)$ for all $t\geq 0$. Lemma~\ref{cuboid 2} then implies that $\phi(p)$ and $\phi(x)$ are joined by an $\ell$--SL geodesic $\g'\cu \mbb{Y}$. We set $\phi_*\g=\g'_{\ell}$. For $0\leq t<\ell$, the equality $\phi(\wh\g(t)\cap A)=\wh{\phi_*\g}(t)\cap B$ is immediate from Lemma~\ref{cuboid 1}.
\end{proof}

We now define $A_{\eta}$ as the union of straight geodesics of length $\eta$ that intersect $A$. The set $B_{\eta}$ is defined similarly and Corollary~\ref{cuboid 3} allows us to extend $\phi\colon A\ra B$ to a bijection $\phi_{\eta}\colon A_{\eta}\ra B_{\eta}$ that maps each straight geodesic ${\g\cu A_{\eta}}$ isometrically onto $\phi_*\g$. Note moreover that the sets $A_{\eta}$ and $B_{\eta}$ are still preserved by $G_1$ and $G_2$, respectively.
 
We will now show that $\phi_{\eta/2}$ preserves distances. For this purpose, consider straight geodesics $\g_1,\g_2\cu \mbb{X}$ of length $\eta$ based at $p_1,p_2\in A$. Lemma~\ref{cuboid 1.5} allows us to extend $\g_i$ to an $\eta$--SL geodesic $\g_i'$ ending at a point $x_i\in A$. We also introduce the notation $\wh\g_2(s_2)^*$ for the only median halfspace satisfying $\g_2\setminus(\wh\g_2(s_2)\cup\wh\g_2(s_2)^*)=\{\g_2(s_2)\}$. Here $\g_2(s_2)$ denotes the point of $\g_2$ at distance $s_2$ from $p_2$. We will need the following analogue of Proposition~\ref{h1=h2}.

\begin{prop}\label{cuboid 4}
Let us consider the Gromov products $r_1=(x_1\cdot p_2)_{p_1}$, ${r_2=(x_2\cdot p_1)_{p_2}}$, $r_1'=(x_1\cdot x_2)_{p_1}$ and  ${r_2'=(x_1\cdot x_2)_{p_2}}$. We also pick real numbers $s_i\geq 0$ with $s_1+s_2\leq\eta$.
\begin{enumerate}
\item If we have $\wh\g_1(s_1)\cap A=\wh\g_2(s_2)\cap A$, then $\wh\g_1(r_1)=\wh\g_2(r_2)$. In particular, $\wh\g_1(s_1)=\wh\g_2(r_2+s_1-r_1)$ and $\wh\g_2(s_2)=\wh\g_1(r_1+s_2-r_2)$.
\item If $\wh\g_1(s_1)\cap A=\wh\g_2(s_2)^*\cap A$, exactly one of the following happens:
	\begin{enumerate}
	\item $\wh\g_2(s_2)^*\supsetneq\wh\g_1(s_1)$ and $\max\{r_1,r_2\}<s_1+s_2$;
	\item $\wh\g_2(s_2)^*=\wh\g_1(s_1)$ and $\max\{r_1,r_2\}=s_1+s_2$, $\min\{r_1',r_2'\}=0$; 
	\item $\wh\g_2(s_2)^*\subsetneq\wh\g_1(s_1)$ and $\max\{r_1,r_2\}\geq s_1+s_2$. If the latter is an equality, then $\min\{r_1',r_2'\}>0$.
	\end{enumerate}
\end{enumerate}
\end{prop}
\begin{proof}
If $\wh\g_1(s_1)\cap A=\wh\g_2(s_2)\cap A$, we can assume that $\wh\g_1(s_1)\cu\wh\g_2(s_2)$. Moreover, $x_1,x_2\in\wh\g_1(s_1)$ and $p_1,p_2\not\in\wh\g_2(s_2)$. The largest median halfspace that contains $\wh\g_2(s_2)$ while being disjoint from $\{p_1,p_2\}$ must then coincide with both $\wh\g_1(r_1)$ and $\wh\g_2(r_2)$. The rest of part~(1) follows.

The key observation for part~(2) is that we cannot simultaneously have $\min\{r_1',r_2'\}>0$ and $\max\{r_1,r_2\}<\eta$. This follows from the fact that $\g_1$ and $\g_2$ have length $\eta$, are straight, and each consist of at most two edge parallels.

Now, assuming that $\wh\g_1(s_1)\cap A=\wh\g_2(s_2)^*\cap A$, we must have either $\wh\g_2(s_2)^*\supsetneq\wh\g_1(s_1)$, ${\wh\g_2(s_2)^*=\wh\g_1(s_1)}$ or $\wh\g_2(s_2)^*\subsetneq\wh\g_1(s_1)$. In the first case, it is clear that $\max\{r_1,r_2\}<s_1+s_2$, since $\pi_{\g_i}$ are $1$--Lipschitz.

If $\wh\g_2(s_2)^*=\wh\g_1(s_1)$, one similarly sees that $\max\{r_1,r_2\}\leq s_1+s_2\leq\eta$. If $\max\{r_1,r_2\}=\eta$, we have $\max\{r_1,r_2\}=s_1+s_2$ and $\min\{r_1',r_2'\}=0$. Otherwise $\max\{r_1,r_2\}<\eta$ and the key observation yields $\min\{r_1',r_2'\}=0$. Without loss of generality $r_1'=0$, in which case $r_2=s_1+s_2$, settling case~(b).

Finally, suppose that $\wh\g_2(s_2)^*\subsetneq\wh\g_1(s_1)$. If $\min\{r_1',r_2'\}>0$, we have already remarked that $\max\{r_1,r_2\}=\eta\geq s_1+s_2$. If instead $\min\{r_1',r_2'\}=0$, say $r_1'=0$, we have $r_2>s_1+s_2$. This concludes the proof.
\end{proof}

Since $\phi$ preserves Gromov products of points of $A$, we obtain:

\begin{cor}\label{cuboid 5}
Given $0\leq s_i\leq\eta/2$, we have $\wh\g_1(s_1)=\wh\g_2(s_2)$ if and only if $\wh{\phi_*\g_1}(s_1)=\wh{\phi_*\g_2}(s_2)$. Furthermore, $\wh\g_1(s_1)=\wh\g_2(s_2)^*$ if and only if $\wh{\phi_*\g_1}(s_1)=\wh{\phi_*\g_2}(s_2)^*$.
\end{cor}

We now pick points $q_i\in\g_i$ and set $d_i=d(p_i,q_i)\geq 0$. Consider the sets:
\begin{align*}
&\{t\in [0,d_1]\mid q_2\in\wh\g_1(t),\ p_2\not\in\wh\g_1(t)\}, &\{t\in [0,d_1]\mid p_2,q_2\in\wh\g_1(t)\}, \\
&\{t\in [0,d_1]\mid p_2\in\wh\g_1(t),\ q_2\not\in\wh\g_1(t)\}, &\{t\in [0,d_2]\mid p_1,q_1\in\wh\g_2(t)\}.
\end{align*}
These are sub-intervals of $[0,d_1]$ and $[0,d_2]$, respectively; let $L(q_1,q_2)$ denote the sum of their lengths. These intervals correspond to the `median hyperplanes' that are crossed two or three times when moving from $q_1$ to $q_2$ through $p_1$ and $p_2$. It follows that:
\[d(q_1,q_2)=d(q_1,p_1)+d(p_1,p_2)+d(p_2,q_2)-2\cdot L(q_1,q_2).\]
Since $\phi$ is isometric on $A$ and along $\g_i$, in order to obtain the equality ${d(\phi(q_1),\phi(q_2))=d(q_1,q_2)}$, it is enough to prove $L(\phi(q_1),\phi(q_2))=L(q_1,q_2)$. When $d_1,d_2\leq\eta/2$, the latter equality follows from Corollary~\ref{cuboid 5}. 

For instance, observe that we have $q_2\in\wh\g_1(t)$ and $p_2\not\in\wh\g_1(t)$ if and only if there exists $0\leq s<d_2$ with $\wh\g_1(t)=\wh\g_2(s)$ and, by Corollary~\ref{cuboid 5}, this happens if and only if $\phi_{\eta}(q_2)\in\wh{\phi_*\g_1}(t)$ and $\phi(p_2)\not\in\wh{\phi_*\g_1}(t)$.

We conclude that the bijection $\phi_{\eta/2}\colon A_{\eta/2}\ra B_{\eta/2}$ is an isometry. Since $A_{\eta/2}$ and $B_{\eta/2}$ are preserved, respectively, by the actions of $G_1$ and $G_2$, maximality of $A$ and $B$ yields $A=A_{\eta/2}$ and $B=B_{\eta/2}$. 

It follows that, for every edge parallel $\g\cu\mbb{X}$, the intersection $A\cap\g$ is clopen in $\g$. Thus, $A$ and $B$ contain every edge parallel that they intersect and this implies that they contain every cube that they intersect. 

We have finally shown that $A=\mbb{X}$ and $B=\mbb{Y}$, completing the proof of Theorem~\ref{ext partial iso thm} in the cuboidal case.

\section{Compactifying the space of cubical actions.} \label{cptf sect}

In this section, we prove Proposition~\ref{cptf main} and Corollary~\ref{CSV application}. The main argument is based on ultralimits and median spaces and occupies Section~\ref{ultralimits sect}. Throughout, it is convenient to place ourselves in the context of $\CAT$ \emph{cuboid} complexes; see Section~\ref{cuboid section} for terminology and notation.

\subsection{Distances to minimal sets.}

Let $\mbb{X}$ be a $\CAT$ cuboid complex of dimension $D<+\infty$. We will denote the metric on $\mbb{X}$ simply by $d$.

Recall that, given $x,y\in \mbb{X}$, we write $\mf{H}(x|y)$ for the set of median halfspaces containing $y$ but not $x$. Given a real number $\delta\in\R_{\geq 0}$, we also denote by $\mf{H}_{\delta}(x|y)$ the subset of those $\mf{h}\in\mf{H}(x|y)$ that satisfy $d(x,\mf{h})=\delta$. Observe that distinct elements of $\mf{H}_{\delta}(x|y)$ must be subordinate to transverse hyperplanes, hence $\#\mf{H}_{\delta}(x|y)\in\{0,1,...,D\}$ for all $x,y\in \mbb{X}$ and $\delta\geq 0$.

Writing $\mc{L}$ for the Lebesgue measure on $\R$, all $x,y\in \mbb{X}$ satisfy:
\[d(x,y)=\int_{\R}\#\mf{H}_{\delta}(x|y)\cdot d\mc{L}(\delta).\]

\begin{lem}\label{cptf 1}
If an automorphism $g\in\Aut(\mbb{X})$ fixes a unique point $p\in\mbb{X}$, we have $d(x,gx)\geq\frac{2}{D}\cdot d(x,p)$ for all $x\in\mbb{X}$.
\end{lem}
\begin{proof}
Given a point $x\in\mbb{X}$ and $\delta>0$, suppose for the sake of contradiction that $g\cdot\mf{H}_{\delta}(p|x)=\mf{H}_{\delta}(p|x)\neq\emptyset$. The elements of $\mf{H}_{\delta}(p|x)$ are subordinate to pairwise-transverse hyperplanes, which determine a cuboid $c\cu\mbb{X}$ preserved by $g$. It follows that $g$ fixes the (unique) point of $c$ that is farthest from $p$, contradicting uniqueness of $p$. We conclude that, for all $\delta>0$, the set $g\cdot\mf{H}_{\delta}(p|x)=\mf{H}_{\delta}(p|gx)$ is either empty or different from $\mf{H}_{\delta}(p|x)$.

Setting $k_{\delta}=\#\mf{H}_{\delta}(p|x)\leq D$ and $m=m(p,x,gx)$, we then have:
\[\#\mf{H}_{\delta}(p|m)=\#\left(\mf{H}_{\delta}(p|x)\cap\mf{H}_{\delta}(p|gx)\right)\leq\tfrac{k_{\delta}-1}{k_{\delta}}\cdot\#\mf{H}_{\delta}(p|x)\leq\tfrac{D-1}{D}\cdot\#\mf{H}_{\delta}(p|x).\]
The lemma now follows from two inequalities:
\[d(p,m)=\int_{\R}\#\mf{H}_{\delta}(p|m)\cdot d\mc{L}(\delta)\leq\tfrac{D-1}{D}\cdot\int_{\R}\#\mf{H}_{\delta}(p|x)\cdot d\mc{L}(\delta)=\tfrac{D-1}{D}\cdot d(p,x),\]
\[d(x,gx)=2\cdot d(x,m)=2\cdot\left(d(p,x)-d(p,m)\right)\geq 2\cdot\left(1-\tfrac{D-1}{D}\right)\cdot d(p,x). \qedhere\]
\end{proof}

\begin{prop}\label{cptf 2}
Let $g\in\Aut(\mbb{X})$ act stably without inversions and non-transversely. For every $x\in \mbb{X}$, we have $d(x,gx)\geq\frac{2}{D}\cdot d(x,{\rm Min}_{\mbb{X}}(g))+\ell_{\mbb{X}}(g)$.
\end{prop}
\begin{proof}
By Proposition~\ref{Haglund FFT}, the minimal set ${\rm Min}_{\mbb{X}}(g)$ is a convex subcomplex of $\mbb{X}$. Passing to the restriction quotient (in the sense of \cite{CS}) that kills all hyperplanes of ${\rm Min}_{\mbb{X}}(g)$, the proposition follows from Lemma~\ref{cptf 1}.
\end{proof}

\subsection{Normalising length functions.}\label{ultralimits sect}

Let us now fix a group $\G$ with a finite generating set $S\cu \G$. We denote by $\|\cdot\|$ the word norm on $\G$ induced by $S$. 

Whenever $\G$ acts by isometries on a metric space $(M,d)$, we write:
\[\tau_M(x)=\max_{s\in S} d(x,sx),  \hspace{.5cm} \tau_M=\inf_{x\in M}\tau_M(x), \hspace{.5cm} \ell_M(g)=\inf_{x\in M} d(x,gx),\]
for all $x\in M$ and $g\in\G$.

\begin{lem}\label{cptf 3}
For every $g\in \G$ and $x\in M$, we have $d(x,gx)\leq\tau_M(x)\cdot\|g\|$. In particular, $\ell_M(g)\leq\tau_M\cdot\|g\|$.
\end{lem}
\begin{proof}
Writing $g=s_1s_2...s_k$ with $s_i\in S$ and $k=\|g\|$, we have:
\[d(x,gx)\leq\sum_{i=0}^{k-1}d(s_1...s_i\cdot x,s_1...s_{i+1}\cdot x)\leq\sum_{i=1}^kd(x,s_ix)\leq k\cdot\tau_{M}(x). \qedhere\]
\end{proof}

If $\G$ acts on a $\CAT$ cuboid complex $\mbb{X}$, the quantity $\tau_{\mbb{X}}$ is always attained at a point of $\mbb{X}$. In fact, it is not hard to see that $\tau_{\mbb{X}}=\tau_{\mbb{X}}(x)$ for a vertex $x$ of the barycentric subdivision $\mbb{X}'$.

Consider now a sequence of actions on $\CAT$ cuboid complexes $\G\acts\mbb{X}_n$. For the sake of simplicity, we will write $\ell_n$ and $\tau_n$ with the meaning of $\ell_{\mbb{X}_n}$ and $\tau_{\mbb{X}_n}$. Let us pick basepoints $o_n\in\mbb{X}_n$ with $\tau_n(o_n)=\tau_n$.

Fixing a non-principal ultrafilter $\om\cu 2^{\N}$, let $\mbb{X}_{\om}$ denote the ultralimit $\lim_{\om}(\mbb{X}_n,o_n)$. This is a complete, geodesic, median space. If moreover $\dim\mbb{X}_n\leq D<+\infty$ for all $n\geq 0$, the rank of $\mbb{X}_{\om}$ is at most $D$ (see e.g.\ Theorem~2.3 in \cite{Bow1}). If $\lim_{\om}\tau_n<+\infty$, we also obtain an isometric action $\G\acts\mbb{X}_{\om}$. Indeed, Lemma~\ref{cptf 3} shows that $d(o_n,go_n)\leq\tau_n(o_n)\cdot\|g\|=\tau_n\cdot\|g\|$ for every $g\in \G$. Let us write $\tau_{\om}$ and $\ell_{\om}$ rather than $\tau_{\mbb{X}_{\om}}$ and $\ell_{\mbb{X}_{\om}}$.

\begin{prop}\label{cptf 3.5}
Suppose that $\lim_{\om}\tau_n<+\infty$ and $\lim_{\om}\dim\mbb{X}_n=D<+\infty$. Then $\tau_{\om}=\lim_{\om}\tau_n$ and $\ell_{\om}(g)=\lim_{\om}\ell_n(g)$ for all $g\in \G$.
\end{prop}
\begin{proof}
Writing $o_{\om}=(o_n)\in\mbb{X}_{\om}$, we have $\tau_{\om}(o_{\om})=\lim_{\om}\tau_n(o_n)=\lim_{\om}\tau_n$. For every $x_{\om}=(x_n)\in\mbb{X}_{\om}$, we have $\tau_{\om}(x_{\om})=\lim_{\om}\tau_n(x_n)\geq\lim_{\om}\tau_n$. Hence $\tau_{\om}=\tau_{\om}(o_{\om})=\lim_{\om}\tau_n$.

Now, barycentrically subdividing the cuboid complexes if necessary, every ${g\in \G}$ acts on every $\mbb{X}_n$ stably without inversions; furthermore, $g^{D!}$ acts non-trans\-versely by Proposition~\ref{Haglund FFT}. Let $p_n\in\mbb{X}_n$ be the point of ${\rm Min}_{\mbb{X}_n}(g^{D!})$ that is closest to $o_n$. By Proposition~\ref{cptf 2} and Lemma~\ref{cptf 3}, we have: 
\[d(o_n,p_n)\leq\tfrac{D}{2}\cdot d(o_n,g^{D!}o_n)\leq\tfrac{D}{2}\cdot\|g^{D!}\|\cdot\tau_n(o_n)=\tfrac{D}{2}\cdot\|g^{D!}\|\cdot\tau_n.\]
In particular, $\lim_{\om}d(o_n,p_n)<+\infty$ and $p_{\om}=(p_n)$ is a point of $\mbb{X}_{\om}$. Now
\[d(p_{\om},g^{D!}p_{\om})=\lim_{\om}d(p_n,g^{D!}p_n)=\lim_{\om}\ell_n(g^{D!}),\]
and, similarly, $d(x_{\om},g^{D!}x_{\om})\geq d(p_{\om},g^{D!}p_{\om})$ for every $x_{\om}\in\mbb{X}_{\om}$. We conclude that $\ell_{\om}(g^{D!})=d(p_{\om},g^{D!}p_{\om})=\lim_{\om}\ell_n(g^{D!})=D!\cdot\lim_{\om}\ell_n(g)$. 

Finally, it follows from Proposition~4.9(1) in \cite{Fio10b} that $\ell_{\om}(g^{D!})=D!\cdot\ell_{\om}(g)$. Thus, we deduce that $\ell_{\om}(g)=\lim_{\om}\ell_n(g)$, concluding the proof.
\end{proof}

Note that, given a $\CAT$ cuboid complex $\mbb{X}$, an action $\G\acts\mbb{X}$ has no global fixed point (equivalently: orbits are unbounded) if and only if $\tau_{\mbb{X}}\neq 0$. In this case, we define the \emph{reduced length function} as $\overline\ell_{\mbb{X}}(g)=\ell_{\mbb{X}}(g)/\tau_{\mbb{X}}$. 

\begin{thm}\label{cptf 4}
If $\dim \mbb{X}_n\leq D$ and the actions $\G\acts \mbb{X}_n$ have unbounded orbits, there exists an element $g\in \G$ such that $\overline\ell_{\mbb{X}_n}(g)\not\ra 0$.
\end{thm}
\begin{proof}
Suppose for the sake of contradiction that $\overline\ell_{\mbb{X}_n}(g)\ra 0$ for all $g\in\G$. Applying Proposition~\ref{cptf 3.5} to the sequence $\G\acts\frac{1}{\tau_n}\cdot\mbb{X}_n$, we obtain an action on a complete, connected, finite rank, median space $\mbb{X}_{\om}$ with $\tau_{\om}=1$ and $\ell_{\om}(g)=0$ for all $g\in \G$. By Corollary~A in \cite{Fio10b}, each $g\in\G$ has a fixed point in $\mbb{X}_{\om}$. Theorem~3.1 in \cite{Fio2} then yields a global fixed point for $\G\acts X_{\om}$. This contradicts $\tau_{\om}=1$.
\end{proof}

As in the introduction, we can consider the composition:
\[\mathrm{Cub}_D(\G)\overset{\overline\ell}{\longrightarrow}\R^{\G}\setminus\{0\}\longrightarrow\mbb{P}(\R^{\G}).\]
Note that we are now mapping $\mathrm{Cub}_D(\G)$ into $\R^{\G}\setminus\{0\}$ via \emph{reduced} length functions, as this does not affect the projection to the projectivisation. Lemma~\ref{cptf 3} shows that reduced length functions all lie in the compact subspace $\prod_{g\in \G}[0,\|g\|]\cu\R^{\G}$ and, by Theorem~\ref{cptf 4}, they do not accumulate on the zero function. The closure of the image of $\overline\ell\colon\mathrm{Cub}_D(\G)\ra\R^{\G}\setminus\{0\}$ is then compact and projects to a compact subset of $\mbb{P}(\R^{\G})$. 

This proves Proposition~\ref{cptf main}.

\medskip
We conclude this section by sketching a proof of Corollary~\ref{CSV application}. We assume that the reader is familiar with the notation and content of \cite{Charney-Stambaugh-Vogtmann}.

\begin{proof}[Proof of Corollary~\ref{CSV application}]
Let $\G$ be the right-angled Artin group defined by a graph $\Lambda$ and let $\mbb{S}$ denote its Salvetti complex. Every point of the simplicial complex $K_{\Lambda}$ (and more generally $\Sigma_{\Lambda}$) defined in \cite{Charney-Stambaugh-Vogtmann} is a pair $\s=(M,\alpha)$, where $M$ is a compact non-positively curved cuboid complex (obtained as a ``Salvetti blow-up'') and $\alpha\colon M\ra\mbb{S}$ is a homotopy equivalence. 

Equivalently, we can view $\s$ as the $\G$--action on the universal cover $X=\wt M$ determined by $\alpha$. Two such actions are identified as points of $\Sigma_{\Lambda}$ exactly when they are $\G$--equivariantly homothetic. Associating to $\s$ the projectivised combinatorial length function of this action, we obtain a map $\Sigma_{\Lambda}\ra\mbb{P}(\R^{\G})$. It will follow from Theorem~\ref{NFF MLSR} and Proposition~\ref{cptf main} that this is an injection with relatively compact image, once we show that their hypotheses are satisfied in this context. Continuity of the map is straightforward.

Assume for a moment that, for every $\s\in\Sigma_{\Lambda}$, the complex $M$ (hence $X=\wt M$) has no free faces; we will sketch a proof of this below. Then $X$ is an essential $\CAT$ cube complex and, since $\G\acts X$ is proper and cocompact, this action is also $\ell^1$--minimal and non-elementary (Lemmas~\ref{geometric implies nonelementary} and~\ref{minimal vs essential}). Note moreover that $\dim X$ is uniformly bounded above in terms of the graph $\Lambda$. Finally, since the right-angled Artin group $\G$ does not split as a direct product, the universal cover $\wt{\mbb{S}}$ of $\mbb{S}$ has nonempty contracting boundary \cite{Charney-Sultan}. Since $X$ and $\wt{\mbb{S}}$ are quasi-isometric, $X$ also has nonempty contracting boundary and, thus, it has at most one unbounded factor. As $X$ is essential, it has no bounded factors, hence $X$ must be irreducible.

We conclude with a sketch of the fact that $M$ and $X$ have no free faces. Suppose that $M=\mbb{S}^{\mathbf{\Pi}}$ for a set $\mathbf{\Pi}=\{\mathbf{P}_1,...,\mathbf{P}_k\}$ of pairwise-compatible $\Lambda$-Whitehead partitions. Suppose for the sake of contradiction that $M$ does have free faces. Then there exists a maximal cube $c\cu M$ and a codimension-one face $c'\cu c$ that is contained in no other maximal cube. We can assume that $\mathbf{\Pi}$ and $c$ are chosen so as to minimise the quantity $(k+\dim c)$.

Let $\mf{w}_i$ the hyperplane of $M$ dual to all edges labelled $e_{\mathbf{P}_i}$ (cf.\ Theorem~3.14(2) in \cite{Charney-Stambaugh-Vogtmann}).

{\bf Step~1:} \emph{the cube $c$ must intersect all $\mf{w}_i$; in particular, the partitions $\mathbf{P}_i$ pairwise commute.} Denote by $M_i$ the collapse of $M$ along the hyperplane $\mf{w}_i$. By Theorem~4.6 in \cite{Charney-Stambaugh-Vogtmann}, this is isomorphic to the blow-up of $\mbb{S}$ with respect to $\mathbf{\Pi}\setminus\{\mathbf{P}_i\}$. If some $\mf{w}_i$ did not cut the cube $c$, then $c$ would project injectively to $M_i$ and $c'$ would still be a free face in $M_i$. However, this would violate minimality of $(k+\dim c)$.

Now, let us a pick a vertex $v\in c'$ and let $\eps$ be the edge at $v$ that lies in $c\setminus c'$. Let $R=(P_1,...,P_k)$ be the region associated to the vertex $v$. Let $\eps_1,...,\eps_s$ be the edges of $c'$ at $v$ that are labelled by elements of $V^{\pm}$; more precisely, say that $\eps_i$ is labelled by $e_{x_i}$ with $x_i\in V^{\pm}$.

{\bf Step~2:} \emph{the edge $\eps$ is not labelled $e_{\mathbf{P}_i}$ for any $1\leq i\leq k$.} Otherwise, consider an element $m\in\max P_i$. Observe that $m\in I(R)$ as $\max P_i\cu\llk P_j$ for all $j\neq i$. Let $\eps'$ be the edge at $v$ labelled $e_m$ and observe that it does not span a square with $\eps$. On the other hand, $\llk m=\llk P_i\supseteq\{x_1,...,x_s\}$ and $m\in\max P_i\cu\llk P_j$ for all $j\neq i$. It follows that $\eps'$ and $c'$ span a cube different from $c$, a contradiction.

{\bf Step~3:} \emph{the edge $\eps$ also cannot be labelled by $e_x$ for any $x\in V^{\pm}$ (a contradiction).} In order to see this, suppose that $\eps$ is labelled $e_x$. Note that $\eps$ cannot be a loop or $c'$ would not be a free face. Thus, let $R'\neq R$ be the region determined by the vertex of $\eps$ other than $v$. Theorem~3.14(1) in \cite{Charney-Stambaugh-Vogtmann} shows that there exists an edge $\eps'$ at $v$ that is labelled $e_{\mathbf{P}_i}$, where $R$ and $R'$ lie on opposite sides of the partition $\mathbf{P}_i$. Since $x\in\sing\mathbf{P}_i$, it follows that $\eps'$ and $\eps$ do not span a square and, moreover, $\{x_1,...,x_s\}\cu\llk x\cu\llk\mathbf{P}_i$. Hence, $\eps'$ spans a cube with $c'$, violating the fact that $c'$ is a free face.
\end{proof}

\bibliography{mybib}
\bibliographystyle{alpha}

\end{document}